\documentclass[11pt,oneside]{amsart}

\usepackage{graphicx}
\usepackage{stmaryrd}
\usepackage{tikz-cd}
\usepackage{hyperref}
\usepackage{comment}
\usepackage{mathtools}
\usepackage{enumitem}
\usepackage{amssymb}
\usepackage{todonotes}
\usepackage{dsfont}
\usepackage[margin=1in]{geometry}

\usepackage{fullpage}

\newtheorem{theorem}{Theorem}[section]
\newtheorem{corollary}[theorem]{Corollary}
\newtheorem{proposition}[theorem]{Proposition}
\newtheorem{lemma}[theorem]{Lemma}
\newtheorem{remark}[theorem]{Remark}
\newtheorem{conjecture}[theorem]{Conjecture}

\theoremstyle{definition}
\newtheorem{definition}[theorem]{Definition}

\numberwithin{equation}{section}

\definecolor{airforceblue}{rgb}{0.36, 0.54, 0.66}

\newcommand{\bX}{\mathbb{X}}
\newcommand{\bx}{{\bf x}}

\newcommand{\bV}{\mathbb{V}}
\newcommand{\bY}{\mathbb{Y}}
\newcommand{\bG}{\mathbb{G}}
\newcommand{\bP}{\mathbb{P}}
\newcommand{\by}{\mathbf{y}}
\newcommand{\bA}{\mathbb{A}}

\newcommand{\rF}{\mathrm{F}}
\newcommand{\Gr}{\mathrm{Gr}}

\newcommand{\bL}{\mathbb{L}}

\newcommand{\Q}{\mathbb{Q}}
\newcommand{\Z}{\mathbb{Z}}

\newcommand{\C}{\mathbb{C}}

\newcommand{\F}{\mathbb{F}}

\newcommand{\spa}{\mathrm{Span}}

\newcommand{\Oo}{\mathcal{O}}
\newcommand{\cV}{\mathcal{V}}
\newcommand{\cZ}{\mathcal{Z}}
\newcommand{\cD}{\mathcal{D}}

\newcommand{\cH}{\mathcal{H}}
\newcommand{\cL}{\mathcal{L}}

\newcommand{\tZ}{\tilde{\mathcal{Z}}}
\newcommand{\tN}{\tilde{\mathcal{N}}}

\newcommand{\cG}{\mathcal{G}}
\newcommand{\cN}{\mathcal{N}}

\newcommand{\cM}{\mathcal{M}}

\newcommand{\cF}{\mathcal{F}}
\newcommand{\Exc}{\mathrm{Exc}}

\newcommand{\Kra}{\mathrm{Kra}}
\newcommand{\Int}{\mathrm{Int}}
\newcommand{\Pap}{\mathrm{Pap}}
\newcommand{\GL}{\mathrm{GL}}
\newcommand{\tr}{\mathrm{tr}}

\newcommand{\Lie}{\mathrm{Lie}\, }
\newcommand{\Nilp}{\mathrm{Nilp}\, }
\newcommand{\bfV}{\mathbf{V}}
\newcommand{\bfF}{\mathbf{F}}
\newcommand{\Hor}{\mathrm{Hor}}

\newcommand{\Spec}{\mathrm{Spec}\, }
\newcommand{\Spf}{\mathrm{Spf}\, }
\newcommand{\SpfOF}{{\mathrm{Spf}\,\mathcal{O}_{\breve{F}} }}

\newcommand{\Herm}{\mathrm{Herm}}
\newcommand{\Hom}{\mathrm{Hom}}

\newcommand{\End}{\mathrm{End}}

\newcommand{\rU}{\mathrm{U}}
\newcommand{\val}{\mathrm{v}}

\newcommand{\q}{q}
\newcommand{\pden}{\partial \mathrm{Den}}
\newcommand{\red}{\mathrm{red}}
\newcommand{\spaF}{\mathrm{Span_{\Oo_F}}}
\newcommand{\Den}{\mathrm{Den}}

\newcommand{\LZ}{{}^{\mathbb{L}}\mathcal{Z}^{\mathrm{Kra}}}

\newcommand{\diag}{\mathrm{Diag}}
\newcommand{\OO}{\mathcal O}

\DeclareFontFamily{U}{matha}{\hyphenchar\font45}
\DeclareFontShape{U}{matha}{m}{n}{
	<5> <6> <7> <8> <9> <10> gen * matha
	<10.95> matha10 <12> <14.4> <17.28> <20.74> <24.88> matha12
}{}
\DeclareSymbolFont{matha}{U}{matha}{m}{n}
\DeclareFontFamily{U}{mathx}{\hyphenchar\font45}
\DeclareFontShape{U}{mathx}{m}{n}{
	<5> <6> <7> <8> <9> <10>
	<10.95> <12> <14.4> <17.28> <20.74> <24.88>
	mathx10
}{}
\DeclareSymbolFont{mathx}{U}{mathx}{m}{n}

\DeclareMathSymbol{\obot}         {2}{matha}{"6B}
\DeclareMathSymbol{\bigobot}       {1}{mathx}{"CB}

\author{Qiao He, Yousheng Shi, and  Tonghai Yang}
\address{Department of Mathematics, University of Wisconsin Madison, Van Vleck Hall, Madison, WI 53706, USA}
\email{qhe36@wisc.edu, ORCID: 0000-0001-9273-2687}
\address{School of Mathematical Science, Zhejiang University, 866 Yuhangtang Rd, Hangzhou, 310058, P. R. China}
\email{shi58@wisc.edu,  ORCID: 0000-0003-4230-9244}
\address{Department of Mathematics, University of Wisconsin Madison, Van Vleck Hall, Madison, WI 53706, USA}
\email{thyang@math.wisc.edu,  ORCID: 0000-0002-5821-5755}


\begin{document}
	\title{Kudla-Rapoport conjecture for Kr\"amer models}
	
	\begin{abstract} In this paper, we propose a modified Kudla-Rapoport conjecture for the Kr\"amer model of unitary Rapoport-Zink space at a ramified prime, which is a precise identity relating intersection numbers of special cycles to derivatives of Hermitian local density polynomials.  We also introduce the notion of special difference cycles, which has surprisingly simple description. Combining this with induction formulas of Hermitian local density polynomials, we prove the modified Kudla-Rapoport conjecture when $n=3$. Our conjecture, combining with known results at inert and infinite primes, implies arithmetic Siegel-Weil formula for all non-singular coefficients when the level structure of the corresponding unitary Shimura variety is defined by a self-dual lattice.
	\end{abstract}
	
	\maketitle

	\setcounter{tocdepth}{1}
	\tableofcontents

	\section{Introduction}
	In their seminal work \cite{KR1} and \cite{KR2},  Kudla and Rapoport made a conjectural local arithmetic Siegel-Weil formula (the Kudla-Rapoport  conjecture) relating the intersection numbers of special divisors on unitary Rapoport-Zink spaces to the central derivative of certain local density polynomials. A unitary Rapoport-Zink (RZ) space is a local version of a  unitary Shimura variety associated to a general unitary  group $\mathrm{GU}(1,n-1)$. The Kudla-Rapoport conjecture plays a central role in  the arithmetic Siegel-Weil formula for unitary Shimura varieties,  which was
	first proposed by Kudla in \cite{Kudla97} for orthogonal Shimura varieties. When $n=1$ or $2$, the  Kudla-Rapoport conjecture was proved in \cite{KR1}. The case when $n=3$ was proved in \cite{Terstiege}. The general case was proved recently in \cite{LZ} by an ingenious induction. The Archimedean analogue of the  Kudla-Rapoport conjecture was proved in \cite{LiuarithmeticI} and \cite{GS}. The analogue of the Kudla-Rapoport conjecture for GSpin Rapoport-Zink space is formulated and proved in \cite{LZ2}.
	
	Originally the  Kudla-Rapoport conjecture was proposed only for good primes, namely inert primes over which the Rapoport-Zink space has hyperspecial level structure. A modified Kudla-Rapoport conjecture for Rapoport-Zink space with minuscule parahoric level structure over inert primes has been proposed in \cite{cho2022special}.
    For ramified primes, there are two kinds of well-understood arithmetic models of RZ spaces. One is the exotic smooth model which has good reduction, the other is the Kr\"amer model proposed in \cite{Kr} which only has semi-stable reduction. The analogue of the Kudla-Rapoport conjecture for the even-dimensional exotic smooth model was studied in \cite{LL2}, in which case the conjecture can be proved by the same strategy as in \cite{LZ}. For the Kr\"amer model, however, it was expected that  serious modification of the original  Kudla-Rapoport conjecture is needed.  A  precise formulation has  not previously been given. One of the main  goals of this paper is to formulate a precise conjecture (Conjecture \ref{conj:main}) based on earlier work of \cite{Shi2} and \cite{HSY} for the case  $n=2$. We then prove Conjecture \ref{conj:main} for $n = 3$. 
    
    In a very recent joint work with Chao Li (\cite{HLSY}), we proved the conjecture completely. One of the major innovations of \cite{HLSY} is a decomposition formula of primitive local density polynomials, which is inspired by the results in the appendix of this work. Moreover, the geometric side of the 'horizontal` part in \cite{HLSY} essentially follows from the current work. To deal with the vertical part in general, \cite{HLSY} uses partial Fourier transform inspired by \cite{LZ2}. The current work uses explicit computation instead.

	\subsection{The naive conjecture}
	Let $p$ be an odd prime and $F$ be a ramified quadratic field extension of a $p$-adic number field $F_0$ with residue field $\F_\q$.  Fix an algebraic closure $k$ of $\F_\q$. Fix a uniformizer $\pi$ of $F$ such that $\pi_0=\pi^2$ is a uniformizer of $F_0$ and let $\mathrm{v}_\pi$ be the valuation on $F$ such that $\mathrm{v}_\pi(\pi)=1$. Let $\breve{F}_0$  be the completion of a maximal unramified extension of $F_0$ and $\breve{F}\coloneqq F\otimes_{F_0} \breve{F}_0$. Let $\Oo_{\breve F}$ and $\Oo_{\breve{F}_0}$ be the rings of integers of $\breve{F}$ and $\breve{F}_0$ respectively. For a Hermitian lattice or space $M$ of rank $n$, we define its sign as
	\begin{equation}\label{eq:sign}
		\chi(M) = \chi((-1)^{\frac{n(n-1)}{2}}\mathrm{det}(M)) =\pm 1
	\end{equation}
	where $\chi$ is the quadratic character of $F_0^\times$ associated to $F/F_0$.
	We call $M$ split or non-split depending on whether   $\chi(M)=1$ or $-1$.  For a Hermitian matrix $T$, define $\chi(T)$ to be the  sign of its associated Hermitian lattice.
	
	Let $\bY$ and $\bX$ be pre-fixed  framing Hermitian formal $\Oo_F$-modules of signature $(0,1)$ and $(1,n-1)$, respectively, over $\Spec k$. Recall that Hermitian formal $\Oo_F$-modules are a particular kind of formal $p$-divisible groups with $\Oo_F$-action, see Section \ref{subsec:RZspaces}. The space of special quasi-homomorphisms
	\begin{equation}\label{eq:bV introduction}
		\bV=\Hom_{\Oo_F}(\bY,\bX)\otimes_\Z \Q
	\end{equation}
	is equipped with a Hermitian form $h(\,,)$, see \eqref{eq:h(x,y)}.
	Let $\epsilon=\chi(\bV)$. The Rapoport-Zink space $\cN^\Kra_{n,\epsilon}$ parameterizes certain classes of supersingular Hermitian formal $\Oo_F$-modules of signature $(1,n-1)$ over $\SpfOF$, see Section \ref{subsec:RZspaces}. It is a formal scheme over $\SpfOF$ with semi-stable reduction and can be viewed as a regular model of the formal completion of the corresponding global unitary Shimura variety along its basic locus over $p$. When $n$ is odd, $\cN^\Kra_{n,1}$ is isomorphic to  $\cN^\Kra_{n,-1}$. We often write $\cN^\Kra$ instead of $\cN^\Kra_{n,\epsilon}$ for simplicity.
	
	For each subset $L \subset \bV$, define $\cZ^\Kra(L)$  to be the formal subscheme of $\cN^\Kra$ where $\bx$ deforms to a homomorphism for any $\bx\in L$.
	Let $L\subset \bV$ be an $\Oo_F$-lattice of rank $r$. We say $L$ is integral if $h(\,,)|_L$ is non-degenerate and takes values in $\Oo_F$. Let $\bx_1,\ldots,\bx_r$ be a basis of $L$. We define
	\begin{equation}
		{}^\bL \cZ^\Kra(L)=[\Oo_{\cZ^\Kra(\bx_1)}\otimes^{\bL}\cdots \otimes^{\bL}\Oo_{\cZ^\Kra(\bx_r)}]\in K_0(\cN^\Kra)
	\end{equation}	
	where $\otimes^\bL$ is the derived tensor product of complex of coherent sheaves on $\cN^\Kra$ and $K_0(\cN^\Kra)$ is the Grothendieck groups of finite complexes of coherent locally free sheaves on $\cN^\Kra$. By \cite[Corollary C]{Ho2}, ${}^\bL \cZ^\Kra(L)$ is independent of the choice of basis of $L$.
	When  $L$ has rank $n$, we define the intersection number
	\begin{equation}
		\mathrm{Int}(L)=\chi(\cN^\Kra,{}^\bL \cZ^\Kra(L))
	\end{equation}
	where $\chi$ is the Euler characteristic. One can show that $\Int(L)$ is finite, see Lemma \ref{lem:finiteness of Int L}.

	  Let $L$ and $M$ are Hermitian lattices of rank $n$ and $m$ respectively. Moreover, we assume $\mathrm{v}(M)\coloneqq\mathrm{min}\{\mathrm{v}_\pi( h(v,v'))\mid v,v'\in M\} \ge -1.$ We use $\mathrm{Herm}_{L,M}$ to denote the scheme of Hermitian $\Oo_F$-module homomorphisms from $L$ to $M$, which is a scheme of finite type over $\Oo_{F_0}$.  More specifically, for an $\Oo_{F_0}$-algebra $R$, we define
\begin{align*}
L_{R}\coloneqq L\otimes_{\Oo_{F_0}}R, \quad  (x\otimes a,y\otimes b)_{R}\coloneqq \pi (x ,y )\otimes_{\Oo_{F_0}} ab\in  \Oo_{F}\otimes_{\Oo_{F_0}} R \text{ where }x,y \in L, a,b\in R.
\end{align*}

Then
	\begin{align*}
	\mathrm{Herm}_{L, M}(R)=\{ \phi \in  \mathrm{Hom}&_{\Oo_F}(L_{R}, M_{R}) \mid     (\phi(x),\phi(y))_{R}\equiv (x,y)_{R}  \text{ for all }  x,y \in L_{R}\}.
	\end{align*}
To simplify the notation, let $I(M,L,d)$ denote $\mathrm{Herm}_{L, M}(\Oo_{F_0}/(\pi_0^d))$. Then  direct calculation shows that
	\begin{equation}\label{eq:local density M,L intro}
		\alpha(M, L) = \q^{-d n (2m -n)} |I(M,L,d)|
	\end{equation}
	becomes constant for sufficiently large integers $d >0$. We call it the local density (of $M$ representing $L$).

	Let $\cH$ be the (Hermitian) hyperbolic plane  with   Gram matrix $ \cH=\begin{pmatrix}
		0 & \pi^{-1}\\
		-\pi^{-1} & 0
	\end{pmatrix}$. One can show that there is a (local density) polynomial $\alpha(M,L,X)\in \Q[X]$ such that
	\[\alpha(M\obot\cH^k,L)=\alpha(M,L,\q^{-2k}).\]
	Define its derivative by
	\[\alpha'(M,L)\coloneqq -\frac{\partial}{\partial X} \alpha(M,L,X)|_{X=1}.\]
As $\alpha(M,L, X)$ (respectively, $\alpha'(M,L)$) only depends on their Gram matrices $S$ and $T$, we will  also denote it by  $\alpha(S, T, X)$ (respectively, $\alpha'(S,T)$).	Let $M$ be the unique unimodular Hermitian $\Oo_F$-lattice of rank $n$ with $\chi(M)=-\chi(L)$.
	The naive analogue of the local Kudla-Rapoport conjecture  is
	\begin{equation}\label{eq:naiveconjecture}
		\mathrm{Int}(L)=2 \frac{\alpha'(M,L)}{\alpha(M,M)}.
	\end{equation}
	However, this conjectural formula is not even  true for $n=2$ according to the main theorem of \cite{HSY}. The analytic side of the conjecture needs to be modified.
	\subsection{The precise conjecture}
	By \cite[Theorem 1.2]{Shi1},   $\cZ^\Kra(L)$ is empty when $L$ is not integral, so we have
	\[\mathrm{Int}(L)=0.\]
	On the analytic side,  the right hand side of \eqref{eq:naiveconjecture} is   automatically zero  only when $\mathrm{v}(L) \leq -2$, and is sometimes non-zero when  $\mathrm{v}(L) =-1$.
	Thus, there should be   correction terms  involving Hermitian  lattices $M$ with $\mathrm{v}(M) =-1$.
	By \cite{J}, there are $n-1$ equivalent classes of  Hermitian lattices which are direct sum of copies of $\cH$ and unimodular lattices:
	\begin{equation}\label{eq:H n i epsilon intro}
		\cH_{n,i}^{\epsilon}\coloneqq \cH^i\obot I_{n-2i}^{\epsilon} \quad \text{ for } 1\leq i \leq \frac{n}{2}, \quad \epsilon =\pm 1
	\end{equation}
	where we use $I_{n-2i}^{\epsilon}$ to denote the  unimodular Hermitian lattice of rank $n-2i$ with
	\noindent $\chi(I_{n-2i}^{\epsilon})= \chi(\cH_{n,i}^{\epsilon} )=\epsilon$.  When  $n=2r$ is even, we take $I_{0, \epsilon} =0$ and  $\cH_{n, r}^1 = \cH^{r}$.  Then the local arithmetic Siegel-Weil formula, (also known as the Kudla-Rapoport conjecture  at a ramified prime) should be of the following form:
	\begin{equation}\label{eq:localASW}
		\Int(L) = 2 \frac{\alpha'(I_{n}^{-\epsilon}, L)}{\alpha(I_{n}^{-\epsilon}, I_{n}^{-\epsilon})} + \sum_{ i} c_{n,i}^\epsilon  \frac{\alpha(\cH_{n,i}^{\epsilon}, L)}{\alpha(I_{n}^{-\epsilon}, I_{n}^{-\epsilon})},
	\end{equation}
	where $\epsilon=\chi( L)$. Since $\Int(\cH_{n, j}^{\epsilon})=0$,  we should have
	\begin{equation} \label{eq:coeff}
		2 \frac{\alpha'(I_{n}^{-\epsilon},\cH_{n, j}^{\epsilon} )}{\alpha(I_{n}^{-\epsilon}, I_{n}^{-\epsilon})} + \sum_{ i} c_{n,i}^\epsilon  \frac{\alpha(\cH_{n,i}^{\epsilon},\cH_{n,j}^{\epsilon})}{\alpha( I_{n}^{-\epsilon}, I_{n}^{-\epsilon})} =0.
	\end{equation}
	This system of equations turns out to determine the coefficients $c_{n,i}^\epsilon $ uniquely by Theorem \ref{thm: A_epsilon}. We propose the following Kudla-Rapoport conjecture at a ramified prime.
	
	\begin{conjecture}\label{conj:main} The identity $(\ref{eq:localASW})$ always holds with the coefficients $c_{n,i}^\epsilon $ uniquely determined by (\ref{eq:coeff}).
	\end{conjecture}
	
	For convenience, we set 
	\begin{equation} \label{eq1.10}
		\partial \Den(L) =2 \frac{\alpha'(I_{n}^{-\epsilon}, L)}{\alpha(I_{n}^{-\epsilon},I_{n}^{-\epsilon} )} + \sum_{ i} c_{n,i}^\epsilon  \frac{\alpha(\cH_{n,i}^{\epsilon}, L)}{\alpha(I_{n}^{-\epsilon},I_{n}^{-\epsilon})}.
	\end{equation}
We remark that because $\Int(L)$ is always an integer, the conjecture	\ref{conj:main} suggests that $\pden(L)$ should be an integer, which is already not obvious.

	The conjecture holds for $\cN^\Kra_{2,\pm 1}$ by results in \cite{Shi2} and \cite{HSY}. In this paper,  we will prove  the conjecture for $n=3$ and provide some partial results in the general case.
	
	\begin{theorem}\label{thm:mainthmintroduction}
		Conjecture \ref{conj:main} is true when $n=3$.
	\end{theorem}

	\subsection{Special difference cycles}\label{subsec:difference cycle intro}
	One of the novelties of this paper is the concept of special difference cycles. Let $L_1$ be an $\Oo_F$-lattice of $\bV$ of rank $n_1 \le n$. Define
	the special difference cycle $\cD(L_1)\in K_0 (\cN^\Kra)$ by
	\begin{equation}\label{eq:difference cycle intro}
		\cD(L_1)={}^\bL\cZ^\Kra(L_1)+\sum_{i=1}^{n_1} (-1)^i \q^{i(i-1)/2} \sum_{\substack{L_1 \subset L'\subset \frac{1}{\pi} L_1\\ \mathrm{dim}_{\F_\q}(L'/L_1)=i}} {}^\bL\cZ^\Kra(L')\in K_0(\cN^\Kra).
	\end{equation}
Here	$\cD(L_1)$ can be seen as a higher codimensional analogue of the difference divisor first introduced in \cite[Definition 2.10]{Terstiege}.
	By the definition and a $q$-adic linear-algebraic inclusion-exclusion principle, we have (see Lemma \ref{lem: decomposition of Z(L) as sums of D(L)})
	\begin{equation}\label{eq: decomposition of Z(L) as sums of D(L)}
		{}^\bL\cZ^\Kra(L_1) =\sum_{\substack{L' \mathrm{ integral}\\ L_1 \subset L' \subset L_{1,F}}}   \cD(L').
	\end{equation}
	Here $L_F =L\otimes_{\Oo_F} F$ for an $\Oo_F$-lattice $L$. The above summation is, in fact, finite.	 Assume that we have a decomposition $L=L_1\oplus L_2$ of $\Oo_F$-lattices such that $L_i$ has rank $n_i$ and $n_1+n_2=n$. Define
	\begin{equation}
		\Int(L)^{(n_1)}=\chi(\cN^\Kra, \cD(L_1)\cdot \cZ^\Kra(L_2))
	\end{equation}
	where $\cdot$ is the product on $K_0(\cN^\Kra)$ induced by tensor product of complexes. Notice that this definition, in fact, depends on the decomposition of $L$.
	
	On analytic side, we  define
	\begin{align}
		\pden(L)^{(n_1)}\coloneqq \pden(L)-\sum_{i=1}^{n_1} (-1)^{i-1} \q^{i(i-1)/2}
		\sum_{\substack{L_1 \subset L_1' \subset L_{1, F} \\ \dim {L_1'/L_1}=i}} \pden(L'_1\oplus L_2).
	\end{align}
	This definition again  depends on the decomposition of $L$. What motivates the definition of $\pden(L)^{(n_1)}$ and  $\cD(L_1)$ is the fact that $\pden(L)^{(n_1)}$ is equal to the derivative of certain primitive local density polynomials, see \cite[Proposition 2.1]{katsurada1999explicit} or Theorem \ref{thm:ind formula reducing valuation} below.
	The analogue of \eqref{eq: decomposition of Z(L) as sums of D(L)} holds for $\pden(L)^{(n_1)}$. As a consequence we have the following theorem (see Theorem \ref{thm: equivalent form of modified KR} for a refinement).
	\begin{theorem} \label{theo1.3}
		Conjecture \ref{conj:main} is true if and only if for every lattice $L=L_1\oplus L_2$ such that $L_i$ has rank $n_i$, we have
		\begin{equation}\label{eq: equivlent form of KR intro}
			\Int(L)^{(n_1)}=\pden(L)^{(n_1)}.
		\end{equation}
	\end{theorem}
	
	We speculate that $\cD(L_1)$ is of a simple form when $n_1=n-1$.
	One strong piece of evidence for this is that the `horizontal' part of $\cD(L_1)$ is either empty or isomorphic to one or two copies of $\Spf W_s$ where $W_s$ is the integer ring of an extension of $\breve{F}$ of degree $\q^s$, see Proposition \ref{prop:horizontalpartofDL}. Another evidence is that
	the intersection of $\cD(L_1)$ with an exceptional divisor in $\cN^\Kra$ is  $\pm 1$ or $0$, see Lemma \ref{lem: Z(L hecke) dot Exc}.
	When $n=3$, we show that $\cD(L_1)$ has a simple decomposition, see Theorem \ref{thm:decomposition of cD intro} below.

	\subsection{The case $n=3$}
	The proof of Theorem \ref{thm:mainthmintroduction} is  divided into three cases, see Section \ref{sec:proof when n=3}. For $\mathrm{v}(L)<0$, we show directly  $\pden(L)=\Int(L)=0$. The case  $\mathrm{v}(L)=0$ is reduced to  the case  $n=2$, which was  proved in \cite{Shi2} and \cite{HSY}. For  $\mathrm{v}(L) >0$, we prove that $\Int(L)^{(2)} = \pden(L)^{(2)}$ for a decomposition $L=L^\flat\obot \spa\{\bx\}$, and then apply Theorem \ref{theo1.3} (more precisely Theorem \ref{thm: equivalent form of modified KR}).
	
	In order to prove $\Int(L)^{(2)} = \pden(L)^{(2)}$, we need to understand the decomposition of $\cD(L^\flat)$.
	We say a lattice $\Lambda\subset \bV$ is a vertex lattice if $\pi \Lambda\subseteq \Lambda^\sharp \subseteq \Lambda$ where $\Lambda^\sharp$ is dual lattice of $\Lambda$ with respect to $h(\, ,)$ and we call $t=\mathrm{dim}_{\F_\q}(\Lambda/\Lambda^\sharp)$ the type of $\Lambda$. This has to be an even integer between $0$ and $n$. We denote the set of vertex lattices of type $t$ by $\cV^t$.
	When $n=3$, a type $2$ lattice $\Lambda_2$ corresponds to a line $\tN_{\Lambda_2}\cong \bP^1_k$ in $\cN^\Kra_3$ and a type $0$ lattice $\Lambda_0$ corresponds to a divisor $\Exc_{\Lambda_0}\cong \bP^2_k$.  Let $H_{\Lambda_0}$ be the hyperplane class of $\Exc_{\Lambda_0}$.
	We have the following theorem.
	
	\begin{theorem}\label{thm:decomposition of cD intro}
		If $\mathrm{v}(L^\flat)>0$, we have the following decomposition of cycles in $\mathrm{Gr}^2 K_0(\cN^\Kra_3)$
		\[\cD(L^\flat)=\sum_{\substack{  \Lambda_2\in \cV^2 \\ L^\flat \subset \Lambda_2^\sharp}}  (2[\Oo_{\tN_{\Lambda_2}}]+\sum_{\substack{\Lambda_0\in\cV^0\\\Lambda_0\subset \Lambda_2}}H_{\Lambda_0})\]
		where $\mathrm{Gr}^\bullet K_0(\cN^\Kra_3)$ is the associated graded ring of $K_0(\cN^\Kra_3)$ with respect to the codimension filtration.
	\end{theorem}
	Theorem \ref{thm:decomposition of cD intro} is proved by intersecting $\cD(L^\flat)$ with special divisors that are isomorphic to $\cN^\Kra_{2,-1}$ and computing the intersection numbers in two different ways. One way relates the intersection numbers to the main result of \cite{Shi2}.  The other way uses the decomposition in Theorem \ref{thm:decomposition of cD intro} and detects the multiplicity of each component that shows up on the right hand side.

	\subsection{Global application}
	In the  last part of the paper, we apply the local results above to the global intersection problem proposed by \cite{KR2}. For brevity and clarity of exposition we restrict our attention to the case when $F$ is an imaginary quadratic field. We remark here that our results can be applied to the case when $F$ is a general CM field given correct local assumptions.
	Let $\cM^\Kra_{(1,n-1)}$ be the moduli functor over $\Spec \Oo_F$ which parametrizes principally polarized abelian varieties $A$ with an action $\iota:\Oo_F \rightarrow \End(A)$, a compatible polarization $\lambda: A\rightarrow A^\vee$ and a filtration $\cF_A\subset \Lie A$ which satisfies the signature $(1,n-1)$ condition (see \S \ref{subsec:unitary shimura}). Let $V$ be a  Hermitian vector space over $F$ of signature $(n-1,1)$ containing a self-dual lattice $L$. The lattice $L$ determines an open and closed substack
	\[\cM\subset\cM_{(0,1)}\times_{\Spec \Oo_F} \cM^\Kra_{(1,n-1)}\]
	which is an integral model of a unitary Shimura variety.
	For a point in $\cM(S)$ ($S$ an $\Oo_F$-scheme), i.e.,
	a pair $(E,\iota_0,\lambda_0)\in \cM_{(0,1)}(S)$, $(A,\iota,\lambda,\cF_A)\in \cM^\Kra_{(1,n-1)}(S)$, define the locally free $\Oo_F$-module
	\[V'(E,A)=\Hom_{\Oo_F}(E,A).\]
	It is equipped with the Hermitian form $h'(x,y)=\iota_0^{-1}(\lambda_0^{-1}\circ y^\vee \circ \lambda \circ x)$. For a $m \times m$ non-singular Hermitian matrix $T$ with values in $\Oo_F$, let $\cZ(T)$ be the stack of collections $(E,\iota_0,\lambda_0, A,\iota, \lambda,\cF_A,\bx)$ such that $ (E,\iota_0,\lambda_0, A,\iota, \lambda,\cF_A)\in \cM(S)$, $\bx \in V'(E,A)^m$ with $h'(\bx,\bx)=T$. Then $\cZ(T)$ is representable by a Deligne-Mumford stack which is finite and unramified over $\cM$ (\cite[Proposition 2.9]{KR2}). When $t\in \Z_{>0}$, each component of $\cZ(t)$ can be viewed as a divisor by \cite[Proposition 3.2.3]{HowardCMII}. In general, $\cZ(T)$ does not necessarily have the expected codimension which is the rank of $T$.

	Let $\mathcal C =\{ \mathcal C_p\} $ be an incoherent collection of  local  Hermitian spaces of rank $n$ associated to $V$ such that $\mathcal C_\ell\cong V_\ell$ for all finite $\ell$ and $\mathcal C_\infty$ is positive-definite. It is ``incoherent" in the sense that it does not come from a global Hermitian space.
	For a non-singular Hermitian matrix $T$ of rank $n$ with values in $\Oo_F$, let $V_T$ be the Hermitian space with Gram matrix $T$. Define
	\begin{equation}
		\mathrm{Diff}(T,\mathcal C)\coloneqq \{p \text{ a place of } \Q  \mid \mathcal C_p \text{ is not isomorphic to } (V_T)_p\}.
	\end{equation}
	Then $\cZ(T)$ is empty if $|\mathrm{Diff}(T,\mathcal C)|>1$. If $\mathrm{Diff}(T,\mathcal C)=\{p\}$ for a finite prime $p$ inert or ramified in $F$, then the support of $\cZ(T)$ is on the supersingular locus of $\cM$ over $\Spec \F_p$. Let $e$ be the ramification index of $F_p/\Q_p$. Define the arithmetic degree
	\begin{equation}	\widehat{\mathrm{deg}}_T=\chi(\cZ(T),\Oo_{\cZ(t_1)}\otimes^{\mathbb{L}} \cdots \otimes^{\mathbb{L}} \Oo_{\cZ(t_n)})\cdot \log p^{2/e},
	\end{equation}
	where $\otimes^\bL$ stands for derived tensor product on the category of coherent sheaves on $\cM$, $\chi$ is the Euler characteristic and $t_i$ ($1 \le i \le n$) are the diagonal entries of $T$. When $\mathrm{Diff}(T,\mathcal C)=\{\infty\}$, then  $\cZ(T)$ is, in fact, empty and one can use integration of a green current to define the arithmetic degree $\widehat{\mathrm{deg}}_{T}(v)$ with the parameter $v$ being a positive-definite Hermitian matrix $v$ of order $n$ (which will be the imaginary part of $\tau$), see, for example, \cite[\S 15.3]{LZ}.

	On the analytic side, we consider an incoherent Eisenstein series $E(\tau,s,\Phi)$ for a non-standard section $\Phi$ in a degenerate principal series representation of $\rU(n,n)(\bA)$, see Section \ref{subsec:Eisenstein series}. Here $\tau$ is in the Hermitian Siegel upper half space
 \begin{equation}\label{eq:Siegel upper half plane}
     \mathbb{H}_n=\{\tau =u+i v\mid u\in \Herm_n, v\in \Herm_{n,>0}\},
 \end{equation}
    where $\Herm_n$ (respectively, $\Herm_{n,>0}$) is the set of $n\times n$ (positive-definite) Hermitian matrices with values in $\C$ and $s\in \C$.
   Our local conjecture and result imply the following theorem, which extends \cite[Theorem 1.3.1]{LZ} to include ramified primes.

	\begin{theorem}(Arithmetic Siegel-Weil formula for non-singular coefficients) \label{thm:main global intro}
		Assume that the fundamental discriminant of $F$ is $d_F \equiv 1 \pmod 8$ and that Conjecture \ref{conj:main}  holds for every $F_p$ with  $p|d_F$.
		For any non-singular Hermitian matrix $T$ with values in $\Oo_F$ of size $n$, we have
		\[E'_T(\tau, 0,\Phi)=C \cdot \widehat{\mathrm{deg}}_T(v) \cdot q^T, \quad  q^T=\exp(2\pi i \tr(T\tau)),\]
		where $E'_T(\tau, 0,\Phi)$ is the $T$th Fourier coefficient of $E'(\tau, 0,\Phi)$ and $C$ is a constant that only depends on $F$ and $L$. In particular, the arithmetic Siegel-Weil formula holds for $n=2, 3$ for non-singular $T$.
	\end{theorem}

In a very recent joint work with Chao Li (\cite{HLSY}), we proved Conjecture \ref{conj:main}, and so this theorem is now unconditional.
	
	\subsection{Notation}	
	For $\Oo_F$-lattices (respectively, $\Oo_{\breve F}$-lattices) $L$ and $L'$, we write $L \stackrel{t}{\subset} L'$ if $L\subset L' \subset \frac{1}{\pi} L$ and $\dim_{\F_q}(L'/L)=t$ (respectively, $\dim_k (L'/L)=t$). We say a vector $v\in L$ is primitive if $\frac{1}{\pi}v\notin L$.
	
Throughout the paper, we always assume a Hermitian lattice is non-degenerate.	For Hermitian lattices $L$ and $L'$, we use $L\obot L'$ to denote orthogonal direct sum, and $L \oplus L'$ as direct sum of lattices. 	Given a Hermitian lattice $L$ with Hermitian form $(\, ,)$, we consider two different dual lattices of $L$. We use $L^{\sharp}$ (respectively,$\, L^{\vee})$ to denote the dual lattice of $L$ with respect to $(\, ,)$ (respectively,$\, \tr_{F/F_0}(\, ,))$.  Recall that $\mathrm{v}(L)$ is defined to be $\mathrm{min}\{\mathrm{v}_\pi( h(v,v'))\mid v,v'\in L\}$.
	For each Hermitian lattice $L$, there  exists a Jordan decomposition $L=\obot_{i\ge s} L_i$ such that $L_i^{\sharp}=\pi^{-i} L_i$. We call $L$ integral if $s\ge 0$. For an integral lattice $L$, we define
	$$
	t(L)\coloneqq \sum_{i\ge 1}\mathrm{rank}_{\Oo_F}(L_i).
	$$

Following \cite[Definition 2.11]{LL2}, for a lattice $L$ with Hermitian form $(\, ,)$, we may find a basis of $L$ whose Gram matrix is
\begin{align*}
\left(\beta_{1} \pi^{2 b_{1}}\right) \oplus \cdots \oplus\left(\beta_{s} \pi^{2 b_{s}}\right) \oplus\left(\begin{array}{cc}
0 & \pi^{2 c_{1}+1} \\
-\pi^{2 c_{1}+1} & 0
\end{array}\right) \oplus \cdots \oplus\left(\begin{array}{cc}
0 & \pi^{2 c_{t}+1} \\
-\pi^{2 c_{t}+1} & 0
\end{array}\right)
\end{align*}
for some $\beta_{1}, \ldots, \beta_{s} \in \Oo_{F_0}^{\times}$ and $b_{1}, \ldots, b_{s}, c_{1}, \ldots, c_{t} \in \mathbb{Z}$. Moreover, we define its (unitary) fundamental invariants $(a_1,\cdots,a_n)$ to be the unique non-decreasing rearrangement of $(2 b_1,\cdots,2b_s,2c_1+1,\cdots,2c_t+1)$. The partial order of $\Z^n$ induces a partial order on the set of fundamental invariants.

Let $\cH_i=\begin{pmatrix}
		0& \pi^{i}\\
		(-\pi)^{i}& 0
	\end{pmatrix}$ and $\cH=\cH_{-1}$. We also use it to denote a Hermitian lattice with Gram matrix $\cH_i$. Given a Hermitian lattice $M$, we use $M^{[k]}$ to denote $M\obot \cH^k$. We use $I_m^\epsilon$ to denote a unimodular Hermitian lattice of rank $m$ and $\chi(I_m^\epsilon)=\epsilon.$  For a Hermitian matrix $T$, we define  $\mathrm{v}(T)=\mathrm{v}(L)$ where $L$ is a lattice whose Gram matrix is $T$.  We use $\Herm_n(F)$ to denote the set of Hermitian matrices over $F$ of size $n$. When there is no confusion, we also simply denote it as $\Herm_n$.  For $T,T'\in \Herm_n(F)$, we say $T$ is equivalent to $T'$ if there is a $U\in \GL_n(\Oo_F)$ such that $U^*TU=T'$, where $U^*= {}^t \bar U$. In this case, we denote it as $T\approx T'$.

		For $t\in \Oo_{F_0}$, let $\mathrm{v}(t)\coloneqq \mathrm{val}_{\pi_0}(t)$ and write $t=t_0(-\pi_0)^{\mathrm{v}(t)}$. For $x\in \bV$, we set $q(x)=(x,x)$ and $\mathrm{v}(x)=\mathrm{v}(q(x))$. We use $\langle t \rangle$ to denote a lattice $\Oo_F x$ of rank one with $q(x)=t$.

 The notation in each section that is not mentioned here  will be explained at the very beginning of the section.

	\subsection{The structure of the paper}
	The paper is divided into three parts.
	In Part \ref{part:I}, we prove some facts about special cycles for arbitrary $n$. More specifically,
	in Section \ref{sec: RZ space} we recall some basic facts about $\cN^\Kra$ and define special cycles and special difference cycles on it. In Section \ref{sec: Exc}, we compute the intersection number between special cycles and the exceptional divisors. In Section \ref{sec: horizontal part}, we prove a decomposition theorem for the horizontal component of ${}^\bL \cZ^\Kra(L^\flat)$ when $L^\flat$ has rank $n-1$.
	
	Part \ref{part:II} is about Hermitian local densities.
	In Section \ref{sec:ind formula}, we study induction formulas of local density polynomials and relate the local density polynomials with primitive local densities. In Section \ref{sec: Kudla-Rapoport conjecture}, we show that the coefficients $ c_{n,i}^\epsilon $ in \eqref{eq:coeff} are uniquely determined and give an algorithm to compute them. In Sections \ref{sec: eg application} and \ref{sec: local density n=3}, we compute the local density polynomials when $n\leq 3$.
	
	In Part \ref{part:III} we prove Theorem \ref{thm:mainthmintroduction}, i.e. Conjecture \ref{conj:main} for $n=3$. In Section \ref{sec:reduced locus n=3}, we study the reduced locus of the special cycles for $n=3$. In Section \ref{sec:intersection of vertical and special}, we decompose ${}^\bL \cZ^\Kra(L^\flat)$ for $L^\flat$ of rank $2$ and $\val(L^\flat)=0$, and compute the intersection number of $\tN_{\Lambda_2}$ with $\cZ^\Kra(\bx)$.
	Finally, we prove Theorem \ref{thm:decomposition of cD intro} and finish the proof of Theorem \ref{thm:mainthmintroduction} in Section \ref{sec:proof when n=3} and explain its global applications in \S \ref{sec:global application}.

	In Appendix \ref{sec: calc of primitive local density}, we compute the primitive local densities that are used in Part \ref{part:II} of the paper.
	
	\subsection{Acknowledgement} We thank Chao Li for his help during the preparation of this paper. We thank the referees for their careful reading of the paper and their comments which make the paper more readable.

	\part{The geometric side}\label{part:I}
	
	\section{Rapoport-Zink space and special cycle}\label{sec: RZ space}
	We denote $\bar{a}$ the Galois conjugate of $a\in F$ over $F_0$. Let $\Nilp \Oo_{\breve F}$ be the category of $\Oo_{\breve F}$-schemes $S$ such that $\pi$ is locally nilpotent on $S$. For such an $S$, denote its special fiber $S\times_{\Spf \Oo_{\breve F}} \Spec k$ by $\bar S$. Let $\sigma$ be the Frobenius element of $ \breve{F}_0/F_0$.
	
	\subsection{RZ spaces}\label{subsec:RZspaces}
	Let $S\in\Nilp \Oo_{\breve F}$. A $p$-divisible strict $\Oo_{F_0}$-module over $S$ is a $p$-divisible group over $S$ with an $\Oo_{F_0}$ action whose induced action on its Lie algebra is via the structural morphism $\Oo_{F_0}\rightarrow \Oo_S$.
	\begin{definition}
		A formal Hermitian $\Oo_F$-module of dimension $n$ over $S$ is a triple $(X,\iota,\lambda)$ where $X$ is a supersingular $p$-divisible strict $\Oo_{F_0}$-module over $S$  of dimension $n$ and $F_0$-height $2n$ (supersingular means the relative Dieudonn\'e module of $X$ at each geometric point of $S$ has slope $\frac{1}{2}$), $\iota:\Oo_F\rightarrow \End(X)$ is an $\Oo_F$-action and $\lambda:X\rightarrow X^\vee$ is a principal polarization in the category of strict $\Oo_{F_0}$-modules such that the Rosati involution induced by $\lambda$ is the Galois conjugation of $F/F_0$ when restricted on $\Oo_F$. We say $(X,\iota,\lambda)$ satisfies the signature condition $(1,n-1)$ if for all $a\in \Oo_F$ we have
		\begin{enumerate}[leftmargin=*, label=({\roman*})]
			\item $\mathrm{char}(\iota(a)\mid \Lie X)=(T-s(a))\cdot(T-s(\bar{a}))^{n-1} $ where $s:\Oo_F\rightarrow \Oo_S$ is the structure morphism;
			\item the wedge condition proposed in \cite{P}:
   \[ \wedge^n(\iota(a)-s(a)\mid \Lie X)=0,\ \wedge^2(\iota(a)-s(\bar{a})\mid \Lie X)=0.\]
		\end{enumerate}
	\end{definition}
    Let $(\bX,\iota_\bX,\lambda_\bX)$ be a formal Hermitian $\Oo_F$-module of dimension $n$ over $k$, and $N$ be its rational relative Dieudonne module.  Then $N$ is an $2n$-dimensional $\breve{F}_0$-vector space equipped with a $\sigma$-linear operator $\bfF$ and a $\sigma^{-1}$-linear operator $\bfV$. The $\Oo_F$-action  $\iota_\bX:\Oo_F\rightarrow \End(\bX)$ induces on $N$ an $\Oo_F$-action  commuting with $\bfF$ and $\bfV$.  We still denote this induced action by $\iota_\bX$ and denote $\iota_\bX(\pi)$ by $\pi$. Let $\tau\coloneqq \pi \bfV^{-1}$ and $ C\coloneqq N^\tau$. Then $C$ is an $n$-dimensional $F$-vector space equipped with a Hermitian form $(\, ,)_\bX$ defined using the polarization $\lambda_\bX$, see \cite[Equation (2.7)]{Shi1}.  When $n$ is odd, there is a unique choice of $(\bX,\iota_\bX,\lambda_\bX)$ up to quasi-isogenies that preserves  the polarization by a factor in $\Oo_{F_0}^\times$. When $n$ is even, there
	are two such choices according to the sign $\epsilon=\chi(C)$ (see \eqref{eq:sign}) of $C$. See \cite[Remark 2.16]{Shi1} and \cite[Remark 4.2]{RTW}. Fix a formal Hermitian $\Oo_F$-module $(\bY,\iota_\bY,\lambda_\bY)$ of signature $(0,1)$ over $\Spec k$.  It is unique up to $\Oo_F$-linear isomorphisms. Define
	\begin{equation}\label{eq:bV}
		\bV=\Hom_{\Oo_F}(\bY,\bX)\otimes \Q,
	\end{equation}
	which is equipped with a Hermitian form
	\begin{equation}\label{eq:h(x,y)}
		h(x,y)=\lambda_\bY^{-1}\circ y^\vee \circ \lambda_\bX \circ x\in \End^0_F(\bY)\overset{\sim}{\rightarrow}F
	\end{equation}
	where $y^\vee$ is the dual quasi-homomorphism of $y$ and $\End^0_F(\bY)$  is the ring of $F$-linear quasi-endomorphisms of $\bY$.
	The Hermitian spaces $(\bV,h(\, ,))$ and $(C,(\, , )_\bX)$ are related by the $F$-linear isomorphism
	\begin{equation}\label{eq: C V isomorphism}
		b: \bV\rightarrow C, \quad \bx\mapsto \bx(e)
	\end{equation}
	where $e$ is a generator of the relative covariant Dieudonn\'e module $M(\bY)$ of $\bY$.
    Let $(\, ,)_\bY$ be the analogue of $(\, ,)_\bX$ for $\bY$, namely the Hermition form on the rational relative Dieudonn\'e module of $\bY$ defined by $\lambda_\bY$.
    By \cite[Lemma 3.6]{Shi1}, we have
	\begin{equation}\label{eq:h and (,)_X}
		h(\bx,\bx)(e,e)_\bY=(b(\bx),b(\bx))_\bX.
	\end{equation}
	By scaling the Hermitian form $(\, ,)_\bY$ we can assume that
	\[(e,e)_\bY=1,\]
	so $\bV$ and $C$ are isomorphic as Hermitian spaces. We will sometimes identify $\bV$ and $C$.

	\begin{definition}\label{def:NPappas}
		Fix a formal Hermitian $\Oo_F$-module $(\bX,\iota_\bX,\lambda_\bX)$ of dimension $n$ over $k$ with the sign $\epsilon=\chi(C)$.
		The moduli space $\cN_{n,\epsilon}^\Pap$ is the functor sending each $S\in \Nilp \Oo_{\breve F}$ to the groupoid of isomorphism classes of quadruples $ (X,\iota,\lambda,\rho)$ where $(X,\iota,\lambda)$ is a formal Hermitian $\Oo_F$-module over $S$ of signature $(1,n-1)$ and $\rho:X\times_{S}\bar{S} \rightarrow \bX \times_{\Spec k} \bar{S}$ is a quasi-morphism of formal $\Oo_F$-modules of height $0$. An isomorphism between two such quadruples $(X,\iota,\lambda,\rho)$ and $(X',\iota',\lambda',\rho')$ is given by an $\Oo_F$-linear isomorphism $\alpha:X\rightarrow X'$ such that $\rho'\circ(\alpha\times_S \bar{S})=\rho $ and $\alpha^*(\lambda')$ is a $\Oo_{F_0}^\times$ multiple of $\lambda$. We drop the subscript $\epsilon$ in $\cN_{n,\epsilon}^\Pap$ when we do not emphasize on the sign.
	\end{definition}

     By the discussion before \eqref{eq:bV}, when $n$ is odd, two different choices of $\epsilon$ give us isomorphic moduli spaces. When $n$ is even, two different choices of $\epsilon$ give us two sets of non-isomorphic moduli spaces.
     By \cite{RTW}, $\cN_n^\Pap$ is representable by a formal scheme flat and of relative dimension $n-1$ over $\Spf \Oo_{\breve F}$. We remark here that although \cite{RTW} works on the category of $p$-divisible groups (namely when $F_0=\Q_p$), their arguments and results easily extend to the category of strict formal $\Oo_{F_0}$-modules using relative Dieudonn\'e theory or more generally the relative display theory developed in \cite{ACZ}. When $n=1$, we have $\cN_1^\Pap\cong \Spf\Oo_{\breve F}$.
	The universal Hermitian $\Oo_F$-module over $\cN_1^\Pap$ is the canonical lifting $(\cG,\iota_\cG,\lambda_\cG,\rho_\cG)$ of $(\bY,\iota_\bY,\lambda_\bY)$ to $\Spf \Oo_{\breve F}$ in the sense of \cite{G}. When $n>1$, $\cN_n^\Pap$ is regular outside the set of superspecial points over $\Spec k$, which are the points characterized by the condition $\iota(\pi)|_{\Lie X}=0$. The set of superspecial points is in fact the set of type $0$ lattices (see Section \ref{subsec:bruhat tits in general}), hence is isolated and we denote it by Sing.
	\begin{definition}\label{def:NKra}
		Fix $(\bX,\iota_\bX,\lambda_\bX)$ be as in Definition \ref{def:NPappas}. The moduli space $\cN_{n,\epsilon}^\Kra$ is the functor sending each $S\in \Nilp \Oo_{\breve F}$ to the groupoid of isomorphism classes of quintuples $ (X,\iota,\lambda,\rho,\cF)$ where $(X,\iota,\lambda,\rho)\in \cN_{n,\epsilon}^\Pap(S)$ and $\cF$ is a locally free direct summand of $\Lie X$ of rank $n-1$ as an $\Oo_S$-module such that $\Oo_F$ acts on $\Lie X/\cF$ by the structural morphism and acts on $\cF$ by the Galois conjugate of the structural morphism. An isomorphism between two such quintuples $(X,\iota,\lambda,\rho,\cF)$ and $(X',\iota',\lambda',\rho',\cF')$ is an isomorphism $\alpha:(X,\iota,\lambda,\rho)\rightarrow (X',\iota',\lambda',\rho')$ in $\cN_{n,\epsilon}^\Pap(S)$ such that $\alpha^*(\cF')=\cF$. Again we drop the subscript $\epsilon$ in $\cN_{n,\epsilon}^\Kra$ when we do not emphasize on the sign.
	\end{definition}

	By \cite{Kr} (see also \cite[Proposition 2.7]{Shi2}), the natural forgetful functor $\Phi:\cN_n^\Kra\rightarrow \cN_n^\Pap$ forgetting $\cF$ is the blow up of $\cN_n^\Pap$ along its singular locus Sing.  For each point $\Lambda\in \mathrm{Sing}$, its inverse image $\Phi^{-1}(\Lambda)$ is an exceptional divisor $\Exc_\Lambda$ isomorphic to $\bP^{n-1}_k$.

	\subsection{Special cycles}\label{subsec:specialcycles}
	\begin{definition}
		For an $\Oo_F$-lattice $L$ of $\bV$, define $\cZ^\Pap(L)$ to be the subfunctor of $\cN_n^\Pap$ sending each $S\in \Nilp \Oo_{\breve F}$ to the isomorphism classes of tuples $(X,\iota,\lambda,\rho)\in \cN_n^\Pap(S)$ such that for any $x\in L$ the quasi-homomorphism
		\[\rho^{-1}\circ x\circ \rho_\cG: \cG\times_{S} \bar{S} \rightarrow X\times_{S} \bar{S}\]
		extends to a homomorphism $\cG_S\rightarrow X$. For $\bx\in \bV^m$, we let $\cZ^\Pap(\bx)\coloneqq \cZ^\Pap (L)$ where $L=\spa \{\bx\}$. Let
		\[\cZ^\Kra(\bx)=\cZ^\Kra(L)\coloneqq \cZ^\Pap(L)\times_{\cN_n^\Pap} \cN_n^\Kra.\]
		By Grothendieck-Messing theory $\cZ^\Pap(L)$ (hence $\cZ^\Kra(L)$) is a closed formal subscheme of $\cN_n^\Pap$ (respectively $\cN_n^\Kra$). We sometimes add the subscript $ _{n,\epsilon}$ to $\cZ^\Pap(L)$, $\cZ^\Pap(\bx)$, $\cZ^\Kra(L)$ and $\cZ^\Kra(\bx)$ to indicate their ambient moduli spaces.
	\end{definition}

	\begin{definition}\label{def:stricttransformdivisor}
		For an $\Oo_F$-lattice $L\subset \bV$, define $\tZ(L)$ to be the strict transform (see the definition after \cite[Chapter II, Corrollary 7.15]{hartshorne2013algebraic}) of $\cZ^\Pap(L)$ under the blow up $\cN_n^\Kra\rightarrow \cN_n^\Pap$.
	\end{definition}

	\begin{proposition}\label{prop:cZxunimodular}
		Suppose $\chi(\bV)=\epsilon$. Let $L$ be a self-dual lattice of rank $m$ in $\bV$ with $\eta=\chi(L)$.  We have
		\[\cZ_{n,\epsilon}^\Pap(L)\cong \cN_{n-m,\epsilon \eta}^\Pap,\hbox{ and } \tZ_{n,\epsilon}(L)\cong \cN_{n-m,\epsilon \eta}^\Kra.\]
	\end{proposition}
	\begin{proof}
		Let us start with the case $L=\spa\{\bx_0\}$ where $\bx_0\in \bV$. Assume that $u=h(\bx_0,\bx_0)$.
		Multiplying the Hermitian form $(\, ,)_\bX$ on $C$ by $u^{-1}$ does not affect the various moduli spaces involved. So we can perform this and assume that $h(\bx_0,\bx_0)=1$. Moreover, the sign of its orthogonal complement in $\bV$ becomes
		\[\epsilon_1=\epsilon \cdot \chi(u^{-1})\cdot \chi(u^{-(n-1)})\cdot\chi(-1)^{n-1}=\epsilon\chi(u)^{n}\chi(-1)^{n-1}.\]
		Then for $(X,\iota,\lambda,\rho)\in \cZ^\Pap_{n,\epsilon}(\bx_0)(S)$, we define
		\[\bx_0^*\coloneqq \lambda_\cG^{-1}\circ \bx_0^\vee\circ \lambda, \ e\coloneqq \bx_0\circ \bx_0^* \in \End(X).\]
		By the fact that $h(\bx_0,\bx_0)=1$ we know that $e$ is an idempotent. It is routine to check that
		\[((1-e)X,(1-e)\iota, (1-e^\vee)\lambda (1-e),\rho (1-e))\]
		is an object in $\cN_{n-1,\epsilon_1}^\Pap(S)$. Conversely given $(Y,\iota_Y,\lambda_Y,\rho_Y)\in \cN_{n-1,\epsilon_1}^\Pap(S)$, the object
		\[(Y\times \cG_S, \iota_Y\times \iota_{\cG_S}, \lambda_Y\times \lambda_{\cG_S},g\circ(\rho_Y\times \rho_{\cG_S}))\]
		is in $\cZ_{n,\epsilon}^\Pap(\bx_0)(S)$ where $g\in \rU(\bV)$ such that $g^{-1}\bx_0$ is the inclusion $0\times \mathrm{id}:\bY \rightarrow \bX_{n-1} \times \bY$ where $\bX_{n-1}$ is the framing object of $\cN_{n-1,\epsilon_1}^\Pap$.
      The above two functors are inverse to each other. This shows that $\cZ_{n,\epsilon}^\Pap(\bx_0)\cong \cN_{n-1,\epsilon_1}^\Pap$. For general $L$ of rank $m$ and determinant $u$, find a basis with Gram matrix $\{1,\ldots,1,u\}$ and apply the above result repeatedly. So we have $\cZ_{n,\epsilon}^\Pap(L)\cong \cN_{n-m,\epsilon_m}^\Pap$ where
		\[\epsilon_m=\epsilon \chi(u)^{n-m+1}\chi(-1)^{(n-m)m+\frac{m(m-1)}{2}}.\]
		Notice that by scaling the Hermitian form by $(-1)^m u$ again we have $\cN_{n-m,\epsilon_m}^\Pap=\cN_{n-m,\epsilon \eta}^\Pap$.
		It then follows from \cite[Chapter II, Corollary 7.15]{hartshorne2013algebraic}  that $\tZ_{n,\epsilon}(L)$ is the blow up of $\cZ_{n,\epsilon}^\Pap(L)$ along its superspecial points, which is $\cN_{n-m,\epsilon \eta}^\Kra$.
	\end{proof}
	
	\begin{corollary}\label{cor:reducedimensionofcZbyone}
		Let $L$ be as in Proposition \ref{prop:cZxunimodular} and $\by\in \bV$ such that $\by\bot L$. Then
		\[\cZ_{n,\epsilon}^\Kra(\by)\cap\tZ_{n,\epsilon}(L)\cong\cZ_{n-m,\epsilon \eta}^\Kra(\by).\]
	\end{corollary}
	
	\begin{remark}\label{rmk:tZLandtZxnotthesame}
		It follows directly from the definition that $\tZ(L)$ is a closed  formal subscheme of $\tZ(\bx_1)\cap\cdots \cap \tZ(\bx_r)$ if $\{\bx_1,\ldots,\bx_r\}$ is a basis of $L$. However, in general, these two cannot be identified.
	\end{remark}
	
	\subsection{Bruhat-Tits stratification}\label{subsec:bruhat tits in general}
 For an $\Oo_{\breve F}$-lattice $M$ of $N$, define $M^\sharp$ to be the dual lattice of $M$ with respect to the form $(\, ,)_\bX$.
	Recall the following results.
	\begin{proposition} (\cite[Proposition 2.2 and 2.4]{RTW}) \label{prop:k points of N}  Let
		$\cN(k)$ be the set of $\Oo_{\breve F}$-lattices
		\begin{align*}
			\cN(k)=\{M\subset C\otimes_F \breve{F}\mid M^{\sharp}=M,\,  \pi \tau(M) \subset M \subset \pi^{-1} \tau(M),\,
	 \dim_{k} (M + \tau(M))/M \le 1 \}.
		\end{align*}
		Then
		the map
		\[
		\cN^\Pap(k) \rightarrow \mathcal \cN(k), \quad x=(X, \iota, \lambda, \rho) \mapsto M(x)= \rho (M(X))\subset N
		\]
		is a bijection.
	\end{proposition}

	We say a lattice $\Lambda\subset C$ is a vertex lattice if $\pi \Lambda\subseteq \Lambda^\sharp \subseteq \Lambda$ where $\Lambda^\sharp$ is dual lattice of $\Lambda$ with respect to $(\, ,)_\bX$, and we call $t=\mathrm{dim}_{\F_\q}(\Lambda/\Lambda^\sharp)$ the type of $\Lambda$. We denote the set of vertex lattices (respectively, of type $t$) by $\cV$ (respectively, $\cV^t$). We say two vertex lattice $\Lambda_1$ and $\Lambda_2$ are neighbours if $\Lambda_1\subset \Lambda_2$ or $\Lambda_2\subset \Lambda_1$.
	Then we can define a simplicial complex $\cL$ as follows.
	When $n$ is odd or when $n$ is even and $C$ is non-split, then an $r$-simplex is formed by $\Lambda_0,\ldots,\Lambda_r$ if any two members of this set are neighbours. When $n$ is even and $C$ is split, we refer to discussion before \cite[3.4]{RTW} for the definition of $\cL$. We also use $\cL_{n,\epsilon}$ to denote $\cL$ if $C$ has dimension $n$ and $\chi(C)=\epsilon$. Again when $n$ is odd, $\cL_{n,1}=\cL_{n,-1}$, hence we use $\cL_n$ to denote it.
	
	By results in Sections 4 and 6 of loc. cit., to each $\Lambda\in\cV^t$ we can associate a Deligne-Lusztig varieties $\cN_\Lambda$ and $\cN_\Lambda^\circ$ of dimension $t/2$, such that
	\begin{align*}
		\cN_{\Lambda}(k)=\{M\in \cN(k)\mid M \subset \Lambda \otimes_{\Oo_F}\Oo_{\breve{F}}\},
	\end{align*}
	and
	\begin{align*}
		\cN_{\Lambda}^{\circ}(k)=\{M\in \cN(k)\mid \Lambda(M)=\Lambda\}.
	\end{align*}
	Here $\Lambda(M)$ is the minimal vertex lattice such that $\Lambda(M)\otimes_{\Oo_F} \Oo_{\breve F}$ contains $M$ which always exists by \cite[Proposition 4.1]{RTW}.
	By Theorem 1.1 of loc.cit., we know that
	\begin{align*}
		\cN_{\Lambda}\coloneqq \bigsqcup_{\Lambda'\in\cL,\Lambda'\subseteq \Lambda}\cN_{\Lambda'}^\circ,
	\end{align*}
	and
	\[\cN^\Pap_{\red}=\bigsqcup_{\Lambda\in\cL}\cN_\Lambda^\circ\]
	where each $\cN_\Lambda$ is a closed subvariety of $\cN^\Pap_{\red}$.
	By loc. cit., we also know that
	\[\cN_\Lambda\cap\cN_{\Lambda'}=\begin{cases}
		\cN_{\Lambda\cap\Lambda'} & \text{ if } \Lambda\cap\Lambda'\in \cV,\\
		\emptyset  & \text{otherwise}.
	\end{cases}\]
	For a lattice $L\subset \bV$, define
	\begin{equation}\label{eq:cVL}
		\cV(L)\coloneqq \{\Lambda\in \cV\mid L\subseteq \Lambda^\sharp\}, \text{ and }\cV^t(L)\coloneqq \{\Lambda\in \cV^t\mid L\subseteq \Lambda^\sharp\}.
	\end{equation}
	When $L=\spa \{\bx\}$ we also denote $\cV(L)$ (respectively, $\cV^t(L)$) by $\cV(\bx)$ (respectively, $\cV^t(\bx)$).
	For any subset $S$ of $\cV$, we define $\cL(S)$ to be the subcomplex of $\cL$ such that a simplex is in $\cL(S)$ if and only if every vertex in it is in $S$.  For a lattice $L$ of $\cV$ and $\bx\in C$, define
	\begin{equation}\label{eq:cLL}
		\cL(L)=\cL(\cV(L)).
	\end{equation}
	When $L=\spa \{\bx\}$ we also denote $\cL(L)$ by $\cL(\bx)$.
	\subsection{Horizontal and vertical part}
	A formal scheme $X$ over $\Spf \Oo_{\breve F}$ is called horizontal (respectively, vertical) if it is flat over $\SpfOF$ (respectively, $\pi$ is locally nilpotent on $\Oo_X$). For a formal scheme $X$ over $\Spf O_{\breve F}$, its horizontal part $X_{h}$ is canonically defined by the ideal sheaf $\Oo_{X,\mathrm{tor}}$ of torsion sections on $\Oo_X$. If $X$ is Noetherian, there exists a $m\in \Z_{>0}$ such that $\pi^m \Oo_{X,\mathrm{tor}}=0$. We define the vertical part $X_{v}\subset X$ to be the closed formal subscheme defined by the ideal sheaf $\pi^m \Oo_X$. Since $ \Oo_{X,\mathrm{tor}}\cap \pi^m \Oo_X=\{0\}$,
	we have the following decomposition by primary decomposition
	\begin{equation}\label{eq:horizontal vertical decomposition of scheme}
		X=X_h\cup X_v
	\end{equation}
	as a union of horizontal and vertical formal subschemes. Notice that the horizontal part  $X_{h}$ is canonically defined whereas the vertical part  $X_{v}$ depends on the choice of $m$.
	
	\begin{lemma}\label{lem:cZNoetherian}
		For a lattice $L^\flat\subset \bV$ of rank greater or equal to $n-1$ with non-degenerate Hermitian form, $\cZ^\Kra(L^\flat)$ is Noetherian.
	\end{lemma}
	\begin{proof}
		The lemma can be proved as in \cite[Lemma 2.9.2]{LZ}.
	\end{proof}
	
	\begin{lemma}\label{lem:cZ_vsupportedoncN_red}
		For a rank $n-1$ lattice $L^\flat\subset \bV$ with non-degenerate Hermitian form, $\cZ^\Kra(L^\flat)_v$ is supported on the reduced locus $\cN^\Kra_{\red}$ of $\cN^\Kra$, i.e., $\Oo_{\cZ^\Kra(L^\flat)_v}$ is annihilated by a power of the ideal sheaf of $\cN^\Kra_{\red}$.
	\end{lemma}
	\begin{proof}
		The proof is the same as that of \cite[Lemma 5.1.1]{LZ}.
	\end{proof}
	
	\subsection{Derived special cycles}
	For a locally Noetherian formal scheme $X$ together with a formal subscheme $Y$, denote by $K_0^Y(X)$ the Grothendieck group of finite complexes of coherent locally free $\Oo_X$-modules acyclic outside $Y$. For such a complex $A^\bullet$, denote by $[A^\bullet]$ the element in $K_0^Y(X)$ represented by it. We use $K_0(X)$
	to denote $K_0^X(X)$. Denote by $\rF^i K_0^Y(X)$ the codimension $i$ filtration on $K_0^Y(X)$ and $\Gr^i K_0^Y(X)$ its $i$-th graded piece.
	We have a cup product $\cdot$ on $K_0^Y(X)$ defined by tensor product of complexes:
	\[[A_1^\bullet]\cdot [A_2^\bullet]=[A_1^\bullet \otimes A_2^\bullet].\]
	When $X$ is a scheme, the cup product satisfies (\cite[Section I.3, Theorem 1.3]{soule1994lectures})
	\begin{equation}\label{eq:cupproductfiltration}
		\rF^i K_0^Y(X)_\Q \cdot \rF^j K_0^Y(X)_\Q\subset \rF^{i+j} K_0^Y(X)_\Q.
	\end{equation}
	It is expected that \eqref{eq:cupproductfiltration} is also true when $X$ is a formal scheme. We will only need special cases of this fact which can be checked directly, see, for example, Lemma \ref{lem:Excselfintersect} and \ref{lem:cZisinF^2}.
	
	Let $K'_0(Y)$  be the Grothendieck group of coherent sheaves of $\Oo_Y$-modules on $Y$. When $X$ is regular we have the following isomorphism
	\begin{equation}\label{eq:K^Y and K'(Y)}
		K^Y_0(X)\cong K'_0(Y).
	\end{equation}
	In particular, $K_0(X)\cong K'_0(X)$.

      When $X$ is a regular scheme of dimension $d$, there is an isomorphism of graded rings defined by the Chern character:
      \[\mathrm{ch}:K_0(X)_\Q\cong  \bigoplus_{i=1}^d \mathrm{CH}^i(X)_\Q.\]
      In particular we have
      \[\Gr^i K_0(X)_\Q\cong \mathrm{CH}^i(X)_\Q. \]
	
	Recall that for $\bx\in \bV$, $\cZ^\Kra(\bx)$ is a divisor,
 see \cite[Proposition 4.3]{Ho2}.
	\begin{definition}
		For $\bx=(\bx_1,\ldots,\bx_r)\in \bV^r$, define ${}^\bL \cZ^\Kra(\bx)$ to be
		\begin{equation}\label{eq:derivedcZ}
			[\Oo_{\cZ^\Kra(\bx_1)}\otimes^\bL\cdots \otimes^\bL \Oo_{\cZ^\Kra(\bx_r)}]\in K_0^{\cZ^\Kra(\bx)} (\cN^\Kra)
		\end{equation}
		where $\otimes^\bL$ is the derived tensor product of complexes of  coherent locally free sheaves on $\cN^\Kra$. By \cite[Theorem B]{Ho2}, ${}^\bL\cZ^\Kra(\bx)$ only depends on $L\coloneqq \spa\{\bx\}$, hence can be denoted as ${}^\bL\cZ^\Kra(L)$.
	\end{definition}
	
	\begin{definition}
		When  $L$ has rank $n$, we  define the intersection number
		\begin{equation}\label{eq:intL}
			\mathrm{Int}(L)=\chi(\cN^\Kra,{}^\bL\cZ^\Kra(L))
		\end{equation}
		where $\chi$ is the Euler characteristic.
	\end{definition}
	
	\begin{lemma}\label{lem:finiteness of Int L}
		$\cZ^\Kra(L)$ is properly supported on $\cN_{\red}^\Kra$. In particular, $\mathrm{Int}(L)$ is finite.
	\end{lemma}
	\begin{proof}
		This can be proved exactly the same way as \cite[Lemma 2.10.1]{LZ}.
	\end{proof}
	
	\subsection{Special difference cycles}
	Conjecture \ref{conj:main} and Theorem \ref{thm:ind formula reducing valuation} motivate us to make the following definition.
	\begin{definition}\label{def:differencecycle}
		For $L\subset \bV$ a rank $\ell$ lattice, define the special difference cycle $\cD(L)\in K_0^{\cZ^\Kra(L)} (\cN^\Kra) $ by
		\begin{equation}\label{eq:differencecycle}
			\cD(L)={}^\bL\cZ^\Kra(L)+\sum_{i=1}^\ell (-1)^i \q^{i(i-1)/2} \sum_{\substack{L\subset L'\subset \frac{1}{\pi} L\\ \mathrm{dim}_{\F_\q}(L'/L)=i}} {}^\bL\cZ^\Kra(L').
		\end{equation}
	\end{definition}
	One interesting observation is the following decomposition of ${}^\bL\cZ^\Kra(L)$.
	\begin{lemma}\label{lem: decomposition of Z(L) as sums of D(L)}
		For $L\subset \bV$ a lattice of rank $\ell$, we have the following identity in $K_0^{\cZ^\Kra(L)} (\cN^\Kra)$ where the summation is finite.
		$$
		{}^\bL\cZ^\Kra(L) =\sum_{\substack{L'\mathrm{ integral} \\L \subset L' \subset L_F}}   \cD(L').
		$$
	\end{lemma}
	\begin{proof}
		First of all, if $L$ is not integral, neither is $L'$ if $L\subset L'$. In this case ${}^\bL\cZ^\Kra(L)=0$ and the summation index on the right hand side of the identity in the lemma is empty. This proves the lemma when $\val(L)<0$. We can now prove the identity by induction on the fundamental invariant of $L$. Assume that the lemma is proved for all $L'\subset L_F$ with $L\subsetneq L'$.
		
		For $L'$ with $L\subsetneq L'\subset \frac{1}{\pi} L$,  we have
		\[{}^\bL\cZ^\Kra(L') =\displaystyle\sum_{L' \subset L'' \subset L'_F}   \cD(L'')\]
		by the induction hypothesis. Combining this with \eqref{eq:differencecycle}, we can write
		\[{}^\bL\cZ^\Kra(L) =\displaystyle\sum_{L \subset L'' \subset L_F} m(L'')  \cD(L'')\]
		where $m(L'')\in \Z$. Now it suffices to show $m(L'')=1$ for  any $L''$ such that $L\subset L'' \subset L_F$.
		
		First, notice that $m(L)=1$. For any $L''$ such that $L\subset L''\subsetneq L_F$, let $M'=\frac{1}{\pi}L\cap L''$ and $m=\dim_{\F_\q} (M'/L)$. We have
		\begin{align}\label{eq: inc exc}
			m(L'')=-\sum_{i=1}^m (-1)^i \q^{i(i-1)/2} \sum_{\substack{L\subset L'\subset M'\\ \mathrm{dim}_{\F_\q}(L'/L)=i}} 1=1
		\end{align}
		by evaluating the identity in the corollary to \cite[Lemma $12$]{tamagawa1963zeta} at $t=1$.
	\end{proof}
	\begin{remark}
		When $\ell=1$ and $L=\mathrm{Span}\{\bx\}$, the Cartier divisor
		\[\cD(L)=\cZ(\bx)-\cZ(\frac{1}{\pi}\bx)\]
		is the difference divisor $\cD(\bx)$ defined in \cite[Definition 2.10]{Terstiege}.
	\end{remark}
	\begin{definition}\label{def:prim int}
		Assume $L=L_1\oplus L_2$, where $L_i$ is of rank $n_i$ and $n_1+n_2=n$. We define
		\begin{align}\label{eq:prim int}
			\Int(L)^{(n_1)}=\chi(\cN^\Kra, \cD(L_1)\cdot {}^{\bL}\cZ^{\Kra}(L_2)).
		\end{align}
		Notice that  $\Int(L)^{(n_1)}$ depends on the decomposition $L=L_1\oplus L_2$.
	\end{definition}

	\section{Special cycles and  exceptional divisors}\label{sec: Exc}
	For a formal subscheme $\cZ$ of $\cN^\Kra$, we use the notation $\otimes_{\cZ}$ (respectively, $\otimes_{\cZ}^\bL$) instead of $\otimes_{\Oo_\cZ}$ (respectively, $\otimes_{\Oo_\cZ}^\bL$). We also simply write $\otimes$ (respectively, $\otimes^\bL$) instead of $\otimes_{\cN^\Kra}$ (respectively, $\otimes_{\cN^\Kra}^\bL$).
	Let us first recall the following distribution law of derived tensor product. In this section, we identify $\bV$ with $C$ by the isomorphism $b$ defined in \eqref{eq: C V isomorphism}.
	\begin{lemma}\label{lem: distribution law of derived tensor}
		Assume that $\mathcal{A}_i$ ($1\leq i \leq k$) is in the derived category of bounded coherent sheaves on $\cN^\Kra$ and $i:\cZ\rightarrow \cN^\Kra$ is a closed embedding of formal subscheme. Then the following identity holds in the derived category of bounded coherent sheaves on $\cZ$.
		\[i^*(\mathcal{A}_1\otimes^\bL\ldots  \otimes^\bL\mathcal{A}_k\otimes^\bL \Oo_{\cZ})
		=i^*(\mathcal{A}_1\otimes^\bL \Oo_{\cZ}) \otimes_\cZ^\bL\ldots \otimes_\cZ^\bL i^*(\mathcal{A}_k\otimes^\bL \Oo_{\cZ}).\]
	\end{lemma}
	\begin{proof}
		We can take locally free representatives of $A_i^\bullet$ of $\mathcal{ A}_i$. Then $A_1^\bullet\otimes\cdots  \otimes A_k^\bullet$ is again a complex of locally free sheaves on $\cN^\Kra$, hence a locally free representatives of $\mathcal{A}_1\otimes^\bL\cdots  \otimes^\bL\mathcal{A}_k$.
		Hence $i^*(\mathcal{A}_1\otimes^\bL\cdots  \otimes^\bL\mathcal{A}_k\otimes^\bL \Oo_{\cZ})$  can be represented by $A_1^\bullet \otimes\cdots \otimes A_k^\bullet \otimes\Oo_\cZ$.
		Meanwhile $A_i^\bullet\otimes\Oo_\cZ$ is a representative of $\mathcal{A}_i\otimes^\bL\Oo_\cZ$ in the derived category of bounded coherent sheaves on $\cN^\Kra$ and is also a complex of locally free sheaves on $\cZ$. Hence $i^*(\mathcal{A}_1\otimes^\bL \Oo_{\cZ}) \otimes_\cZ^\bL\cdots \otimes_\cZ^\bL i^*(\mathcal{A}_k\otimes^\bL \Oo_{\cZ})$ can be represented by $(A_1^\bullet \otimes\Oo_\cZ) \otimes_{\cZ}\cdots \otimes_{\cZ} (A_k^\bullet \otimes\Oo_\cZ)$.
		Now by the distribution law of tensor products we have
		\[A_1^\bullet \otimes\cdots \otimes A_k^\bullet \otimes\Oo_\cZ=(A_1^\bullet \otimes\Oo_\cZ) \otimes_{\cZ}\cdots \otimes_{\cZ} (A_k^\bullet \otimes\Oo_\cZ).\]
		This finishes the proof of the lemma.
	\end{proof}
	\begin{proposition}\label{prop:multiplicityofExc}
		Assume that the dimension of $\bV$ is $n\geq 2$. Then for each $\bx\in \bV$, $\cZ^\Kra(\bx)$ is a divisor. Moreover, we have the following decomposition of Cartier divisors
		\begin{equation}\label{eq:m_Lambda}
			\cZ^\Kra(\bx)=\tZ(\bx)+\sum_{\Lambda \in \cV^0, \bx \in \Lambda} (m_\Lambda(\bx)+1) \Exc_\Lambda
		\end{equation}
		where $m_\Lambda(\bx)$ is the largest integer $m$ such that $\pi^{-m} \cdot \bx\in \Lambda$.
	\end{proposition}
	\begin{proof}
		The fact that $\cZ^\Kra(\bx)$ is a divisor is due to \cite[Proposition 4.3]{Ho2}. By \cite[Proposition 3.7]{Shi1}, the superspecial point corresponding to a type $0$ lattice $\Lambda$ is in $\cZ^\Pap(\bx)$ if and only if $\bx \in \Lambda$. Hence $\Exc_\Lambda\subset \cZ^\Kra(\bx)$  if and only if $\bx \in \Lambda$. Since $\cN_{n,\epsilon}^\Kra$ is regular, we must have a decomposition as in \eqref{eq:m_Lambda} and the only job left is to determine the multiplicity of each $\Exc_\Lambda$.
		
		Fix a type 0 lattice $\Lambda$ and let $m\coloneqq m_\Lambda(\bx)$. Then $\pi^{-m}\cdot\bx$ is a primitive vector in $\Lambda$. By Lemma \ref{lem:unimodularL}, there exists a decomposition
		\[\Lambda=\Lambda_2\obot\Lambda'\]
		where $\Lambda_2$ and $\Lambda'$ are unimodular lattices of rank $2$ and $n-2$ respectively and $\pi^{-m}\cdot\bx\in \Lambda_2$. Let $\eta=\chi(\Lambda')$. By applying Proposition \ref{prop:cZxunimodular}, we see that $\tZ_{n,\epsilon}(\Lambda')\cong \cN_{2,\epsilon \eta}^\Kra$. Moreover we have
		the following proper intersections
		\begin{align*}
			\cZ^\Kra_{n,\epsilon}(\bx) \cap \tZ_{n,\epsilon}(\Lambda')&=\cZ^\Kra_{2,\epsilon \eta}(\bx),\ \tZ_{n,\epsilon}(\bx) \cap \tZ_{n,\epsilon}(\Lambda')=\tZ_{2,\epsilon \eta}(\bx),
		\end{align*}
		and
		\begin{align*}
			\Exc_\Lambda \cap \tZ_{n,\epsilon}(\Lambda')&=\Exc_{\Lambda_2},
		\end{align*}
		where $\Exc_{\Lambda_2}$ is the exceptional divisor in $\cN_{2,\epsilon \eta}^\Kra$ corresponding to the vertex lattice $\Lambda_2$. Hence the multiplicity of $\Exc_{\Lambda}$ in $\cZ^\Kra_{n,\epsilon}(\bx)$ is the same as the multiplicity of $\Exc_{\Lambda_2}$ in $\cZ^\Kra_{2,\epsilon \eta}(\bx)$. Now the proposition follows from \cite[Theorem 4.6]{Shi2} and \cite[Theorem 4.1]{HSY}.
	\end{proof}

	The Chow ring $\mathrm{CH}^\bullet(\Exc_\Lambda)\cong \Gr^\bullet K_0(\Exc_\Lambda)$ is isomorphic to $\Z[H_\Lambda]/(H_\Lambda^{n-1}-1)$ where $H_\Lambda$ is the hyperplane class of $\Exc_\Lambda$ represented by any $\bP^{n-2}_k$ in $\Exc_\Lambda$.
	\begin{proposition}\label{prop:tZintersectExc}
		Assume  $ \dim \bV =n \geq 2$. Assume $\bx\in \bV$ such that $h(\bx,\bx)\neq 0$ and $\Lambda$ is a type $0$ vertex lattice containing $\bx$. Let $m\coloneqq m_\Lambda (\bx)$ as in Proposition \ref{prop:multiplicityofExc}. Then $\tZ(\bx)$ and $\Exc_\Lambda$ intersect properly and
		\[[\Oo_{\tZ(\bx)\cap \Exc_\Lambda}]=(2m+1)H_\Lambda \in \mathrm{CH}^1(\Exc_\Lambda).\]
	\end{proposition}
	\begin{proof}
		First $\tZ(\bx)$ and $\Exc_\Lambda$ are Cartier divisors with no common component, so they intersect properly.
		Let $m=m_\Lambda(\bx)$ and $\bx'\coloneqq \pi^{-m}\cdot\bx$.
		By assumption $m\geq 0$. By Proposition \ref{prop: ind for t}, we have
		\[\{v\in \Lambda\mid h(\bx',v)=0\}=\mathrm{Span}\{\by\}\obot \Lambda'\]
		where $\val(\by)=\val(\bx')$ and $\Lambda'$ is unimodular. Let $\eta=\chi(\Lambda')$ and
		\[\Lambda_2\coloneqq \{v\in \Lambda\mid v\bot \Lambda'\}.\]
		$\Lambda_2$ is rank $2$ unimodular and contains $\bx'$.
		
		By Proposition \ref{prop:cZxunimodular},  we have $\tZ(\Lambda')\cong \cN_{2,\epsilon \eta}^\Kra$. In particular, $\tZ(\Lambda')$ is regular.
		By Corollary \ref{cor:reducedimensionofcZbyone}, we know that $\tZ(\Lambda')\cap\tZ(\bx)=\tZ_{2,\epsilon \eta}(\bx)$. In particular $\tZ(\Lambda')$ and $\tZ(\bx)$ intersect properly as $\tZ_{2,\epsilon \eta}(\bx)$ is a divisor in $\cN^\Kra_{2,\epsilon \eta}$. On the other hand $\tZ(\Lambda')\cap \Exc_\Lambda$ is the exceptional divisor $\Exc_{\Lambda_2}$ in $\cN^\Kra_{2,\epsilon \eta}$.
		Since $\Exc_\Lambda\cong \bP^{n-1}_k$, it is also regular. Our strategy is to compute the intersection number
		\[\chi(\cN^\Kra,\Oo_{\tZ(\bx)}\otimes^{\bL} \Oo_{\Exc_\Lambda}\otimes^{\bL} \Oo_{\tZ(\Lambda')})\]
		in two different ways. By Lemma \ref{lem: distribution law of derived tensor}, one way is
		\begin{equation}\label{eq:firstwaytZandExc}
			\chi(\tZ(\Lambda'), \Oo_{\tZ(\Lambda')\cap \tZ(\bx)}\otimes^{\bL}_{\tZ(\Lambda')} \Oo_{\tZ(\Lambda')\cap \Exc_\Lambda})
		\end{equation}
		where we use the fact that the intersections  $\tZ(\Lambda')\cap \tZ(\bx)$ and $\tZ(\Lambda')\cap \Exc_\Lambda$ are proper (see for example \cite[Lemma B.2]{zhang2021AFL}).
		The other way is, by Lemma \ref{lem: distribution law of derived tensor},
		\begin{equation}\label{eq:secondwaytZandExc}
			\chi(\Exc_\Lambda, \Oo_{\tZ(\bx)\cap \Exc_\Lambda}\otimes^\bL_{\Exc_\Lambda} \Oo_{\tZ(\Lambda')\cap \Exc_\Lambda}).
		\end{equation}
		When $\epsilon \eta=-1$, by Proposition 3.11 and Theorem 4.5 of \cite{Shi2}, we know that \eqref{eq:firstwaytZandExc} is equal to $2m+1$. When $\epsilon \eta=1$, by Lemma 3.10, Theorem 4.1 and Lemma 5.2 of \cite{HSY}, we know that \eqref{eq:firstwaytZandExc} is equal to $2m+1$ as well. Since the intersection number of $H_\Lambda$ with $\Exc_{\Lambda_2}\cong \bP^1_k$ in $\Exc_\Lambda$ is $1$, the proposition follows.
	\end{proof}
	
	\subsection{Intersection numbers involving the exceptional divisors}
	\begin{lemma}\label{lem:Excselfintersect}
		The class of $\underbrace{\Oo_{\Exc_\Lambda}\otimes^\bL \cdots\otimes^\bL\Oo_{\Exc_\Lambda}}_{m}$ in $\mathrm{CH}^{m-1}(\Exc_\Lambda)$ is
		\noindent  $(-2H_\Lambda)^{m-1}$.
	\end{lemma}
	\begin{proof}
		To study this intersection, it suffices to consider the local model $N^\Kra$ constructed in \cite{Kr}. Let $N^\Kra_s$ be its special fiber.
		Recall by equation (4.11) loc. cit., we have
		$$
		N^\Kra_s =\Exc +  Z_2
		$$
		as Cartier divisors where $\Exc$ is the exceptional divisor of $N^\Kra$ and $Z_2$ is a divisor in $N^\Kra$ which intersect properly with $\Exc$. Their intersection is $2H$ where $H$ is the hyperplane class of $\Exc$.
		Since $\Exc$ is properly supported on $N^\Kra$, we have
		\[[\OO_{\Exc}\otimes^\bL \Oo_{N^\Kra_s}]=0.\]
		Hence
		\begin{align*}
			0&=[\Oo_\Exc\otimes^\bL_{N^\Kra} \Oo_{N^\Kra_s}]\\
			&=[\Oo_\Exc\otimes^\bL_{N^\Kra} \Oo_\Exc]+[\Oo_\Exc \otimes^\bL_{N^\Kra} \Oo_{Z_2}]\\
			&= [\Oo_\Exc\otimes^\bL_{N^\Kra} \Oo_\Exc]+2H.
		\end{align*}
		This proves the lemma when $m=2$. The general case now follows from Lemma \ref{lem: distribution law of derived tensor}.
	\end{proof}
	\begin{corollary}\label{cor:cZintersectExc}
		Let $\Lambda\in \cV^0$ and $\bx\in \Lambda$. Then we have the following identity in $\mathrm{CH}^1(\Exc_\Lambda)$:
		\[[\Oo_{\Exc_{\Lambda}}\otimes^\bL \Oo_{\cZ^\Kra(\bx)}]=-H_{\Lambda}.\]
	\end{corollary}
	\begin{proof}
		By Propositions \ref{prop:multiplicityofExc}, \ref{prop:tZintersectExc} and Lemma \ref{lem:Excselfintersect}, we have the following identity in $\mathrm{CH}^1(\Exc_\Lambda)$:
		\[[\Oo_{\cZ^\Kra(\bx)}\otimes^\bL\Oo_{\Exc_\Lambda}]=[(2m_\Lambda(\bx)+1)-2(m_\Lambda(\bx)+1)] H_\Lambda=-H_\Lambda.\]
		This finishes the proof of the corollary.
	\end{proof}
	
	\begin{corollary}\label{cor:multiplecZintersectExc}
		Assume that $n-m\geq 1$ and $\Exc_\Lambda\subset \cZ^\Kra(\bx_1)\cap\ldots\cap \cZ^\Kra(\bx_m)$, then
		\[\chi(\cN_n^\Kra,\Oo_{\cZ^\Kra(\bx_1)}\otimes^{\bL}\ldots\Oo_{\cZ^\Kra(\bx_m)}\otimes^\bL\underbrace{\Oo_{\Exc_\Lambda}\otimes^{\bL} \cdots \otimes^{\bL} \Oo_{\Exc_\Lambda}}_{n-m})=(-1)^{n-1}\cdot 2^{n-m-1}.\]
	\end{corollary}
	\begin{proof}
		By Corollary \ref{cor:cZintersectExc}, Lemmas \ref{lem: distribution law of derived tensor} and \ref{lem:Excselfintersect}, we have
		\begin{align*}
			&\chi(\cN_n^\Kra,\Oo_{\cZ^\Kra(\bx_1)}\otimes^{\bL}\ldots\Oo_{\cZ^\Kra(\bx_m)}\otimes^\bL\underbrace{\Oo_{\Exc_\Lambda}\otimes^{\bL} \cdots \otimes^{\bL} \Oo_{\Exc_\Lambda}}_{n-m})\\
			&=\chi(\Exc_\Lambda,\underbrace{(-H_\Lambda)\otimes_{\Exc_\Lambda}^{\bL} \cdots \otimes_{\Exc_\Lambda}^{\bL} (-H_\Lambda)}_{m}\otimes_{\Exc_\Lambda}^\bL\underbrace{(-2H_\Lambda)\otimes_{\Exc_\Lambda}^{\bL} \cdots \otimes_{\Exc_\Lambda}^{\bL} (-2H_\Lambda)}_{n-m-1})\\
			&=(-1)^m\cdot(-2)^{n-m-1}.
		\end{align*}
	\end{proof}
	
	For  $\Lambda\in \cV^0$, let  $\bP^1_\Lambda$ be any $\bP^1_k$ in $\Exc_\Lambda$, and
	\begin{equation}\label{eq: definition of Int Lambda_0}
		\Int_{\Lambda}(\bx)=\chi(\cN^\Kra, \Oo_{\cZ^\Kra(\bx)}\otimes^\bL\Oo_{\bP_\Lambda^1}).
	\end{equation}
	
	\begin{corollary}\label{cor: P1 dot Exc}
		For $\Lambda\in \cV^0$,   we have
		\begin{equation}\label{eq:ExcintersectP1}
			\chi(\cN^\Kra,\Oo_{\Exc_\Lambda}\otimes^\bL \Oo_{\bP^1_{\Lambda}})=-2.
		\end{equation}
	\end{corollary}
	\begin{proof}
		By Lemma \ref{lem:Excselfintersect}, we have
		\begin{align*}
			&\chi(\cN^\Kra,\Oo_{\Exc_\Lambda}\otimes^\bL \Oo_{\bP^1_{\Lambda}})\\
			&=\chi(\cN^\Kra,\Oo_{\Exc_\Lambda}\otimes^\bL (\Oo_{\Exc_\Lambda} \otimes_{\Oo_{\Exc_\Lambda}} \Oo_{\bP^1_{\Lambda}}))\\
			&=\chi(\Exc_\Lambda,(\Oo_{\Exc_\Lambda}\otimes^\bL \Oo_{\Exc_\Lambda}) \otimes_{\Oo_{\Exc_\Lambda}} \Oo_{\bP^1_{\Lambda}})\\
			&=-2\chi(\Exc_\Lambda,H_\Lambda\cdot [\Oo_{\bP^1_{\Lambda}}])\\
			&=-2.
		\end{align*}
	\end{proof}
	
	\begin{corollary}\label{cor: P1 dot cZ} For $\Lambda\in \cV^0$,   we have
		\[\Int_{\Lambda}(\bx)=-1_{\Lambda}(\bx).\]
	\end{corollary}
	\begin{proof}
		If $\bx\notin \Lambda$, then the intersection number is apparently $0$. Otherwise,  by Corollary \ref{cor:cZintersectExc}, we have
		\begin{align*}
			&\chi(\cN^\Kra,\Oo_{\cZ^\Kra(\bx)}\otimes^\bL \Oo_{\bP^1_{\Lambda}})\\
			&=\chi(\Exc_\Lambda,(\Oo_{\cZ^\Kra(\bx)}\otimes^\bL \Oo_{\Exc_\Lambda}) \otimes_{\Oo_{\Exc_\Lambda}} \Oo_{\bP^1_{\Lambda}})\\
			&=-\chi(\Exc_\Lambda,H_\Lambda\cdot [\Oo_{\bP^1_{\Lambda}}])\\
			&=-1.
		\end{align*}
	\end{proof}

	The above results suggest that the difficulty in computing $\mathrm{Int}(L)$ mainly lies in computing
	\[\chi(\cN^\Kra,\Oo_{\tZ(x_1)}\otimes^{\bL}\cdots \otimes^{\bL}\Oo_{\tZ(x_n)}).\]
	We end this section by studying the intersection number of difference cycle with exceptional divisors.
	\begin{lemma}\label{lem: Z(L hecke) dot Exc}
		If $L^\flat$ has rank $n-1$, then for any $\Lambda\in \cV^0(L^\flat)$, we have
		\[\chi(\cN^\Kra,\cD(L^\flat)\cdot [\Oo_{\Exc_{\Lambda}}])=\begin{cases}
			(-1)^{n-1} & \text{ if } L^\flat=\Lambda\cap L^\flat_F,\\
			0 & \text{ otherwise}.
		\end{cases}\]
	\end{lemma}
	\begin{remark}
		We have $L^\flat=\Lambda\cap L ^\flat_F$ if and only if $L^\flat$ is of type (see \eqref{eq:t(L)} and Lemma \ref{lem:tMflatleq1} below) $1$ or $0$ and $\Lambda$ is at the boundary of the $\cL(L^\flat)$.
	\end{remark}
	\begin{proof}
		Define
		\[M'\coloneqq \frac{1}{\pi}L^\flat \cap \Lambda \text{ and } m\coloneqq \dim_{\F_\q}(M'/L^\flat).\]
		Then for $L'$ such that $L^\flat\subset L'\subset \frac{1}{\pi} L^\flat$, we know that $\cZ^\Kra(L')$ intersects $\Exc_{\Lambda}$ if and only if $L'\subset M'$. For such $L'$, by Corollary \ref{cor:multiplecZintersectExc}, we have
		\begin{equation}
			\chi(\cN^\Kra,{}^\bL\cZ^\Kra(L')\cdot [\Oo_{\Exc_{\Lambda}}])=(-1)^{n-1}.
		\end{equation}
		Hence
		\[\chi(\cN^\Kra,\cD(L^\flat)\cdot [\Oo_{\Exc_{\Lambda}}])=(-1)^{n-1}[1+\sum_{i=1}^m (-1)^i \q^{i(i-1)/2} \sum_{\substack{L^\flat\subset L'\subset M'\\ \mathrm{dim}_{\F_\q}(L'/L^\flat)=i}} 1].\]
		Notice that $m=0$ if and only if $M'=L^\flat$ which is equivalent to the condition $L^\flat=\Lambda\cap L^\flat_F$. In this case the summation in \eqref{eq:multiplicityofMflatinD_h} is over an empty set hence \eqref{eq:multiplicityofMflatinD_h} is equal to $1$.
		If $m>0$ we know \eqref{eq:multiplicityofMflatinD_h} is equal to $0$ by \eqref{eq: inc exc}.
	\end{proof}

	\section{Horizontal components of special cycles}\label{sec: horizontal part}
	Given an integral Hermitian lattice $L$ we can have its Jordan decomposition:
	\begin{equation}\label{eq:Jordandecomposition}
		L=\obot_{t\geq 0} L_t
	\end{equation}
	where $L_t$ is $\pi^t$-modular, see \cite{J}. Define the type of $L$ to be
	\begin{equation}\label{eq:t(L)}
		t(L)=\sum_{t\geq 1} \mathrm{rank}_{\Oo_F}(L_t).
	\end{equation}

	\subsection{Quasi-canonical lifting cycles}\label{subsec:quasi canonical lifting}
	Assume that $\mathrm{dim}(\bV)=2$.
	When $\chi(\bV)=-1$, for $\by\in \bV$, by \cite[Theorem 4.5]{Shi2}, we have the following equality of Cartier divisors on $\cN_{2,-1}^\Kra$.
	\[\tZ_{2,-1}(\by)=\cZ_0+\sum_{s=1}^{\val(\by)} (\cZ_{s}^+ +\cZ_{s}^-).\]
    Here $\cZ_0$ (respectively, $\cZ_{s}^{\pm}$) is a canonical (respectively, quasi-canonical) lifting cycle of level $0$ (respectively, $s$), see \cite[\S 3]{Shi2}. Moreover by \cite[Proposition 3.12]{Shi2}, $\cZ_s^+$ and $\cZ_s^-$ do not intersect when $s\geq 1$.
    Let $\Oo_s:=\Oo_{F_0}+\Oo_{F}\cdot \pi_0^s$ and $M_s$ be the finite abelian extension of ${\breve F}$ corresponding to the subgroup $\Oo_s^\times$ under local class field theory. Let $W_s$ be the integral closure of $\Oo_{\breve F}$ in $M_s$. Then we have $\cZ_0\cong \SpfOF$ and $\cZ_{s}^{\pm}\cong \Spf W_s$. Define the primitive part of $\tZ_{2,-1}(\by)$ to be
	\[\tZ_{2,-1}(\by)^\circ\coloneqq \left\{\begin{array}{cc}
		\cZ_{\val(\by)}^+ +\cZ_{\val(\by)}^-   & \text{ if } \val(\by)>0,  \\
		\cZ_0 & \text{ if } \val(\by)=0.
	\end{array}\right.\]
	When $\chi(\bV)=1$, for $\by\in \bV$ such that $\val(\by)\geq 0$, by \cite[Theorem 4.1]{HSY}, we have  the following equality of Cartier divisors on $\cN_{2,1}^\Kra$.
	\[\tZ_{2,1}(\by)=\cZ_0+\cZ_v(\by),\]
    where $\cZ_0\cong \SpfOF$ is a canonical lifting cycle and $\cZ_v(\by)$ is a Cartier divisor whose structure sheaf is annihilated by $\pi^N$ for some $N> 0$. Define the primitive horizontal part of $\tZ_{2,1}(\by)$ to be
    \[\tZ_{2,1}(\by)^\circ\coloneqq \left\{\begin{array}{cc}
		0   & \text{ if } \val(\by)>0,  \\
		\cZ_0 & \text{ if } \val(\by)=0.
	\end{array}\right.\]

	\subsection{Horizontal cycles}
\begin{definition}\label{def:horizontal lattice}
Let $M^\flat$ be a rank $n-1$ integral lattice in $\bV$. We say that  $M^\flat$ is horizontal if one of the following conditions is satisfied
\begin{enumerate}
    \item $M^\flat$ is unimodular.
    \item $M^\flat$ is of the form $M^\flat=M\obot \spa\{\by\}$ where $M$ is a unimodular sublattice of rank $n-2$ such that $(M_F)^\bot$ (the perpendicular complement of $M_F$ in $\bV$) is non-split.
\end{enumerate}
Notice that condition (2) is independent of the choice of $M$. We denote the set of horizontal lattices by $\Hor$.
\end{definition}
     For a rank $n-1$ integral lattice $L^\flat$, define
	\begin{equation}\label{eq:horizontalmodules}
		\Hor(L^\flat)\coloneqq \{M^\flat \in \Hor \mid L^\flat\subseteq M^\flat\}.
	\end{equation}
	Let $M^\flat \subset \bV$ be a lattice of rank $n-1$ and type $1$ or $0$. We can decompose $M^\flat$ as
    \begin{equation}\label{eq:Mflat decom}
       M^\flat= M\obot \spa\{\by\},
    \end{equation}
     for some unimodular lattice $M$ of rank $n-2$. Then Proposition \ref{prop:cZxunimodular} and its corollary imply that
	\[\tZ(M^\flat)\cong \tZ_{2,\chi((M_F)^\bot)}(\by).\]
    Under this isomorphism, define $\tZ(M^\flat)^\circ$ to be the formal subscheme of $\tZ(M^\flat)$ isomorphic to $\tZ_{2,\chi((M_F)^\bot)}(\by)^\circ$. By the discussion in \S \ref{subsec:quasi canonical lifting}, $\tZ(M^\flat)^\circ$ is nonempty if and only if $M^\flat\in \Hor$, in which case it consists of the union of irreducible components of $\tZ(M^\flat)$ isomorphic to $\Spf W_s$. In particular, $\tZ(M^\flat)^\circ$ is independent of the choice of $M$.
	\begin{theorem}\label{thm:horizontalpart}
		Let $L^\flat$ be a rank $n-1$ integral lattice in $\bV$, then
		\begin{equation}\label{eq:horizontalpart}
			\cZ^\Kra(L^\flat)_h=\bigcup_{M^\flat \in \Hor(L^\flat)} \tZ(M^\flat)^\circ.
		\end{equation}
		In particular, $\cZ^{\Kra}(L^\flat)_h$ is  of pure dimension $1$. Moreover we have the following identity in $\Gr^{n-1}K_0(\cN^\Kra)$:
		\[[\Oo_{\cZ^\Kra(L^\flat)_h}]=\sum_{M^\flat \in \Hor(L^\flat)} [\Oo_{\tZ(M^\flat)^\circ}].\]
	\end{theorem}
	\begin{proof} The proof largely follows \cite[Section 4.4]{LZ}. Let $K$ be a finite extension of $\breve F$. Assume that $z$ is  an irreducible component of $ \cZ^\Kra(L^\flat)(\Oo_K)=\cZ^\Pap(L^\flat)(\Oo_K)$, and let $G$ be the corresponding formal $\Oo_F$-module over $\Oo_K$.
		Define
		\[L\coloneqq \Hom_{\Oo_F}(T_p\cG,T_p G)\]
		where $\cG$ is the canonical lifting and $T_p$ is the integral $p$-adic Tate module. Here $L$  is  an $\Oo_F$-module of rank $n$ equipped with the Hermitian form
		\[\{x,y\}=\lambda_\cG^\vee \circ y^\vee \circ \lambda_G\circ x,\]
		under which it is self-dual.
		We have two inclusions (preserving Hermitian forms)
		\[i_K:\Hom_{\Oo_F}(\cG,G)_F \rightarrow L_F,\]
		and
		\[i_k:\Hom_{\Oo_F}(\cG,G)_F \rightarrow \bV.\]
		By Lemma 4.4.1 of loc.cit., we have
		\begin{equation}
			\Hom_{\Oo_F}(\cG,G)=i^{-1}_K(L).
		\end{equation}
		Let
		\[M^\flat\coloneqq (L^\flat_F)\cap i_{k}(i_K^{-1}(L)) \cong \Hom_{\Oo_F}(\cG,G).\]
		Then  $z \subset \cZ(M^\flat)(\Oo_K)$.
		Lemma \ref{lem:tMflatleq1} below implies that $t(M^\flat)\leq 1$. Hence we know that $z$ is one of the irreducible component of $\tZ(M^\flat)^\circ\cong \tZ_{2,\chi((M_F)^\bot)}(\by)$ assuming the decomposition of $M^\flat$ as in \eqref{eq:Mflat decom}. The non-emptiness of $\tZ(M^\flat)^\circ$ implies that $M^\flat\in \Hor$.
        It remains to prove that $z$ has multiplicity $1$ in $\cZ^\Kra(L^\flat)$. Consider $R$-points of both sides of \eqref{eq:horizontalpart}, where $R\coloneqq \Oo_K[x]/(x^2)$. As in \cite{Kr} (see \cite[Appendix of Chapter 3]{RZ}) we know
		\[\mathbb{D}(\cG)(R)\cong \Oo_F\otimes_{\Oo_{F_0}}R, \text{ and }  \mathbb{D}(G)(R)\cong (\Oo_F\otimes_{\Oo_{F_0}}R)^n \]
		where $\mathbb{D}$ is the $\Oo_{F_0}$-relative Dieudonn\'e crystal. Define
		\[\tilde{e}_0=1\otimes 1\in \mathbb{D}(\cG)(R),\quad  \tilde{f}_0=\pi\otimes 1 \in \mathbb{D}(\cG)(R).\]
		Then the Hodge submodule $\cF_0$ of $\mathbb{D}(\cG)(R)$ is spanned by
		\[(1\otimes \pi )\tilde{e}_0+\tilde{f}_0.\]
		Here $\mathbb{D}(G)(R)$ is equipped with an $\Oo_F$-invariant symplectic form $\langle,\rangle$ and we can assume that  $\mathbb{D}(G)(R)$ has  a basis $\{\tilde{e}_1,\ldots,\tilde{e}_n,\tilde{f}_1,\ldots,\tilde{f}_n\}$ such that
		\[(\pi\otimes 1)\tilde{e}_i=\tilde{f}_i, \quad  \langle \tilde{e}_i,\tilde{f}_j \rangle=\delta_{ij}.\]
		Since any element in $L^\flat$ is $\Oo_F$-linear, we can arrange a change of basis if necessary and assume that
		\begin{align*}
			&L^\flat((1\otimes \pi )\tilde{e}_0+\tilde{f}_0)=\spa_{R}\{(1\otimes\pi^{a_1})((1\otimes \pi )\tilde{e}_1+\tilde{f}_1), \ldots,(1\otimes\pi^{a_{n-1}})((1\otimes \pi )\tilde{e}_{n-1}+\tilde{f}_{n-1})\}.
		\end{align*}
		Now $\mathbb{D}(G)(\Oo_K)=\mathbb{D}(G)(R)\otimes_R \Oo_K$. Let $e_i=\tilde{e}_i\otimes 1$ and $f_i=\tilde{f}_i\otimes 1$ respectively.
		There is an exact sequence of free $\Oo_{F}\otimes_{\Oo_{F_0}}\Oo_K$-modules (the Hodge filtration)
		\[0\rightarrow \mathrm{Fil}\rightarrow \mathbb{D}(G)(\Oo_K)\rightarrow \Lie G\rightarrow 0\]
		where $\mathrm{Fil}$ is isotropic with respect to $\langle,\rangle$. We must have $L^\flat((1\otimes \pi )e_0+f_0)\subset \mathrm{Fil}$. Hence we have
		\[(1\otimes \pi )e_1+f_1, \ldots,(1\otimes \pi )e_{n-1}+f_{n-1}\subset \mathrm{Fil}.\]
		Since $\mathrm{Fil}$ is isotropic and by the signature condition, we have
		\[\mathrm{Fil}=\spa_{\Oo_K}\{(1\otimes \pi )e_1+f_1, \ldots,(1\otimes \pi )e_{n-1}+f_{n-1},(1\otimes \pi )e_n-f_n\}.\]
		Since $(x)\subset R$ has a nilpotent p.d. structure, by Grothendieck-Messing theory, a lift $\tilde{z}$ of $z$ to $\cZ^\Kra(L^\flat)(R)$ corresponds to a lift of $\mathrm{Fil}$ to an isotropic $\Oo_F\otimes_{\Oo_{F_0}}R$-module $\widetilde{\mathrm{Fil}}$ in $\mathbb{D}(G)(R)$ containing the image of $L^\flat$. By the same reasoning as above, we must have
		\[\widetilde{\mathrm{Fil}}=\spa_{R}\{(1\otimes \pi )\tilde{e}_1+\tilde{f}_1, \ldots,(1\otimes \pi )\tilde{e}_{n-1}+\tilde{f}_{n-1},(1\otimes \pi )\tilde{e}_n-\tilde{f}_n\}.\]
		Hence such lift is unique. This implies that the multiplicity of $z$ in $\cZ^\Kra(L^\flat)$ is one.
	\end{proof}
	\begin{lemma}\label{lem:tMflatleq1}
		Let $L$ be a self-dual Hermitian lattice of rank $n$ and $W$ be a $n-1$ dimensional subspace of $L_F $. Then $t(M^\flat)\le 1$ for $M^\flat=L\cap W$.
	\end{lemma}
	\begin{proof}
		This is exactly the same as the proof of \cite[Lemma 4.5.1]{LZ}. Notice that in our case we may need some blocks $\begin{pmatrix}
		    0 & \pi^a \\
			(-\pi)^a & 0
		\end{pmatrix}$ in the upper left $(n-1)\times (n-1)$ block of $T$ as in loc.cit. Alternatively, see \cite[Lemma 2.24(2)]{LL2}.
	\end{proof}
	
	We end this subsection with the following lemma.
	\begin{lemma}\label{lem:quasicanonicalliftingintersectExc}
		Assume $M^\flat\in \Hor$. Then $\tZ(M^\flat)^\circ$ intersects the special fiber of $\cN_{n,\epsilon}^\Kra$ at a unique $\Exc_\Lambda$ for some  $\Lambda \in \cV^0$. Moreover \[\chi(\cN^\Kra,\Oo_{\tZ(M^\flat)^\circ}\otimes^{\bL} \Oo_{\Exc_\Lambda})=\begin{cases}
		    1 & \text{ if } M^\flat \text{ is unimodular}, \\
            2 &  \text{ otherwise}.
		\end{cases} \]
	\end{lemma}
	\begin{proof}
		By the definition of $\Hor$, we can find a decomposition of $M^\flat$
		\[M^\flat=M\obot \{\bx\}\]
		such that $M$ is self-dual. Let $\Lambda$ be any vertex lattice containing $M^\flat$.
If $M^\flat$ is unimodular, then $\Lambda$ has to be of the form $M^\flat\obot L'$ where $L'$ is the unique unimodular lattice in $(M^\flat_F)^\bot$. If $M^\flat$ is of the form $M\obot L'$ such that $M$ is of rank $n-2$ and  $(M_F)^\bot$ is non-split, then the proof of \cite[Theorem 3.10]{Shi1} implies that there is a unique vertex lattice $\Lambda'$ in $(M_F)^\bot$ which is of unimodular (this fact is the same as the fact that the Bruhat-Tits building of $(M_F)^\bot$ has only one point). Then $\Lambda$ must be of the form $M\obot\Lambda'$. In both cases, $\Lambda$ is unique and is of type $0$.

        Assume $\chi(M)=\eta$. By Proposition \ref{prop:cZxunimodular}, $\tZ(M)\cong \cN^\Kra_{2,\epsilon\eta}$. Moreover $\tZ(M)\cap \Exc_{\Lambda}=\bP^1_k$ is an exceptional divisor in $\cN^\Kra_{2,\epsilon\eta}$. Hence by Lemma \ref{lem: distribution law of derived tensor}, we have
		\[\chi(\cN_{n,\epsilon}^\Kra,\Oo_{\tZ(M^\flat)^\circ}\otimes^{\bL} \Oo_{\Exc_\Lambda})= \chi(\cN_{2,\epsilon\eta}^\Kra,\Oo_{\tZ(M^\flat)^\circ}\otimes_{\cN^\Kra_{2,\epsilon\eta}}^{\bL} \Oo_{\bP^1_k}).\]
		Now the lemma follows from \cite[Lemma 5.2]{HSY} when $\epsilon\eta=1$, and from \cite[Proposition 3.11]{Shi2} when $\epsilon\eta=-1$.
	\end{proof}
	
	\subsection{The horizontal part of special difference cycles}
	Definition \ref{def:differencecycle} motivates us  to make the following definition.
	\begin{definition}
		When $L^\flat$ is a rank $n-1$ integral lattice, define $\cD(L^\flat)_h\in \Gr^{n-1} K_0(\cN^\Kra)$ by
		\begin{equation}
			\cD(L^\flat)_h=[\Oo_{\cZ^\Kra(L^\flat)_h}]+\sum_{i=1}^{n-1} (-1)^i \q^{i(i-1)/2} \sum_{\substack{L^\flat\subset L'\subset \frac{1}{\pi} L^\flat\\ \mathrm{dim}_{\F_\q}(L'/L^\flat)=i}} [\Oo_{\cZ^\Kra(L')_h}].
		\end{equation}
	\end{definition}
	\begin{proposition}\label{prop:horizontalpartofDL}
		Assume $L^\flat$ is a rank $n-1$ integral lattice, then
		\[\cD(L^\flat)_h=\begin{cases}
			\tZ(L^\flat)^\circ  & \text{ if } L^\flat\in \Hor,\\
			0    & \text{ if } L^\flat\notin \Hor.
		\end{cases}\]
	\end{proposition}
	\begin{proof}
		By Theorem \ref{thm:horizontalpart}, it suffices to compute the multiplicity of an irreducible component in $\tZ(M^\flat)^\circ$ in $\cD(L)_h$ for all $M^\flat\in \Hor(L^\flat)$ (see \eqref{eq:horizontalmodules}). For such a $M^\flat$, define
		\[M'\coloneqq \frac{1}{\pi}L^\flat \cap M^\flat \text{ and } m\coloneqq \mathrm{dim}_{\F_\q}(M'/L^\flat).\]
		Then for a lattice  $L'$ with $L^\flat\subset L'\subset \frac{1}{\pi} L^\flat$, we know that $\tZ(M^\flat)^\circ$ is in $\cZ^\Kra(L')_h$ if and only if $L'\subset M'$. Hence the multiplicity of an irreducible components in $\tZ(M^\flat)^\circ$ in $\cD(L)_h$ is
		\begin{equation}\label{eq:multiplicityofMflatinD_h}
			1+\sum_{i=1}^m (-1)^i \q^{i(i-1)/2} \sum_{\substack{L^\flat\subset L'\subset M'\\ \mathrm{dim}_{\F_\q}(L'/L^\flat)=i}} 1.
		\end{equation}
		Notice that $m=0$ if and only if $M'=M^\flat=L^\flat$, in this case the summation in \eqref{eq:multiplicityofMflatinD_h} is over an empty set, hence \eqref{eq:multiplicityofMflatinD_h} is equal to $1$.
		If $m>0$,  \eqref{eq:multiplicityofMflatinD_h} is equal to $0$ by \eqref{eq: inc exc}.
	\end{proof}

	\part{The analytic side}\label{part:II}
	
	\section{Induction formula and primitive local density}\label{sec:ind formula}
	
	In this section, we study various induction formulas of local density polynomials.
	Let $M$ be a  Hermitian  $\Oo_F$-lattice of rank $m$ with $\mathrm{v}(M)\coloneqq\mathrm{min}\{\mathrm{v}_\pi( h(v,v'))\mid v,v'\in M\} \ge -1$.  and let $M^{[k]}= \cH^k\obot M$ for an integer $k \ge 0$.  Let $L$ be a Hermitian  $\Oo_F$-lattice of rank $n$.

	There is a polynomial $\alpha(M, L, X)$ of $X$---the local density polynomial---such that
	\begin{equation}\label{eq:definition of local density polynomial}
		\alpha(M, L,  q^{-2k}) =\int_{\Herm_n(F)} \int_{(M^{[k]})^n} \psi(\langle Y,T(\bx)-T\rangle) d\bx \, dY,
	\end{equation}
	where $T(\bx)$ is the moment matrix of $\bx$, $d\bx$ is the Haar measure on $(M^{[k]})^n$ with total volume $1$, $dY$ is the Haar measures on $\Herm_n(F)$ such that $\Herm_n(\Oo_{F})$ has total volume $1$,  and $\psi$ is an additive character of $F_0$ with conductor $\Oo_{F_0}$.  Finally, we define  $\langle X,Y\rangle =\mathrm{Tr}(XY)$ on $\Herm_n$. We also use the notation $\alpha(M, L)=\alpha(M, L, 1)$ and
	\begin{equation}
		\alpha'(M, L) =-\frac{\partial}{\partial X} \alpha(M,L, X)|_{X=1}.
	\end{equation}

There is another way to define $\alpha(M,L,X)$ as follows.  We use $\mathrm{Herm}_{L,M}$ to denote the scheme of Hermitian $\Oo_F$-module homomorphisms from $L$ to $M$, which is a scheme of finite type over $\Oo_{F_0}$.   More specifically, for an $\Oo_{F_0}$-algebra $R$, we define
\begin{align*}
L_{R}\coloneqq L\otimes_{\Oo_{F_0}}R, \quad  (x\otimes a,y\otimes b)_{R}\coloneqq \pi (x ,y )\otimes_{\Oo_{F_0}} ab\in  \Oo_{F}\otimes_{\Oo_{F_0}} R \text{ where }x,y \in L, a,b\in R.
\end{align*}
Then
	\begin{align*}
	\mathrm{Herm}_{L, M}(R)=\{ \phi \in  \mathrm{Hom}&_{\Oo_F}(L_{R}, M_{R}) \mid     (\phi(x),\phi(y))_{R}\equiv (x,y)_{R}  \text{ for all }  x,y \in L_{R}\}.
	\end{align*}
To simplify the notation, we let
\begin{align}\label{eq:I(M,L_T,d)}
 I(M,L,d) \coloneqq \mathrm{Herm}_{L, M}(\Oo_{F_0}/(\pi_0^d)).
 \end{align}
 Then a direct calculation as in \cite[Lemma 6.1]{Shi2} shows that
	\begin{equation}\label{eq:local density M,L}
		\alpha(M, L) = \q^{-d n (2m -n)} |I(M,L,d)|
	\end{equation}
	for sufficiently large integers $d >0$. Since $\alpha(M,L, X)$ only depends on the Gram matrices  of $M$ and $L$, we may also denote it by $\alpha(S, T, X)$ if $S$ and $T$ are the Gram matrices of $M$ and $L$.
	
	Now we define primitive local density polynomials. For $1 \le \ell \le n$, let
	\begin{align}\label{eq: def of M^{n,l}}
		(M^{[k]})^{n,(\ell)}=\{ \bx=(x_1, \cdots, x_n) \in (M^{[k]})^n \mid  \dim \hbox{Span}\{x_1, \cdots, x_\ell\} =\ell  \hbox{ in } M^{[k]}/\pi M^{[k]}\}.
	\end{align}
	For
	$L=L_1\oplus L_2$, where $L_1=\mathrm{Span}\{l_1,\cdots,l_{\ell}\}$ and $L_2=\mathrm{Span}\{l_{\ell+1},\cdots,l_n\}$, we define the local $\ell$-primitive density to be
	\begin{equation}
		\beta(M^{[k]}, L_1\oplus L_2)^{(\ell)}= \int_{\Herm_n(F)} \int_{(M^{[k]})^{n,(\ell)}} \psi( \langle Y, T(\bx) -T \rangle) d\bx \,  dY.
	\end{equation}
	When $\ell \not= n$, the above definition depends on a choice of $L=L_1\oplus L_2$. Hence we always fix such a decomposition $L=L_1\oplus L_2$ in this case. When $L=L_1\obot L_2$, and $L_i$ is represented by $T_i$, we also denote $\beta(M, L_1\obot L_2)^{(\ell)}$ as $\beta(S, \diag(T_1,T_2))^{(\ell)}$. When $\ell=n$, we simply denote $\beta(M, L_1\oplus L_2)^{(\ell)}$ as $\beta(M,L)$.

	\begin{lemma}\label{lem: ind formula to prim ld}
		Assume $L=L_1 \oplus L_2$ where $\mathrm{rank}(L_1)=n_1$. Then
		$$\alpha(M,L,X)=\sum_{L_1\subset L_1'\subset L_{1,F}}(\q^{n-m}X)^{\ell(L_1'/L_1)}\beta(M,L_1'\oplus L_2,X)^{(n_1)},$$
		where $\ell(L_1'/L_1)=\mathrm{length}_{\Oo_F}L_1'/L_1$.
	\end{lemma}
	\begin{proof}
		This is the analogue of \cite[Lemma 3]{kitaoka1983note}. Let
		$G=\GL_{n_1}(F)\cap \mathrm{M}_{n_1}(\Oo_F)$ and $U=\GL_{n_1}(\Oo_F)$. By choosing a basis $\{l_1,\cdots,l_{n_1}\}$ of $L_1$, we may identify $U\backslash G$ with $\{L_1'\mid L_1\subset L_1'\subset L_{1,F}\}$ by sending $g$ to $L_1\cdot g^{-1}$.  Then the identity we want to prove is equivalent to
		\[\alpha(M,L,X)=\sum_{g\in U\backslash G}|\det g|^{2k+m-n} \beta(M,L_1\cdot g^{-1}\obot L_2,X)^{(n_1)},\]
		where $|\pi|=\q^{-1}.$
		By a partition of $M^n_{k}$, we have
		\begin{align*}
			\alpha(M,L,X)&=\int_{\Herm_n(F)} dY \int_{(M^{[k]})^n} \psi(\langle Y,T(\bx)-T\rangle) d\bx\\
			&=  \sum_{g\in U\backslash  G } \int_{\Herm_n(F)} dY \int_{(M^{[k]})^{n,(n_1)}\cdot g_1} \psi(\langle Y,T(\bx)-T\rangle) d\bx,
		\end{align*}
		where $g_1=\diag(g, I_{n-n_1})$, and the action of $g_1$ is simply matrix multiplication on the $n$ components of $M^{n,(n_1)}$. Now
		\begin{align*}
			&\int_{\Herm_n(F)} dY \int_{(M^{[k]})^{n,(n_1)}\cdot g_1} \psi(\langle Y,T(\bx)-T\rangle) d\bx\\
			&= |\det g_1|^{2k+m} \int_{\Herm_n(F)} dY \int_{(M^{[k]})^{n,(n_1)}} \psi(\langle Y,T(\bx g_1)-T\rangle) d\bx\\
			&= |\det g_1|^{2k+m} \int_{\Herm_n(F)} dY \int_{(M^{[k]})^{n,(n_1)}} \psi(\langle Y,(T(\bx)-T[g_1^{-1}])[g_1]\rangle) d\bx\\
			&= |\det g_1|^{2k+m} \int_{\Herm_n(F)} dY \int_{(M^{[k]})^{n,(n_1)}} \psi(\langle Y[g_1^*],T(\bx)-T[g_1^{-1}]\rangle) d\bx\\
			&= |\det g_1|^{2k+m-n} \int_{\Herm_n(F)} dY \int_{(M^{[k]})^{n,(n_1)}} \psi(\langle Y,T(\bx)-T[g_1^{-1}]\rangle) d\bx\\
			&= |\det g_1|^{2k+m-n} \beta(M^{[k]},L\cdot g_1^{-1})^{(n_1)}.
		\end{align*}
		Here $T[g]\coloneqq g^{*}Tg$.		Now the  lemma is clear.
	\end{proof}

	\begin{theorem}\label{thm:ind formula reducing valuation} Let  $L$ be as in Lemma \ref{lem: ind formula to prim ld}. Then
		\begin{align*}
			\alpha(M, L, X)
			&= \sum_{i=1}^{n_1} (-1)^{i-1} \q^{i(i-1)/2+i(n-m)}X^{i}
		 \cdot \sum_{\substack{L_1 \subset L_1' \subset \pi^{-1}L_{1} \\ \dim {L_1'/L_1}=i}} \alpha(M, L_1'\oplus L_2,X) + \beta(M, L,X)^{(n_1)}.
		\end{align*}
	\end{theorem}
	\begin{proof}
		This is an analogue of \cite[Proposition 2.1]{katsurada1999explicit}.
		The proof follows from a combination of the argument (in a reverse order) in	\ref{lem: decomposition of Z(L) as sums of D(L)} and Lemma \ref{lem: ind formula to prim ld}.
	\end{proof}

	Motivated by Theorem \ref{thm:ind formula reducing valuation}, we give the following definition.
	\begin{definition}\label{def:prim pden}
		Let $L =L_1 \oplus L_2$ be as in Lemma \ref{lem: ind formula to prim ld}. We define
		\begin{align}\label{eq:prim den}
			\pden(L)^{(n_1)}\coloneqq \pden(L)-\sum_{i=1}^{n_1} (-1)^{i-1} \q^{i(i-1)/2}
			\sum_{\substack{L_1 \subset L_1' \subset L_{1, F} \\ \dim {L_1'/L_1}=i}} \pden(L'_1\oplus L_2).
		\end{align}
	\end{definition}
	
	\begin{corollary}\label{cor: ind structure of partial Den(T)} Let $L =L_1 \oplus L_2$ be as in Lemma \ref{lem: ind formula to prim ld}, and $\epsilon=\chi(L)$. Then
		\begin{align*}
			\pden(L)^{(n_1)}=\frac{1}{\alpha(I_{n}^{-\epsilon},I_{n}^{-\epsilon})}\left(2\beta'(I_{n}^{-\epsilon},L)^{(n_1)}+\sum_{i}c^{n,i}_{\epsilon}\beta(\cH_{n,i}^{\epsilon},L)^{(n_1)}\right).
		\end{align*}
	\end{corollary}
	
	As a corollary of Lemma \ref{lem: ind formula to prim ld}, we have the following.
	\begin{corollary}\label{cor: decomp of pden(L)}
		Let $L =L_1 \oplus L_2$ be as in Lemma \ref{lem: ind formula to prim ld}. Then we have the following identity where the summation is finite:
		$$
		\pden(L)=\sum_{L_1 \subset L_1' \subset L_{1,F}}   \pden(L_1'\oplus L_2)^{(n_1)}.
		$$
	\end{corollary}
	We may reduce the identity $\Int(L)=\pden(L)$ to a primitive version as the following theorem shows.
	\begin{theorem}\label{thm: equivalent form of modified KR}
		Let  $L=L_1 \oplus L_2 \subset \bV$ be  as in Lemma \ref{lem: ind formula to prim ld}.

		(1) \quad 	 Conjecture \ref{conj:main} is true for $L$ if for every $L_1 \subset L_1' \subset L_{1, F}$, we have
		$$
		\Int(L_1' \oplus L_2)^{(n_1)}= \pden(L_1' \oplus L_2)^{(n_1)}.
		$$
		
		(2)\quad   If  Conjecture \ref{conj:main} holds for all lattices $L'=L_1'\oplus L_2$ of $\bV$ of rank $n$ with $L_1 \subset L_1' \subset L_{1, F}$, then
		$$ \Int(L_1 \oplus L_2)^{(n_1)} =\pden(L_1 \oplus L_2)^{(n_1)}.$$

		(3) \quad   For $1 \le n_1 \le n$,  Conjecture \ref{conj:main} is true if and only if for every lattice $L =L_1 \oplus L_2 \subset \bV$ with $\mathrm{rank}(L_1)=n_1$, one has
		$$ \Int(L_1 \oplus L_2)^{(n_1)} =\pden(L_1 \oplus L_2)^{(n_1)}.$$
	\end{theorem}
	\begin{proof} (1)  follows  from Lemma \ref{lem: decomposition of Z(L) as sums of D(L)} and Corollary \ref{cor: decomp of pden(L)}.  (2) follows from Definitions \ref{def:differencecycle} and \ref{def:prim pden}. (3) follows from  (1) and (2).
	\end{proof}

	For the rest of this section, we assume that $M$ is unimodular of rank $m$ with a Gram matrix $\diag(I_{m-1}, \nu)$. To go further with
	the calculation of $\alpha(M,L,X)$, we need an induction formula for  $\beta(M,L,X)^{(\ell)}$ as follows.  The proof is essentially the same as that of Corollary $9.11$ of \cite{KR1}, and is left to the reader.

	\begin{proposition}\label{prop:ind formula reducing size}
		Let $L=L_1 \obot L_2$, where $L_j$ is of rank $n_j$. Let $C(M^{[k]}, L_1)$ be the   $\mathrm{U}(M^{[k]})$-orbits of sublattices $M(i) \subset M^{[k]}$ such that $M(i)$ is isometric to $L_1$, and write $C(M^{[k]}, L_1)=\sqcup_{i\in J}\{ M(i) \}$. Then
		\begin{align*}
			&\beta(M, L, X)^{(n_1)}
		 =\sum_{i\in J} |M:M(i) \obot M(i)^{\perp}|^{-n_2} |M(i)^{\vee}:M(i)|^{n_2} \beta_i(M, L_1, X) \alpha(M(i)^{\perp},L_2),
		\end{align*}
		where
		\begin{align*}
			&\beta_i(M, L_1, X)
			=\lim_{d\to \infty}\q^{-d n_1 (2m+4k -n_1)}
		 \#\{ \phi \in I(M^{[k]}, L_1,d)^{(n_1)} \mid  \exists\, \ \Phi \in \mathrm{U}(M) \text{ with } \phi(L_1)=\Phi(M(i))\} ,
		\end{align*}
		and
		\begin{align*}
			I(M^{[k]},L_1,d)^{(n_1)}\coloneqq \{\phi\in I(M^{[k]},L_1,d)\mid \mathrm{rank}_{\mathbb{F}_\q}\phi(L_1)\otimes_{\Oo_F} \mathbb{F}_\q =n_1\}.
		\end{align*}
		Recall that $I(M^{[k]},L_1,d)$ is defined in \eqref{eq:I(M,L_T,d)}.
	\end{proposition}
	
	One special case is that $L=\cH^i\obot L_2$.	Since any sublattice of $M^{[k]}=M\obot \cH^k$ isometric to $\cH^i$ is always a direct summand of $M^{[k]}$ and $\alpha(M,\cH^i,X)=\beta(M,\cH^i,X)^{(2i)}=\beta(M,\cH^i,X)$,  the above proposition specializes as follows.
	\begin{corollary}\label{cor: L=H^i obot L_2}
		Assume $L=\cH^i\obot L_2$, then 	
		\begin{align}
			\alpha(M,L,X)=\beta(M,\cH^i,X)\alpha(M,L_2,\q^{2i}X)=\alpha(M,\cH^i,X)\alpha(M,L_2,\q^{2i}X).
		\end{align}
	\end{corollary}

	We end this section with two more special cases of Proposition \ref{prop:ind formula reducing size}. Proofs are given in Appendix \ref{sec: calc of primitive local density}.
	
	\begin{proposition} \label{prop: ind for t} Let the notation  be as in Proposition \ref{prop:ind formula reducing size}. Assume $n_1=1$ and $L_1 =\langle t \rangle$ where $t\in \Oo_{F_0}.$
		\begin{enumerate}
			\item There always exists a primitive vector $M(1)\in \cH^k$ with $q(M(1))=t$, and
			\begin{align*}
				M(1)^{\perp} \cong \cH^{k-1} \obot I_{m}^{\chi(M)} \obot \langle -t \rangle.
			\end{align*}
			Here $\langle t \rangle$ denotes a lattice $\Oo_F v$ of rank one with $(v, v)=t$.
			
			\item 	If $\mathrm{v}(t)=0$, then there exist a primitive vector $M(0)\in M$ with $q(M(0))=t$, and
			\begin{align*}
				M(0)^{\perp}\cong  	\cH^{k}\obot I_{m-2}^{\epsilon_{m-2}} \obot \langle \nu t \rangle.
			\end{align*}
			Here $\epsilon_{m-2}=\chi((-1)^{ (m-2)(m-3)/2})$.
			
			\item If $\mathrm{v}(t)>0$, then there exist a primitive vector $M(0)\in M$ with $q(M(0))=t$ only when $M$ is isotropic (i.e. $\exists\, v\in M$ with $q(v)=0$). In this case,
			\begin{align*}
				M(0)^{\perp}\cong 	\cH^{k}\obot I_{m-2}^{\chi(M)} \obot \langle -t \rangle.
			\end{align*}
			Assuming the existence of $M(1)$ and $M(0)$, we have
			\begin{align*}
				|M^{[k]}:M(i)\obot M(i)^{\perp}|^{-1}|M(i)^{\vee}:M(i)|=\begin{cases}
					1 & \text{ if $i=1$},\\
					\q & \text{ if $i=0$}.
				\end{cases}
			\end{align*}
			\item
			Under the action of $\mathrm{U}(M^{[k]})$,	$v$ is either in the same orbit of a fixed vector $M(1)\in \cH^k$ or a fixed vector $M(0)\in M$.

			\item We have the following induction formula:
			\begin{align*}
				&\beta(M, L,X)^{(1)} =\beta_1(M,L_1,X) \alpha(M(1)^{\perp},L_2)+q^{n-1}\beta_0(M, L_1,X) \alpha(M(0)^{\perp},L_2).
			\end{align*}
			Moreover,
			\begin{enumerate}
				\item  for any $L_1$, $$\beta_1(M, L_1,X)=1-X.$$
				\item assume $\mathrm{v}(t)=0$, then
				\begin{align*}
					\beta_0(M, L_1 ,X)=\begin{cases}
						(1+\chi(M) \chi(L)q^{-\frac{m-1}{2}})X & \text{ if $m$ is odd},\\
						(1-\chi(M)q^{-\frac{m}{2}})X & \text{ if $m$ is even}.
				\end{cases}\end{align*}
				\item assume $\mathrm{v}(t)>0$, then
				\begin{align*}\beta_0(M, L_1 ,X)=\begin{cases}
						(1-q^{1-m})X& \text{ if $m$ is odd},\\
						\left(1-q^{1-m}+\chi(M)(q-1)q^{-\frac{m}{2}}\right)X& \text{ if $m$ is even}.
				\end{cases}\end{align*}
			\end{enumerate}
		\end{enumerate}
	\end{proposition}
	\begin{proof}
		Parts  (1)---(4) are proved in subsection \ref{subsec: 1-4 5.9}. The induction formula for $\beta(M, L,X)^{(1)}$ follows from Proposition \ref{prop:ind formula reducing size}. For the formula of $\beta_i(M, L_1 ,X)$, see Corollaries \ref{cor: beta2} and  \ref{cor: beta N2}.
	\end{proof}
	\begin{proposition}\label{prop: ind rank 2}
		Let the notation be as in Proposition \ref{prop:ind formula reducing size}. Assume $\mathrm{v}(L_1)>0$ and $n_1=2$. Then we have a partition of $ C(M^{[k]}, L_1)=\bigsqcup_{i=0}^{2}C_i(M^{[k]},L_1)$ such that for any $M(i)\in C_i(M^{[k]},L_1)$, $M(i)^{\bot}$ is isometric to
		$$
		(-L_1) \obot \cH^{k-i} \obot M^{(i)}.
		$$
		Here $M^{(i)}$ is a unimodular $\Oo_F$-lattice of rank $m-2(2-i)$ and has determinant $(-1)^i \det L$.
		
		Moreover,  we have
		\begin{align}\label{eq: T_1}
			\beta(M, L,X)^{(2)}=\sum_{i=0}^{2}\q^{(2-i)(n-2)}\beta_i(M,L_1,X)\alpha(M(i)^{\perp}, L_2,X),
		\end{align}
		where
		\begin{align*}
			\beta_2(M,L_1,X)&=(1-X)(1-\q^2 X),\\
			\beta_1(M,L_1,X)&=
			\q(\q+1)\left( (1-\q^{1-m})+\delta_{e}(m)\chi(M)(\q-1)\q^{-\frac{m}{2}}\right)X(1-X),\\
			\beta_0(M,L_1,X)
			 &=\begin{cases}
				\q(1-\q^{1-m})(1-\q^{3-m})X^2& \text{if $m$ is odd},\\
				\q\left( (1-\q^{2-m})+\chi(M)(\q^2-1)\q^{-\frac{m}{2}}\right)(1-\q^{2-m})X^2& \text{if $m$ is even}.
			\end{cases}
		\end{align*}
		Here $\delta_e(m)=1$ or $0$ depending on whether $m$ is even or odd.
	\end{proposition}
	\begin{proof}
		Equation \eqref{eq: T_1} follows from Proposition \ref{prop:ind formula reducing size} and Proposition \ref{prop:A perbofHi}. For the formula of $\beta_i(S,L_1,X)$, see Corollaries \ref{cor: beta2}, \ref{cor:beta0} and  Proposition \ref{prop:beta1 new method}.
	\end{proof}

	\section{The modified Kudla-Rapoport conjecture}\label{sec: Kudla-Rapoport conjecture}

Recall that the Hermitian lattices used to define the correction terms are of the following forms:
	\begin{equation}\label{eq:H n i epsilon}
		\cH_{n,i}^{\epsilon}\coloneqq \cH^i\obot I_{n-2i}^{\epsilon} \quad \text{ for } 1\leq i \leq \frac{n}{2}, \quad \epsilon =\pm 1
	\end{equation}
	where  $I_{n-2i}^{\epsilon}$ is the  unimodular Hermitian lattice of rank $n-2i$ with
	\noindent $\chi(I_{n-2i}^{\epsilon})= \chi(\cH_{n,i}^{\epsilon} )=\epsilon$. When $n=2r$ is even, we take $I_{0}^{\epsilon} =0$ and  $\cH_{n, r}^1 = \cH^{r}$.

	\begin{theorem} \label{thm: A_epsilon}
		Let $r_\epsilon=\frac{n-1}2$ when $n$ is odd, and  $r_\epsilon=\lfloor\frac{n+\epsilon}2\rfloor$ when $n$ is even. In the following we just write $r_{\epsilon}$ as $r$:
		$$
		A^\epsilon =(A_{i, j}^\epsilon) =
		\begin{pmatrix}
			\alpha(\cH_{n,1}^{\epsilon}, \cH_{n,1}^{\epsilon})
			&\alpha(\cH_{n,2}^{\epsilon}, \cH_{n,1}^{\epsilon})
			&\cdots
			&\alpha(\cH_{n,r}^{\epsilon}, \cH_{n,1}^{\epsilon})
			\\
			0 &\alpha(\cH_{n,2}^{\epsilon}, \cH_{n,2}^{\epsilon})
			&\cdots
			&\alpha(\cH_{n,r}^{\epsilon}, \cH_{n,2}^{\epsilon})
			\\
			\cdots &\cdots &\cdots &\cdots
			\\
			0& 0&0
			&\alpha(\cH_{n,r}^{\epsilon}, \cH_{n,r}^{\epsilon})
		\end{pmatrix},
		$$
		$$
		B^\epsilon
		= {}^t(\alpha'(I_{n}^{-\epsilon}, \cH_{\epsilon}^{n, 1}), \cdots,
		\alpha'(I_{n}^{-\epsilon}, \cH_{n,r}^{\epsilon})),
		$$
		and
		$$
		C^\epsilon={}^t(c_\epsilon^{n, 1}, \cdots, c_\epsilon^{n, r}),
		$$
		where $c_{n,i}^\epsilon $ is as in Conjecture \ref{conj:main}.

		Then  $C^\epsilon$ is the solution of the equation
		\begin{equation}\label{eq: AC=B}
			A^\epsilon C^\epsilon=-2 B^{\epsilon}.
		\end{equation}
		Moreover,
		\begin{align}\label{eq: A^{jj}}
			A_{j,j}^\epsilon=2q^{\frac{(n-2j)(n-2j-1)}{2}} \prod_{0<s\le j}(1-\q^{-2s})\prod_{1 \le s \le \lfloor\frac{n-2j-1}{2}\rfloor}(1-\q^{-2s})
			\begin{cases}
				1 &\text{ if $n$ is odd},\\
				1 -\epsilon  \q^{-\frac{n-2j}{2}}  &\text{ if $n$ is even}.
			\end{cases}
		\end{align}
		Finally,  for $i<j$,
		\begin{align}\label{eq: A^{ij}}
			A_{i,j}^\epsilon= A_{j,j}^\epsilon\cdot \begin{cases}
				I(n-2i,\frac{n-2i-1}{2},j-i)& \text{ if $n$ is odd},\\
				I(n-2i,\frac{n-2i-1+\epsilon}{2},j-i)& \text{ if $n$ is even},
			\end{cases}
		\end{align}
		where
		\begin{align*}
			I(n,d,k)\coloneqq  \prod_{s=1}^{k}\frac{(\q^{d-s+1}-1)(\q^{n-d-s}+1)}{\q^s-1}.
		\end{align*}
	\end{theorem}

	\begin{proof}
		First notice that $\alpha(\cH_{n, i}^\epsilon, \cH_{n,j}^{\epsilon})=0$ if  $i<j$. Thus, (\ref{eq:coeff}) is indeed equivalent to \eqref{eq: AC=B}, and there exists a unique solution $C_{\epsilon}$.
		
		Now we compute $A_{j,j}^\epsilon$ explicitly. Corollary \ref{cor: L=H^i obot L_2} and Lemma \ref{lem: cancel beta_2} imply that
		\begin{align*}
			\alpha(\cH_{n,j}^{\epsilon},\cH_{n,j}^{\epsilon})&=\alpha(\cH^j,\cH^j)\alpha(I_{n-2j}^{\epsilon},I_{n-2j}^{\epsilon}).
		\end{align*}
		According to Lemma \ref{lem:beta(H,L)},
		\begin{align*}
			\alpha(\cH^j,\cH^j)=\prod_{0<s\le j}(1-\q^{-2s}).
		\end{align*}
		By Lemma \ref{lem: cancel beta_0},
		\begin{align*}
			\alpha(I_{n-2j}^{\epsilon},I_{n-2j}^{\epsilon})=|\mathrm{O}(\overline{I}_{n-2j}^{\epsilon})(\mathbb{F}_\q)|,
		\end{align*}
		where $\overline{I}_{n-2j}^{\epsilon}=I_{n-2j}^\epsilon \otimes_{\Oo_F} \Oo_{F}/(\pi)$ is the  space over $\mathbb{F}_q$ with the naturally induced quadratic form.
		Now \eqref{eq: A^{jj}} follows from  the well-known formula:
		\begin{align*}
			&|\mathrm{O}(\overline{I}_{n-2j}^{\epsilon})(\mathbb{F}_\q)|
		 =\begin{cases}
				2q^{\frac{(n-2j)(n-2j-1)}{2}}\prod _{s=1}^{\frac{n-2j-1}{2}}(1-\q^{-2s})& \text{ if $n$ is odd},\\
				2q^{\frac{(n-2j)(n-2j-1)}{2}}(1-\epsilon \q^{-\frac{n-2j}{2}})\prod _{s=1}^{\frac{n-2j}{2}-1}(1-\q^{-2s})& \text{ if $n$ is even}.
			\end{cases}
		\end{align*}
		
		To obtain \eqref{eq: A^{ij}}, notice that
		(Corollary \ref{cor: L=H^i obot L_2})
		\begin{align*}
			\alpha(\cH_{n,j}^{\epsilon},\cH_{n,i}^{\epsilon})=\alpha(\cH_{n,j}^{\epsilon},\cH^{i}) \alpha(\cH_{n-2i,j-i}^{\epsilon},I_{n-2i}^{\epsilon}),
		\end{align*}
		and
		\begin{align*}
			\alpha(\cH_{n,j}^{\epsilon},\cH_{n,j}^{\epsilon})=\alpha(\cH_{n,j}^{\epsilon},\cH^{i}) \alpha(\cH_{n-2i,j-i}^{\epsilon},\cH_{n-2i,j-i}^{\epsilon}).
		\end{align*}
		Hence
		\begin{align*}
			\frac{A_{i, j}^\epsilon}{A_{j, j}^\epsilon}=		\frac{\alpha(\cH_{n,j}^{\epsilon},\cH_{n,i}^{\epsilon})}{\alpha(\cH_{n,j}^{\epsilon},\cH_{n,j}^{\epsilon})}=\frac{\alpha(\cH_{n-2i,j-i}^{\epsilon},I_{n-2i}^{\epsilon})}{\alpha(\cH_{n-2i,j-i}^{\epsilon},\cH_{n-2i,j-i}^{\epsilon})}.
		\end{align*}
		
		Fix an $\Oo_F$-lattice $L$ that is represented by $I_{n-2i}^{\epsilon}$. According to Lemma \ref{independentofT}, to compute $\frac{A_{i, j}^\epsilon}{A_{j, j}^\epsilon}$, we need to count the number of lattices $L'$ in $ L_{F}$ such that contain $L\subset L'$ and $L' \cong \cH_{n-2i,j-i}^{\epsilon}$, which is equivalent to the following condition:
		\begin{align*}
			\pi L \stackrel{j-i}{\subset} \pi L' \stackrel{n-2j}{\subset} (L')^\sharp\stackrel{j-i}{\subset}  L \stackrel{j-i}{\subset}{L'}.
		\end{align*}
		Since $L'$ and $\pi L'$ determine each other,  we just need to count $\pi L'$ satisfying the above condition. We regard $\pi L'/\pi L$ as a $(j-i)$-dimensional subspace of $L/\pi L$, where $L/\pi L$ is equipped with quadratic form $(x,y)/\pi$.
		
		Claim: The condition
		\begin{align*}
			\pi L' \stackrel{}{\subset} (L')^\sharp
		\end{align*}
		is equivalent to the condition that $\pi L'/\pi L$ is an isotropic subspace of $L/\pi L.$
		
		Indeed,  assume $\pi L'/\pi L$ is an isotropic subspace of $L/\pi L.$ Then $(\pi x, \pi y)\in \pi \Oo_F$ for any $x,y\in L'$, which is equivalent to $(x,\pi y)\in \Oo_F$ for any $x,y\in L'$.  The latter condition is the same as $  L' \stackrel{}{\subset} (L')^\sharp$. The other direction is clear.
		
		Therefore
		$\frac{A_{i, j}^\epsilon}{A_{j, j}^\epsilon}$ is the number of $(j-i)$-dimensional isotropic subspaces of $L/\pi L$. According to \cite[Lemma 3.2.2]{LZ2}, it equals to
		$$\begin{cases}
			I(n-2i,\frac{n-2i-1}{2},j-i)& \text{ if $n$ is odd},\\
			I(n-2i,\frac{n-2i-1+\epsilon}{2},j-i)& \text{ if $n$ is even}.
		\end{cases}$$
	\end{proof}

	According to Theorem \ref{thm: A_epsilon},
		in order to solve $C^\epsilon$, we need to know  $B^\epsilon$ and $A^\epsilon$. Here, $B^\epsilon$ can be calculated by applying Corollary \ref{cor: L=H^i obot L_2} and Proposition \ref{prop: ind for t} inductively. The following lemma can be used to compute $A^\epsilon$.
	\begin{lemma} \label{independentofT}
		Let $F/F_0$ be a quadratic $p$-adic field extension, and let $L$ and $M$ be two $\Oo_F$-Hermitian lattices of rank $n$.  Then $\frac{\alpha(M, L)}{\alpha(M, M)}$ is equal to the number of lattices $L'$ in  $L_F$   containing  $L$  and isometric to $M$.
	\end{lemma}
	\begin{proof}
		The proof is a generalization of that of Proposition 10.2 of \cite{KR2} and works for both inert and ramified primes.
		
		Let us assume that there is an isometric embedding from $L$ into $M$, otherwise both sides of the identity in the lemma are zero. In this case, we have a fixed $L_F \cong M_F$.
		Let $\alpha$ (respectively, $\beta$) be a top degree translation invariant form on $L_F^n$ (respectively, $\Herm_n(F)$). Let $\nu_p=\alpha/h^*(\beta)$ where
		\[h: \quad L_F^n\rightarrow \Herm_n(F),\quad x\mapsto (x,x).\]
		Define $X$ to be the set of $F$-linear isometric embeddings from $L$ into $M$. By fixing a basis of $L_F$ and regarding $\phi\in X$ as a linear isometry from $L_F$ to itself, we identify $X$ as a subset of $L_F^n$.
		By the argument in Section 3 of \cite{GY} (in particular Lemma 3.4), we know that
		\begin{equation}\label{eq:alpha(S',T)}
			\alpha(M,L)=\mathrm{vol}(X,d\nu_p) \frac{\mathrm{vol}(\Herm_n(\Oo_{F}),d\beta)}{\mathrm{vol}((M)^n,d\alpha)}.
		\end{equation}
		For any $\phi\in X$ regarded as a linear isometry from $L_F$ to itself, the lattice $L_\phi\coloneqq \phi^{-1}(M)$ is a lattice containing $L$. Conversely, for any $L'$ containing $L$ and isometric to $M$, there is a $\phi\in X$ such that $L_\phi=L'$. Hence we have a partition
		\[X=\bigsqcup_{L \subset L'} X_{L'}, \quad X_{L'}\coloneqq \{\phi\in X\mid L_\phi=L'\}.\]
		Since each $L'$ is isomorphic to $M$,  all the $X_{L'}$ have the same volume  as that of $X_M$. Specializing (\ref{eq:alpha(S',T)}) to $L=M$, we see
		\begin{equation}\label{eq:alpha(S',S')}
			\alpha(M,M)=\mathrm{vol}(X_{M},d\nu_p) \frac{\mathrm{vol}(\Herm_n(\Oo_{F}),d\beta)}{\mathrm{vol}((M)^n,d\alpha)}.
		\end{equation}
		Dividing  equation \eqref{eq:alpha(S',T)} by \eqref{eq:alpha(S',S')}, we prove the lemma.
	\end{proof}
	\begin{remark}
		When $F/F_0$  is unramified and $M$ is unimodular, the lemma was proved by equation (3.6.1.1) of \cite{LZ}.
	\end{remark}
	Now we specialize Theorem \ref{thm: A_epsilon} to the case $n=3$.
	\begin{lemma}\label{lem: coeff n=3}
		Assume  $n=3$ and $\epsilon =\chi(L)$. Then $c_{3,1}^\epsilon=\frac{q^2}{1+q}$, hence $$\pden(L)=2\frac{\alpha'(I_{3}^{-\epsilon},L)}{\alpha(I_{3}^{-\epsilon},I_{3}^{-\epsilon})}+\frac{\q^2}{1+\q}\frac{\alpha(\cH_{3,1}^{\epsilon},L)}{\alpha(I_{3}^{-\epsilon},I_{3}^{-\epsilon})}.$$
	\end{lemma}
	\begin{proof}
		First of all, according to  Theorem \ref{thm: A_epsilon}, 	\begin{align}\label{eq: alpha(H3,H3)}
			\alpha(\cH_{3,1}^{\epsilon},\cH_{3,1}^{\epsilon})=2(1-\q^{-2}).
		\end{align}
		By Corollary \ref{cor: L=H^i obot L_2}, we have
		\begin{align*}
			\alpha (I_{3}^{-\epsilon}, \cH_{3,1}^{\epsilon},X)=\alpha(I_{3}^{-\epsilon}, \cH,X)\alpha(I_{3}^{-\epsilon},I_{1}^{\epsilon},\q^2 X).
		\end{align*}
		According to Lemmas   \ref{lem:beta(H,L)} and \ref{lem: cancel beta_2},
		\begin{align*}
			\alpha(I_{3}^{-\epsilon}, \cH,X)= \beta(\cH^k,\cH)=1-X.
		\end{align*}
		Lemma \ref{lem:alpha(I_{a,b})} gives that
		\begin{align*}
			&\alpha(I_{3}^{-\epsilon},I_{1}^{\epsilon},\q^2 X)=1-qX .
		\end{align*}
		Hence
		\begin{align*}
			\alpha (I_{3}^{-\epsilon}, \cH_{3,1}^{\epsilon},X)=(1-X)(1-qX),
		\end{align*}
		and
		\begin{align*}
			\alpha' (I_{3}^{-\epsilon}, \cH_{3,1}^{\epsilon})=1-q.
		\end{align*}
		Combining this with \eqref{eq: alpha(H3,H3)}, we solve \eqref{eq: AC=B} and obtain
		\begin{align*}
			c_{3,1}^\epsilon=\frac{q^2}{1+q}.
		\end{align*}
		Now the lemma follows from (\ref{eq1.10}).
	\end{proof}

	\section{Local density formula when $\mathrm{rank}(T)\le 2$}\label{sec: eg application}
	The main purpose of this section is to give an explicit formula for $\alpha(I, T, X)$ where $I=\diag(I_{m-1}, \nu)$ with $\nu \in \Oo_{F_0}^\times$ and $\mathrm{rank}(T)\le 2$.

	\subsection{The case $T=(t)$.}  In order to apply induction formulas to calculate $\alpha(I, T, X)$ for $T$ with $\hbox{rank}(T) =2$, we need to consider the case $T=(t)$ first.
	Write  $t=t_0(-\pi_0)^{\mathrm{v}(t)}$ for $t_0\in \Oo_{F_0}^\times$,  and
	\begin{equation}
		I_{a, b}=\diag(I,\nu_1(-\pi_0)^a,\nu_2(-\pi_0)^{b}) =\diag(s_1,\dots,s_{m+2})
	\end{equation}
	for integers $0\le a \le b$.	
	\begin{lemma}\label{lem:alpha(I_{a,b})}
		Assume $0\le a\le b\le \mathrm{v}(t)$.
		\begin{enumerate}
			\item If $m$ is odd, then
			\begin{align*}
				\alpha(I_{a,b},(t),X)=1&+\chi(I)\chi(-\nu_1)(\q-1)\sum_{s=a+1}^{b}q^{-ms+a+\frac{m-1}{2}}X^s
				\\
				&  +\chi(I_{a,b})\chi( t_0)q^{-(m+1)\mathrm{v}(t)+a+b-\frac{m+1}{2}}X^{\mathrm{v}(t)+1}.
			\end{align*}
			\item If $m$ is even, then
			\begin{align*}
				\alpha(I_{a,b},(t),X)
				=1&+\chi(I) (\q-1) \sum_{s=1}^{a}\q^{-(m-1)s+\frac{m}{2}-1}X^s\\
				&+\chi(I_{a,b}) q^{a+b}\left((\q-1) \sum_{s=b+1}^{\mathrm{v}(t)}\q^{-(m+1)s+\frac{m}{2}} X^s - \q^{-(m+1)\mathrm{v}(t)-1-\frac{m}{2}}X^{\mathrm{v}(t)+1}\right).
			\end{align*}
		\end{enumerate}
	\end{lemma}
	\begin{proof}
		Direct calculation gives
		\begin{align*}
			\alpha(I_{a,b},(t),X)
			&=\int_{F_0}\, dY \int_{\Oo_F^{2k+m+2}} \psi (\langle Y, \diag(\cH^k,I_{a,b})[\bx]-t \rangle)\, d\bx
			\\ \notag
			&=\int_{F_0} \psi(-tY)\, dY
		  \int_{\Oo_F^{2k} \times \Oo_F^{m+2}}\psi ( Y \sum_{i=1}^{k} \tr(\frac{1}{\pi}x_i\bar{y}_i) +Y \sum_{l=1}^{m+2}s_lz_l\bar{z}_{l})  \prod_{i}dx_i dy_i \prod_{l}dz_l
			\\ \notag
			&=1+\sum_{s=1}^{\infty}\int_{\mathrm{v}(Y)=-s} I_k(Y) I_{I_{a,b}}(Y)\psi(-tY) \, dY.
		\end{align*}
		Here, according to \cite[Lemma 7.6]{Shi2},
		$$
		I_k(Y)=\int_{\Oo_F^{2k}}\psi ( Y  \sum_{i=1}^{k} \tr( \frac{1}{\pi}x_i\bar{y}_i)) \prod dx_i dy_i= q^{-2ks} ,
		$$
		and
		$$
	I_{I_{a,b}}(Y) =\int_{\Oo_F^{m+2}}\psi( Y \sum_{l=1}^{m+2}s_lz_l\bar{z}_{l}) \prod dz_l=\prod_{l=1}^{m+2} J(s_l Y),
		$$
		where
		\begin{equation}
			J(t) =\int_{\Oo_F} \psi(t z \bar z) dz = \begin{cases}
				1&\hbox{ if } \mathrm{v}(t) \ge 0,
				\\
				q^{\mathrm{v}(t)}  \chi(-t_0) g(\chi,\psi_{\frac{1}{\pi_0}})&\hbox{ if }  \mathrm{v}(t) <0,
			\end{cases}
		\end{equation}
		and
		$$
		g(\chi,\psi_{\frac{1}{\pi_0}}) =\sum_{x \in \Oo_{F_0}/\pi_0} \chi(x) \psi(\frac{x}{\pi_0})
		$$
		is the Gauss sum. Write $\psi'=\psi_{\frac{1}{\pi_0}}$. Then
		\begin{align*}
			\alpha(I_{a,b},(t),X)
			&=1+\sum_{s=1}^{a}q^{s}\int_{\Oo_{F_0}^{\times}}q^{-2ks}\cdot q^{-ms} \chi(\nu (-Y)^m)g(\chi,\psi')^{m}\psi(-(-\pi_0)^s Yt) dY\\ \notag
			&+\sum_{s=a+1}^{b}\int_{\Oo_{F_0}^{\times}}q^{-2ks}\cdot q^{-m s+a} \chi(\nu_1\nu (-Y)^{m+1})g(\chi,\psi')^{m+1}\psi(-(-\pi_0)^{-s} Yt)\, dY\\ \notag
			&+\sum_{s=b+1}^{\infty}\int_{\Oo_{F_0}^{\times}}q^{-2ks}\cdot q^{-(m+1)s+a+b} \chi(\nu_1\nu_2\nu (-Y)^{m+2})g(\chi,\psi')^{m+2}\psi(-(-\pi_0)^{-s} Yt)\, dY.
		\end{align*}
		Recall the well-known facts that
		\newcommand{\cha}{\operatorname{Char}}
		\begin{align}
			g(\chi,\psi')^2&=\chi(-1)\cdot \q, \notag
			\\
			\int_{\Oo_{F_0}^\times} \psi((-\pi_0)^{-s} Y t) dY &= \cha(\pi_0^s\Oo_{F_0})(t) - q^{-1}\cha(\pi_0^{s-1}\Oo_{F_0})(t),
			\\
			\int_{\Oo_{F_0}^\times }\chi(Y)\psi((-\pi_0)^{-s} Y t) dY &=\chi(-t_0) q^{-1} g(\chi, \psi')\cha(\pi_0^{s-1}\Oo_{F_0}^\times)(t). \notag
		\end{align}
		When $m$ is odd, we have
		\begin{align*}
			\alpha(I_{a,b}, (t), X)=1&+\chi((-1)^{\frac{m+1}{2}}\nu_1\nu )(\q-1)\sum_{s=a+1}^{b}\q^{-ms+a+\frac{m-1}{2}}X^s
			\\
			& +\chi((-1)^{\frac{m+1}{2}}\nu_1\nu_2\nu t_0)q^{-(m+1)(\mathrm{v}(t)+1)+a+b+\frac{m+1}{2}}X^{\mathrm{v}(t)+1}.
		\end{align*}
		When $m$ is even, we have
		\begin{align*}
			&\alpha(I_{a,b}, (t), X)
			\\&=1+\chi((-1)^{\frac{m}{2}}\nu) (\q-1) \sum_{s=1}^{a}\q^{-(m-1)s+\frac{m}{2}-1}X^s +\chi((-1)^{\frac{m+2}{2}}\nu_1\nu_2\nu)
			\\
			&\quad \cdot \left((\q-1) \sum_{s=b+1}^{\mathrm{v}(t)}\q^{-(m+1)s+a+b+\frac{m}{2}} X^s - \q^{-(m+1)(\mathrm{v}(t)+1)+a+b+\frac{m}{2}}X^{\mathrm{v}(t)+1}\right).
		\end{align*}
		Finally, notice that for $I$ of rank $m$ we have
		$$\chi(I)=\begin{cases}
			\chi((-1)^{\frac{m-1}{2}}\nu)& \text{ if $m$ is odd},\\
			\chi((-1)^{\frac{m}{2}}\nu)& \text{ if $m$ is even}.
		\end{cases}$$
		Now the lemma is clear.
	\end{proof}

	Similarly, we have the following lemma.
	\begin{lemma}\label{lem:alpha(S^{[k]} obot H_a, t)}
		Let $I$ be unimodular with odd rank $m$. Then
		\begin{align*}
			\alpha( I\obot \cH_{i},(t), X)=\begin{cases}
				1+\chi(I)\chi( t_0)\q^{-(\mathrm{v}(t)+1)(m+1)+\frac{m+1}{2}+i} X^{\mathrm{v}(t)+1}& \text{ if $i\le 2\mathrm{v}(t)$},\\
				1+\chi(I)\chi( t_0)\q^{-(\mathrm{v}(t)+1)(m-1)+\frac{m-1}{2}}X^{\mathrm{v}(t)+1}& \text{ if $i> 2\mathrm{v}(t)$}.
			\end{cases}
		\end{align*}
	\end{lemma}

	\subsection{The case  $T=\diag(u_1(-\pi_0)^a,u_2(-\pi_0)^b)$}\label{subsection: S unimodular of even rank>2}
	In this subsection,	we compute $\alpha(I,T,X)$ for $I$ unimodular of rank $m\ge 2$ and $T=\diag(u_1(-\pi_0)^a,u_2(-\pi_0)^b)$ with $0\le a \le b$. Notice that $\alpha(I, T, X)=0$ when $ a<0$.
	\begin{proposition}\label{prop: unimodular m even}
		Assume $T=\diag(u_1(-\pi_0)^a,u_2(-\pi_0)^b)$ and that $I$ is isotropic of even rank $m \ge 2$,  then
		\begin{align*}
			\alpha(I,T,X)=&(1-X)\left(\sum_{i=0}^{a}(\q^{2-m}X)^{i}+\gamma_e(I,T,X)\right)\\
			 &+\q X(\q^{2-m}X)^{a}(1-\chi(I)\q^{-\frac{m}{2}})( 1+ \chi(I)\chi(T)\q^{\frac{m-2}{2}}(\q^{2-m} X)^{b+1})\\
			&+ \left(1-q^{-(m-1)}+(q-1)\chi(I)q^{-\frac{m}{2}} \right) X
			\\
			&\quad
			\cdot \left(\q \sum_{i=0}^{a-1}(\q^{2-m}X)^i+\gamma_e(I,T,X)-\chi(I)\chi(T)\q^{\frac{m}{2}}(\q^{2-m}X)^{a+b+1} \right),
		\end{align*}
		where
		\begin{align*}
			&\gamma_{e}(I,T,X)
			 = \chi(I)q^{\frac{m}{2}}\left(\sum_{d=1}^a(\q^d-1)(\q^{2-m}X)^{d} +\chi(T)\q^a (\q^{2-m}X)^{b+1}\sum_{i=0}^{a}(\q^{1-m}X)^{i}\right).
		\end{align*}
	\end{proposition}
	
	\begin{proof}
		Since $I$ is of even rank, $u_1\cdot I^{[k]}\approx I^{[k]}$, and we may  assume $T$ is of the form $\diag( (-\pi_0)^a,u(-\pi_0)^b)$  without loss of generality.
		
		According to Theorem \ref{thm:ind formula reducing valuation} and  Proposition \ref{prop: ind for t}, we have
		\begin{align*}
			\alpha(I,T, X)
			&=\beta_1(I, (-\pi_0)^a,X) \alpha(M(1)^{\perp},u(-\pi_0)^b)
			\\
			&\quad +q\beta_0(I, (-\pi_0)^a,X) \alpha(M(0)^{\perp},u(-\pi_0)^b) \\
			&\quad +\q^{2-m}X\alpha(I, \diag((-\pi_0)^{a-1},u(-\pi_0)^b),X)
		\end{align*}
		where $M(1)^{\perp}=\diag(\cH^{k-1},-(-\pi_0)^i,I)$ and
		$$M(0)^{\perp}=\diag(\cH^k,\underbrace{-(-\pi_0)^i,1,\dots,1,-\nu}_{m-1}). $$
		Continuing this process, we obtain
		\begin{align*}	
		&\alpha(I,T,X)\\
		&=\sum_{i=0}^{a}(\q^{2-m}X)^{a-i}
		  \cdot \left(\beta_1(I,(-\pi_0)^i,X)\alpha(M(1)^{\perp},u(-\pi_0)^b)+\q\beta_0(I,(-\pi_0)^i,X)\alpha(M(0)^{\perp},u(-\pi_0)^b)\right). \notag
		\end{align*}
		By the formulas in Proposition \ref{prop: ind for t} and Lemma \ref{lem:alpha(I_{a,b})}, the above is equal to
		\begin{align*}
			&\sum_{i=0}^{a}(\q^{2-m}X)^{a-i}(1-X)
			\\
			&\cdot \bigg(1+(q-1)\chi(I)q^{\frac{m-2}{2}}\sum_{s=-i}^{-1}(q^{(m-1)}(\q^2 X)^{-1})^s+\chi(I)\chi(T)q^{i+\frac{m}{2}}(\q^{2-m} X)^{b+1}\bigg)\\
			&+\q (\q^{2-m}X)^{a}X(1-\q^{-\frac{m}{2}}\chi(I))( 1+ \chi(I)\chi(T)\q^{-(b+1)(m-2)+\frac{m-2}{2}}X^{b+1})\\
			&+ \q\sum_{i=1}^{a}(\q^{2-m}X)^{a-i}X\left((1-q^{-(m-1)})+(q-1)\chi(I)q^{-(m-1)+\frac{m-2}{2}} \right)\\
			&\cdot \left(1+(q-1)\chi(I)q^{\frac{m-4}{2}}\sum_{s=-i}^{-1}(q^{(m-3)}X^{-1})^s+\chi(I)\chi(T)q^{i+\frac{m-2}{2}}(\q^{2-m}X)^{b+1}\right).
		\end{align*}
		Now the transformation
		\begin{align*}
			\sum_{i=0}^a \sum_{s=1}^{i}\q^s (\q^{2-m}X)^{a-i+s}=\sum_{d=1}^a\sum_{s=1}^{d}\q^s (\q^{2-m}X)^d
		\end{align*}
		and some calculation gives us the result we want.
	\end{proof}
	The case that  $I$ is anisotropic (i.e. when $m=2$ and $\chi(I)=-1$) can be computed similarly and is simpler. We omit the detail here. In particular, we may recover the following formula.
	\begin{proposition}\cite[Theorem 6.2(1)]{Shi2}\label{prop:localdensity, n=2}
		Assume $I=\diag(1,\nu)$, then
		\begin{align*}
			&\alpha(I, T, X)
			\\
			&= (1-X) (1+ \chi(I) + \q \chi(I)) \sum_{e=0}^\alpha (\q X)^e
			-\chi(T) \q^{\alpha+1} X^{\beta+1} (1-X) \sum_{e=0}^{\alpha} (\q^{-1}X)^e
			\\
			&\quad -\chi(T) (1+\q) (  X^{\alpha+ \beta +2}+\chi(I)\chi(T))
			+ (1+\chi(I)) \q^{\alpha+1} X^{\alpha+1} (1+\chi(T) X^{\beta-\alpha}).
		\end{align*}
	\end{proposition}

	Moreover, a similar computation yields the following, and we leave the detail to reader.
	\begin{proposition}Assume that $I$ is unimodular of odd rank $m\ge 3$.  Then
		\begin{align*}
			&\alpha(I,T,X)\\
			&=(1-X)\left(\sum_{i=0}^{a}(\q^{2-m} X)^{i}
			+\gamma_{o,1}(I,T,X)\right)
			  +(1-\q^{-(m-1)})X \left(\q \sum_{i=0}^{a-1}(\q^{2-m} X)^{i}+\gamma_{o,0}(S,T,X)\right)\\
			&\quad +\q X(\q^{2-m} X)^{a} \left(1+\chi(I)\chi(u_1)\q^{-\frac{m-1}{2}}\right)\left(1-\chi(I)\chi(u_1)\q^{(2-m)b-\frac{m-1}{2}}X^{b+1}\right),
		\end{align*}
		where $\gamma_{o,1}(I,T,X)$ equals
		\begin{align*}
			\chi(I)\chi(u_1)q^{\frac{m-1}{2}}\left(\sum_{d=a+1}^{a+b}(\q^{a+b+1-d}-1)(\q^{2-m}X)^d-\sum_{i=b+1}^{a+b+1}(\q^{2-m} X)^{i}\right),
		\end{align*}
		and $\gamma_{o,0}(I,T,X)$ equals
		\begin{align*}
			\chi(I)\chi(u_1)q^{\frac{m-1}{2}}\left(\sum_{d=a+1}^{a+b}(\q^{a+b+1-d}-\q)(\q^{2-m}X)^d-\sum_{i=b+1}^{a+b}(\q^{2-m} X)^{i}\right).
		\end{align*}
	\end{proposition}

	\section{Local density formula when $\mathrm{rank}(T)=3$}\label{sec: local density n=3}
	In this section, we always assume   $\mathrm{rank}(T)=3$  and $S=I_3^{-\chi(T)}$. The aim of this section is to compute $\pden(T)$ explicitly.   We treat the case $\mathrm{v}(T)\le -1$ in the first subsection. In the second subsection, we deal with the case when $T=\diag(1,T_2)$ for $T_2$ diagonal. In the last subsection, instead of $\pden(T)$, we compute $\pden(T)^{(2)}$ for $T$ of the form not covered by previous subsections.
	
	\subsection{$\pden(T)$ for $T$ with $\mathrm{v}(T)\le -1$}
	
 \begin{proposition}\label{prop: v(T)<0}
		If $\mathrm{v}(T)\le -1$, then $\mathrm{Int}(T)=\pden (T)=0$.
	\end{proposition}
	\begin{proof}
		If $\mathrm{v}(T)< -1$, then $\pden (T)=0$ since $\mathrm{v}(S^{[k]})\ge -1$ where  $S^{[k]}=S\obot \cH^k$. If $\mathrm{v}(T)= -1$, then $T$ is of the form $\diag(\cH,(u_3(-\pi_0)^c))$ with $\chi(T)=\chi(u_3)$. In this case, according to Corollary \ref{cor: L=H^i obot L_2}, Lemmas    \ref{lem:beta(H,L)} and \ref{lem: cancel beta_2}, we have
		\begin{align*}\alpha(S,T,X)&=(1-X)\alpha(S,(u_3(-\pi_0)^c),\q^2 X).
		\end{align*}
		Similarly, we have
		\begin{align*}\alpha(\cH^{3,1}_{\chi(T)},T)&=\beta(\cH^{3,1}_{\chi(T)},\cH)\alpha(u_3,(u_3(-\pi_0)^c))\\
			&=(1-\q^{-2})\alpha(u_3,(u_3(-\pi_0)^c)).\end{align*}
		Hence, 	applying Lemma \ref{lem:alpha(I_{a,b})} to $I_{0,0}=S$ where $I$ is of rank 1, we have
		\begin{align*}
			\pden(T)&=2\alpha(S,(u_3(-\pi_0)^c),\q^2)+\frac{\q^2}{1+\q}(1-\q^{-2})\alpha((u_3),(u_3(-\pi_0)^c))\\
			&=2(1+\chi(S)\chi(u_3)\q)+2(\q-1)\\
		    &=2(1-\q)+2(\q-1)\\
			&=0.
		\end{align*}
Here we are using the fact $\chi(S)\chi(T)=\chi(S)\chi(u_3)=-1$.
	\end{proof}

	\subsection{$\pden(T)$ for $T=\diag(1,T_2)$ with $T_2$ diagonal}\label{subsec: T=diag(1,T_2)}
	In this subsection, we assume $T=\diag(1,T_2)$, where $T_2=\diag(u_1(-\pi_0)^a,u_2(-\pi_0)^b)$ with $0\le a\le b$. Let $u=u_1u_2$. In addition, let $S=\diag(1,1,\nu)$ and $S_2=\diag(1,\nu)$. 	We compare $\pden(T)$ and $\pden(T_2)$ in this subsection.
	
	Recall that
	\begin{align*}
		\partial \Den(T)&=2\frac{\alpha'(S,T)}{\alpha(S,S)}+\frac{\q^2}{1+\q}\frac{\alpha(\cH^{3,1}_{\chi(T)},T)}{\alpha(S,S)}.
	\end{align*}	
	Moreover, according to \cite[Theorem 1.3]{Shi2} and \cite[Theorem 1.1]{HSY}, the analytic side in the case $n=2$ is
	\begin{align*}
		\partial \Den(T_2)&=2\frac{\alpha'(S_2,T_2)}{\alpha(S_2,S_2)}-\frac{2\q^2}{\q^2-1}\frac{\alpha(\cH,T_2)}{\alpha(S_2,S_2)}.
	\end{align*}
	
	\begin{proposition}
		\begin{align*}
			\pden(T)-\pden(T_2)&=\begin{cases}
				1+2\sum_{i=1}^a \q^i & \text{ if $\chi(T)=1$},\\
				1 & \text{ if $\chi(T)=-1$}.
			\end{cases}
		\end{align*}
	\end{proposition}
	\begin{proof}
		Proposition \ref{prop: ind for t} implies that $\alpha(S,T,X)$ equals
		\begin{align*}
			(1-X)\alpha( \diag(-1,S),T_2,\q^2 X)+q^{2} (1+q^{-1}\chi(S))X\alpha( S_2,T_2,X).
		\end{align*}
		Hence
		\begin{align}
			\alpha'(S,T)
			=\alpha( \diag(-1,S),T_2,\q^2 )+q^{2} (1+q^{-1}\chi(S))\alpha'( S_2,T_2).
		\end{align}
		According to Lemma \ref{lem: cancel beta_0}, one can check that $\alpha(S,S)=\beta(S,S)=2\q(\q^2-1)$, and  $\alpha(S_2,S_2)=2(\q-\chi(S_2))$. Then	
		\begin{align}
			\frac{\alpha'(S,T)}{\alpha(S,S)}-\frac{\alpha'(S_2,T_2)}{\alpha(S_2,S_2)}
			=\frac{\alpha( \diag(-1,S),T_2,\q^2 )}{\alpha(S,S)}.
		\end{align}
		Hence we just need to check that
		\begin{align*}
			&2\frac{\alpha( \diag(-1,S),T_2,\q^2 )}{\alpha(S,S)}+ \frac{\q^2}{1+\q^2}\frac{\alpha( \cH^{3,1}_{\chi(T)},T)}{\alpha(S,S)}+\frac{2\q^2}{\q^2-1}\frac{\alpha(\cH,T_2)}{\alpha(S_2,S_2)}\\
			&=\begin{cases}
				1+2\sum_{i=1}^a \q^i & \text{ if $\chi(T)=1$},\\
				1 & \text{ if $\chi(T)=-1$}.
			\end{cases}
		\end{align*}
		By Proposition \ref{prop: unimodular m even}, we may check that
		\begin{align}\label{eq: (-1,S)}
			&2\alpha(\diag(-1,S),T_2,\q^2)=
			\begin{cases}
				2(2\q^{a+2}-(\q+1)^2)(\q-1) & \text{ if  $\chi(T)=1$},\\
				2(\q-1)(\q^2-1)& \text{ if $\chi(T)=-1$}.
			\end{cases}
		\end{align}
		To compute $\frac{\q^2}{1+\q}\alpha(\cH^{3,1}_{\chi(T)},T)$,  we may choose $\cH^{3,1}_{\chi(T)}=\diag(\cH, 1)$ when $\chi(T)=1$.
		By Corollary \ref{cor: L=H^i obot L_2}, Proposition \ref{prop:localdensity, n=2}, and a direct calculation, we have
		\begin{align*}
			&\frac{\q^2}{1+\q}\frac{\alpha(\cH^{3,1}_{\chi(T)},T)}{\alpha(S,S)} =\frac{1}{2\q(\q^2-1)}\cdot\begin{cases}
				(\q-1)	\alpha(\diag(-1,1),T_2)+\frac{2\q^2}{\q-1} \alpha(\cH,T_2)& \text{ if $\chi(T)=1$}\\
				(\q-1)		\alpha(\diag(-1,-u),T_2)& \text{ if $\chi(T)=-1$}.
			\end{cases}
		\end{align*}
		Combining this with the formulas in \cite[Theorem 6.1]{HSY}, we have
		\begin{align}\label{eq: error comparison}
			&\frac{\q^2}{1+\q^2}\frac{\alpha( \cH^{3,1}_{\chi(T)},T)}{\alpha(S,S)}+\frac{2\q^2}{\q^2-1}\frac{\alpha(\cH,T_2)}{\alpha(S_2,S_2)} =\frac{1}{\q(\q^2-1)}\cdot\begin{cases}
				4\q^{a+2}-\q^2-2\q-1& \text{ if $\chi(T)=1$}\\
				(\q^2-1)& \text{ if $\chi(T)=-1$}.
			\end{cases}
		\end{align}
		Now a direct computation combined with \eqref{eq: (-1,S)} and \eqref{eq: error comparison} proves the proposition.
	\end{proof}

	\begin{corollary}\label{cor: pDen(T)-pDen(T_2)}
		Assume $L$ is a Hermitian lattice with Gram matrix $T$, then
		\begin{align}\label{eq: pDen(T)-pDen(T_2)}
			\partial \Den(T) -\partial \Den(T_2)=|\{\cV^0(L)\}|.
		\end{align}
	\end{corollary}
	\begin{proof}
		We can write $L=L^\flat \obot \Oo_F \bx$ where $q(\bx)=1$. If $L^\flat$ is non-split, then $|\{\cV^0(L)\}|=1$.
		
		If $L^\flat$ is split, then $|\{\cV^0(L)\}|= 1+2\sum_{i=1}^a \q^i$ since $\cL_3(L)$ can be identified with $\cL_{2,1}(L^\flat)$, which is a ball in $\cL_{2,1}$ centered at a vertex lattice of type $0$ with radius $a$ (see \cite{HSY} for more detail). Here $\cL_{2,1}$ is the Bruhat-Tits tree associated with $\cN^{\Kra}_{2,1}$ and $\cL_{2,1}(L^\flat)$ is the subtree of $\cL_{2,1}$  associated with $L^\flat$.
	\end{proof}

	\subsection{$\pden(T)^{(2)}$}\label{subsec: pden(T)^(2)}
	In this subsection,  we assume $T=\diag(T_2,u_3(-\pi_0)^c)$ with $\mathrm{v}(T_2)>0$, and compute $\pden(T)^{(2)}$. Recall that  $\pden(T)^{(2)}=\pden(L^\flat \obot \Oo_F \bx)^{(2)}$ where the Gram matrix of $L=L^\flat \obot \Oo_F \bx$ is $T$. We consider two cases separately in Propositions \ref{prop: a le b<c} and \ref{prop: pden T_a}.
	\begin{proposition}\label{prop: a le b<c}
		Let $T=\diag(u_1(-\pi_0)^a,u_2(-\pi_0)^b, u_3(-\pi_0)^c)$ where $0<a\le b \le c$. Then
		\begin{align*}
			\pden(T)^{(2)}=
			1+\chi(-u_2u_3)\q^{a}(\q^a-q^b)-\q^{a+b}.
		\end{align*}
	\end{proposition}
	\begin{proof}
		Recall that
		$$
		\pden(T)^{(2)}=\frac{1}{2\q(\q^2-1)}\left(2\beta'(S,T)^{(2)}+\frac{\q^2}{1+\q}\beta(\cH^{3,1}_{\chi(T)},T)^{(2)}\right).
		$$
		We compute $\beta'(S,T)^{(2)}$ first. According to Proposition \ref{prop: ind rank 2}, $\beta_0(S,T_2,X)=0$ and
		\begin{align*}
			\beta(S,T,X)^{(2)}&
			=\beta_2(S,T_2,X)\alpha(\diag(S,-T_2),u_3(-\pi_0)^c,\q^4 X)\\
		 &\quad +\q \beta_1(S,T_2,X)\alpha(\diag(-\nu,-T_2),u_3(-\pi_0)^c,\q^2 X)\\
			&=(1-X)(1-\q^2 X)\alpha(\diag(S,-T_2),u_3(-\pi_0)^c,\q^4 X) \\ &\quad+(\q+1)(\q^2-1)X(1-X) \alpha(\diag(-\nu,-T_2),u_3(-\pi_0)^c,\q^2 X).
		\end{align*}
		
		According to Lemma \ref{lem:alpha(I_{a,b})},
		\begin{align*}
			&\alpha(\diag(S, -T_2),u_3(-\pi_0)^c,\q^4 X)
			=1+\chi(S)\chi(u_1)(\q-1)\sum_{s=a+1}^b\q^{a+1}(\q X)^s+\chi(u_1u_2u_3\nu)\q^{a+b+2} X^{c+1},
		\end{align*}
		and
		\begin{align*}
			&\alpha(\diag(-\nu,-T_2),u_3(-\pi_0)^c,\q^2 X)
			=1+\chi(S)\chi(u_1)(\q-1)q^a\sum_{s=a+1}^b(\q X)^s+\chi(u_1u_2u_3\nu)\q^{a+b+1}X^{c+1}.
		\end{align*}
		The relation $\chi(u_1u_2u_3\nu)=\chi(S)\chi(T)=-1$ and a direct calculation show that
		\begin{align*}
			\beta'(S,T_2)^{(2)}&=	1+\chi(-u_2u_3)\q^{a}(\q^a-q^b)-\q^{a+b}.
		\end{align*}
		Finally,  $\beta(\cH^{3,1}_{\chi(T)},T)^{(2)}=0$ by Proposition \ref{prop: ind rank 2}.  The proposition is proved.
	\end{proof}

	\begin{proposition}\label{prop: pden T_a}
	Recall that $\cH_a=\begin{pmatrix}
	    0& \pi^a\\
	    (-\pi)^a &0\end{pmatrix}$.		Let $T= \diag(\cH_a, u_3(-\pi_0)^c)$ where $a$ is a positive odd integer and $c\ge 0$. Then
		\begin{align*}
			\pden(T)^{(2)}= \begin{cases}
				(1-\q^{a})& \text{ if $a\le 2c$},\\
				(1-\q^{2c+1})& \text{ if $a> 2c$}.
			\end{cases}
		\end{align*}
	\end{proposition}
	\begin{proof}
		Recall that $$\pden(T)^{(2)}=\frac{1}{2\q(\q^2-1)}\left(2\beta'(S,T)^{(2)}+\frac{\q^2}{1+\q}\beta(\cH^{3,1}_{\chi(T)},T)^{(2)}\right).$$
		We need to compute $\beta'(S,T)^{(2)}$ and $\beta(\cH^{3,1}_{\chi(T)},T)^{(2)}$.

		According to Proposition \ref{prop: ind rank 2}, $\beta_0(S,T_2,X)=0$ and
		\begin{align*}
			\beta(S,T,X)^{(2)}
			&=\beta_2(S,\cH_a,X)\alpha(\diag(S, \cH_a),u_3(-\pi_0)^c,\q^4 X)
			\\
			&\quad	+\q \beta_1(S,\cH_a,X)  \alpha(\diag( -\nu,\cH_a),u_3(-\pi_0)^c,\q^2 X)\\
			&=(1-X)\Big((1-\q^2 X)\alpha(\diag(S, \cH_a),u_3(-\pi_0)^c,\q^4 X)\\ &\quad\, +(\q+1)(\q^2-1)X \alpha(\diag( -\nu,\cH_a),u_3(-\pi_0)^c,\q^2 X)\Big).
		\end{align*}
		According to Lemma \ref{lem:alpha(S^{[k]} obot H_a, t)},
		\begin{align*}
			\alpha(\diag(S, \cH_a),u_3(-\pi_0)^c,\q^4 X)=
			\begin{cases}
				1+\chi(S)\chi(u_3)\q^{2+a} X^{c+1}  & \text{ if $a\le 2c$},\\
				1+ \chi(S)\chi(u_3)\q^{2c+3} X^{c+1}   & \text{ if $a> 2c$},
			\end{cases}
		\end{align*}
		and
		\begin{align*}
			\alpha(\diag( -\nu,\cH_a),u_3(-\pi_0)^c,\q^2 X)=\begin{cases}
				1+\chi(S)\chi(u_3)\q^{1+a} X^{c+1} & \text{ if $a\le 2c$},\\
				1+ \chi(S)\chi(u_3)\q^{2c+2} X^{c+1}   & \text{ if $a> 2c$}.
			\end{cases}
		\end{align*}
		A short computation shows that
		\begin{align*}
			\beta'(S,T)^{(2)}=\q(\q^2-1) \cdot \begin{cases}
				1+\chi(S)\chi(u_3)\q^{a} & \text{ if $a\le 2c$},\\
				1+\chi(S)\chi(u_3)\q^{2c+1} & \text{ if $a> 2c$}.
			\end{cases}
		\end{align*}
		Notice that $\chi(S)\chi(u_3)=\chi(S)\chi(T)=-1$. Finally,  $\beta(\cH^{3,1}_{\chi(T)},T)^{(2)}=0$ by Proposition \ref{prop: ind rank 2}. The proposition is proved.
	\end{proof}

	\part{Proof of the main theorem}\label{part:III}
	\section{Reduced locus of special cycle}\label{sec:reduced locus n=3}
     As remarked in \S \ref{sec: RZ space}, results of \cite{RTW} extend to the category of strict formal $\Oo_{F_0}$-modules using relative Dieudonn\'e theory.
	\subsection{The Bruhat-Tits building for $n=3$} From now on we assume $n=3$ and $\cL=\cL_3$ as in Section \ref{subsec:bruhat tits in general}.
	\begin{lemma} \label{lem:B-T building}\hfill
		\begin{enumerate}
			\item For every $\Lambda_2\in \cV^2$, $\cN_{\Lambda_2}$ is isomorphic to the projective line $\bP^1$ over $k$. Its $\q+1$ rational points correspond to all $\Lambda_0\in \cV^0$ contained in $\Lambda_2$.
			\item Every $\Lambda_0\in \cV^0$ is contained in $\q+1$ type $2$ lattices. In other words, there are $\q+1$ projective lines in $(\cN_3^\Pap)_{\red}$ passing through the superspecial point  $\cN_{\Lambda_0}(k)$. Moreover
			\begin{equation}\label{eq:lambdasharpintersection}
				\bigcap_{\Lambda_2\in \cV_2,\, \Lambda_0 \subset \Lambda_2} \Lambda_2^\sharp=\pi \Lambda_0.
			\end{equation}
		\end{enumerate}
	\end{lemma}
	\begin{proof}
		Suppose $z\in \cN(k)$ and $M\coloneqq M(z)\subset N$ is defined as in Proposition \ref{prop:k points of N}. Since $n=3$, by \cite[Proposition 4.1]{RTW} we have $\Lambda(M)\otimes_{\Oo_F} \Oo_{\breve F}=M+\tau(M)$.
		
		\noindent Proof of (1):
		Suppose $z\in \cN_{\Lambda_2}(k)$, i.e. $M\subset \Lambda_2$.
		
		If $M=\tau(M)$, then $M=\Lambda_0\otimes_{\Oo_F} \Oo_{\breve F}$ for some  $\Lambda_0\in \cV^0$ contained in $\Lambda_2$.
		
		If $M\neq \tau(M)$, then by taking the dual of $M\subset \Lambda_2\otimes_{\Oo_F} \Oo_{\breve F}$ we have the following sequence of inclusions
		\begin{equation}\label{eq:MtauMandLambda_2}
			(\Lambda_2\otimes_{\Oo_F} \Oo_{\breve F})^\sharp  \stackrel{1}{\subset} M \stackrel{1}{\subset} M+\tau(M)= \Lambda_2\otimes_{\Oo_F} \Oo_{\breve F}.
		\end{equation}
		In both cases the class of $M$ in $\Lambda_2\otimes_{\Oo_F} \Oo_{\breve F}/(\Lambda_2\otimes_{\Oo_F} \Oo_{\breve F})^\sharp\cong k^2$ is a line. This finishes the proof of (1).
		
		\noindent
		Proof of (2): For each $\Lambda_0\in \cV^0$ we just need to count the number of lattices $\Lambda_2\in \cV^2$ that contains $\Lambda_0$. We have the following sequence of inclusions
		\[\pi \Lambda_0 \stackrel{2}{\subset} \Lambda_2^\sharp \stackrel{1}{\subset} \Lambda_0\stackrel{1}{\subset}\Lambda_2.\]
		With respect to the quadratic form $(\, ,)\pmod{\pi}$ on $\Lambda_0/\pi \Lambda_0$, the dual lattice $\Lambda_2^\sharp$ corresponds to the $2$-dimensional subspaces $U\coloneqq \Lambda_2^\sharp/\pi \Lambda_0$ in $\Lambda_0/\pi \Lambda_0$ such that $U^\bot \stackrel{1}{\subset} U$. Thus, we just need to count the number of isotropic lines $U^\bot$. Assume that $\{e_1,e_2,e_3\}$ is a basis of $\Lambda_0/\pi\Lambda_0$ whose Gram matrix with respect to the quadratic form $(\, ,)_\bX \pmod \pi$ is
		\[\left(\begin{array}{ccc}
			& 1 &\\
			1 & & \\
			& & \epsilon
		\end{array}\right).\]
		It is easy to see that the isotropic lines are $\mathrm{Span}\{e_1\}$, $\mathrm{Span}\{e_2\}$ and $\mathrm{Span}\{e_1-\frac{\epsilon a^2}{2}e_2+a e_3\}$ ($a\in \F_\q^\times$). Finally,
		equation \eqref{eq:lambdasharpintersection} can be checked directly using this basis.
	\end{proof}
	It is well known that $\cL$ is a tree, see for example \cite[Section 3 of Chapter VI]{brown1989buildings}. More specifically, the vertices of $\cL$ correspond to vertex lattices of type $2$ or $0$. There is an edge between $\Lambda\in \cV^2$ and $\Lambda_0\in \cV^0$ if $\Lambda_0\subset \Lambda$. We give each edge length $\frac{1}{2}$. This defines a metric $d(\, ,)$ on $\cL$. Recall that we have defined $\cL(L)$ in \eqref{eq:cLL}.
	Then the boundary  of $\cL(L)$ is the set
	\begin{equation}\label{eq:cBL}
		\mathcal{B}(L)=\{\Lambda\in \cV^0(L)\mid \exists\, \Lambda_2\in\cV^2\text{ such that }\Lambda\subset\Lambda_2,\Lambda_2\notin\cL(L)\}.
	\end{equation}
	
	Recall we have the isomorphism $b:\bV\rightarrow C$ defined in \eqref{eq: C V isomorphism}.
	Recall from \cite{KR3} or \cite{HSY} that the vertices of $\cL_{2,1}$ correspond to vertex lattices of type $2$, and an edge corresponds to a vertex lattice of type $0$. Each vertex of $\cL_{2,1}$ is contained in $\q+1$ edges and each edge connects exactly two vertices.
	For $\bx\in \bV$ with $\val(\bx)=0$ and $\mathrm{Span}_F\{\bx\}^{\perp}$ split, recall that $\cL_{2,1}$ is the Bruhat-Tits tree of $\cZ^\Pap(\bx)\cong \cN^\Pap_{2,1}$.
	Then $\bx$ determines an embedding $\cL_{2,1}\hookrightarrow \cL$ defined as follows. First we send each vertex of $\cL_{2,1}$ corresponding to a vertex lattice $\Lambda\subset\mathrm{Span}_F\{b(\bx)\}^{\perp}$ of type $2$ to the vertex of $\cL$ corresponding to the type $2$ lattice $\Lambda\obot\spa \{b(\bx)\}$. An edge of $\cL_{2,1}$ corresponding to a type 0 lattice $\Lambda_0\subset \spa_F\{b(\bx)\}^{\perp}$ is broken into two pieces evenly and sent to the union of the two edges in $\cL$ joining the two vertices corresponding to $\Lambda\obot\spa \{b(\bx)\}$ and $\Lambda'\obot\spa \{b(\bx)\}$ where $\Lambda$ and $\Lambda'$ are the two type $2$ lattices containing $\Lambda_0$.

	\subsection{Rank $1$ case}
	\begin{lemma}\label{lem:cNandcV_Lambda}
		A point $z\in \cN_3^\Pap(k)$ is in $\cZ^\Pap(\bx)(k)$ if and only if $b(\bx)\in M(z)$.
		\begin{enumerate}
			\item Assume $\Lambda_0\in \cV^0$, then the superspecial point $\cN_{\Lambda_0}(k)$  is in $\cZ^\Pap(\bx)(k)$ if and only if $b(\bx)\in \Lambda_0$.
			\item Assume $\Lambda_2\in \cV^2$, then
			\begin{align*}\cZ^\Pap(\bx)(k) \cap \cN_{\Lambda_2}(k)=\begin{cases}
					\cN_{\Lambda_2}(k) & \text{ if } b(\bx) \in \Lambda_2^\sharp, \\
					\text{a superspecial point in } \cN_{\Lambda_2}(k)  & \text{ if } b(\bx)\in \Lambda_2 \backslash \Lambda_2^\sharp,\\
					\emptyset & \text{ if } b(\bx)\notin \Lambda_2.
			\end{cases}\end{align*}
		\end{enumerate}
	\end{lemma}
	\begin{proof}
		By Dieudonn\'e theory, $z\in \cZ^\Pap(\bx)(k)$ if and only if $\bx(M(\bY))\subset M(z)$ if and only if $b(\bx)\in M(z)$ since $e$ is a generator of $M(\bY)$. For $z=\cN_{\Lambda_0}(k)$ where $\Lambda_0\in\cV^0$, we have $M(z)=\Lambda_0\otimes_{\Oo_F} \Oo_{\breve F}$. Hence (1) immediately follows.
		
		Now we proceed to prove (2).
		If $b(\bx)\in \Lambda^\sharp$, then \eqref{eq:MtauMandLambda_2} tells us that $z\in \cZ^\Pap(\bx)(k)$ for any $z\in\cN^\circ_{\Lambda_2}(k)$. The  fact that $\Lambda_2^\sharp\subset\Lambda_0$ for any $\Lambda_0\in \cL^0$ contained in $\Lambda_2$ implies that $\cN_{\Lambda_0}(k)\in\cZ^\Pap(\bx)(k)$. So $\cN_{\Lambda_2}(k)\subset \cZ^\Pap(\bx)(k)$.
		
		If $b(\bx)\in \Lambda_2\backslash \Lambda_2^\sharp$, then $\Lambda_0\coloneqq \Lambda^\sharp+\spa\{b(\bx)\}$ is a type $0$ lattice contained in $\Lambda_2$ and $\cN_{\Lambda_0}(k)\in \cZ^\Pap(\bx)(k)$. On the other hand,  since $\tau(\Lambda_0\otimes_{\Oo_F} \Oo_{\breve F})=\Lambda_0\otimes_{\Oo_F} \Oo_{\breve F}$, equation \eqref{eq:MtauMandLambda_2} tells us that $\cZ^\Pap(\bx)$ does not contain any point in $\cN^\circ_{\Lambda_2}(k)$.

		If $b(\bx)\notin \Lambda_2$, then $b(\bx)\notin M(z)$ for any $z\in \cN_{\Lambda_2}(k)$, hence $\cZ^\Pap(\bx)(k) \cap \cN_{\Lambda_2}(k)=\emptyset$.
	\end{proof}

	\begin{corollary}\label{cor:cV_Lambda contained in Z(pi x)}
		Let $L \subset \bV$. Assume $z\in \cZ^\Pap(L)(k)$ and $z \in \cN_{\Lambda}(k)$ where $\Lambda\in \cV^2$. Then $\cN_{\Lambda}\subset \cZ^\Pap(\pi L).$
	\end{corollary}
	
	\begin{corollary}\label{cor: cV Lambda subset Z(bx/pi) or not}
		Assume $\bx \in \bV$ and $\mathrm{v}(\bx)>0$. Assume $\cN_{\Lambda}\subset \cZ^\Pap(\bx)_{\red}$ where $\Lambda\in\cV^2$, then either $\cN_{\Lambda}\subset \cZ^\Pap(\frac{1}{\pi}\bx)_{\red}$ or $\cN_{\Lambda}\cap \cZ^\Pap(\frac{1}{\pi}\bx)_{\red}$ is a unique superspecial point.
	\end{corollary}
	
	\begin{lemma}\label{lem:cZisconnected}
		For $L\subset \bV$ a lattice of arbitrary rank, $\cZ^\Pap(L)_{\red}$ is connected.
	\end{lemma}
	\begin{proof}
		Suppose $\cZ^\Pap(L)_{\red}$ has  two different connected components $U_1$ and $U_2$. Since $\mathrm{SU}(\bV)$ acts transitively on $\cL$, we can find a $\bx\in\bV$ such that $\cZ^\Pap(\bx)\cong \cN^\Pap_{2,1}$ (i.e. $ \{\bx\}^\bot$ is split) and $\cZ^\Pap(\bx)_{\red}\cap U_i\neq \emptyset$ for $i=1,2$. Hence the reduced locus of
		\[\cZ^\Pap(L\oplus\spa\{\bx\})\cong \cZ^\Pap_{2,1}(L')\]
		is not connected where $L'$ is the orthogonal projection of $L$ onto $\{\bx\}^\bot$. This contradicts Corollaries 3.13, 3.15 and Lemma 3.16 of \cite{HSY}.
	\end{proof}
	
	Recall that for a lattice $L\subset \bV$ (respectively, $\bx\in \bV$), we have defined $\cV(L)$ and $\cL(L)$ (respectively, $\cV(\bx)$ and $\cL(\bx)$) in Section \ref{subsec:bruhat tits in general}.
	\begin{proposition}\label{prop:n=3reducedlocusI}
		Assume that $\bx\in \bV$ such that $h(\bx,\bx)\neq 0$. Then  we have
		\[\cZ^\Pap(\bx)_{\red}=\bigcup_{\Lambda\in \cV(\bx)} \cN_\Lambda,\]
		where $\cV(\bx)$ is given  as follows.
		\begin{enumerate}
			\item When $\mathrm{v}(\bx)=0$ and $\spa_F\{\bx\}^{\perp}$ is non-split,  there is a unique vertex lattice $\Lambda_\bx\in \cV^0$ containing $b(\bx)$. In this case $\cV(\bx)=\{\Lambda_\bx\}$.
			\item When  $\mathrm{v}(\bx)=d$ and $\spa_F\{\bx\}^{\perp}$ is non-split, we  have
			\[\cV(\bx)=\{\Lambda\in \cV\mid d(\Lambda, \Lambda_{\bx/\pi^d})\le d \}\]
			where $\Lambda_{\bx/\pi^d}$ is as in (1).
			\item When $\mathrm{v}(\bx)=0$ and $\spa_F\{\bx\}^{\perp}$ is split, $\cL(\bx)$ is the tree $\cL_{2,1}$.
			\item When
			$\mathrm{v}(\bx)=d$ and $\spa_F\{\bx\}^{\perp}$ is split, we have
			\[\cV(\bx)=\{\Lambda\in \cV\mid d(\Lambda, \cL(\bx/\pi^d))\le d \}\]
			where $\cL(\bx/\pi^d)$ is as in (3).
			\item When $h(\bx, \bx) \notin \Oo_{F_0}$, $\cV(\bx)$ is empty.
		\end{enumerate}
	\end{proposition}
	\begin{proof}
		Proof of (1): This is a direct consequence of Proposition \ref{prop:cZxunimodular} and the fact that $\cN_{2,-1}^\Pap$ has only one reduced point, see \cite[Section 2]{Shi2} or \cite[Section 8]{RSZ}. Alternatively since $\spa_F\{b(\bx)\}^{\perp}$ is non-split of dimension $2$, it contains a unique self-dual lattice $\Lambda'$, then $\Lambda_\bx\coloneqq \spa\{b(\bx)\}\oplus \Lambda'$ is the unique type $0$ lattice containing $b(\bx)$.
		
		Proof of (3): Applying Proposition \ref{prop:cZxunimodular},we see that  $\cZ^\Pap(\bx)\cong\cN_{2,1}^\Pap$ is the Drinfeld $p$-adic half space, see \cite{KR3} and \cite{HSY}. The required properties of $\cL(\bx)$ and $\cV(\bx)$ follow.
		
		Proof of (2): We prove this by induction. The case $d=0$ is just (1). Now we assume $d>0$ and that the statement holds for $d-1$, i.e.
		\[\cV(\bx/\pi)=\{\Lambda\in \cV\mid d(\Lambda, \Lambda_{\bx/\pi^d})\le d-1 \}.\]
		Then applying   Corollary \ref{cor:cV_Lambda contained in Z(pi x)} to the lattice $L=\spa\{\bx/\pi\}$ we have
		$$
		\bigcup_{\Lambda\in \cV^2,\ d(\Lambda, \Lambda_{\bx/\pi^d})\le d} \cN_{\Lambda}\subset \cZ^\Pap(\bx)_{\red}.
		$$
		Corollary \ref{cor: cV Lambda subset Z(bx/pi) or not} and the induction hypothesis imply that every $\Lambda_2\in \cV^2(\bx)$ satisfies $d(\Lambda, \Lambda_{\bx/\pi^d})\le d$. By Lemma \ref{lem:cZisconnected} there is no isolated $\Lambda_0\in \cV^0(\bx)$, i.e. every $\Lambda_0\in \cV^0(\bx)$ is contained in some $\Lambda_2\in \cV^2(\bx)$  if $\mathrm{v}(\bx)>0$. This finishes the proof of (2).
		
		Similarly we can prove (4) by an induction on $d$, the case $d=0$ is just (3).
		
		(5) follows directly from Lemma \ref{lem:cNandcV_Lambda}.
	\end{proof}
	
	\subsection{Rank $2$ case}
	\begin{proposition}\label{prop:n=3reducedlocusII}
		Assume that $L^\flat=\spa\{\bx_1,\bx_2\}\subset \bV$ is integral of rank $2$. Then
		\[\cZ^\Pap(L^\flat)_{\red}=\bigcup_{\Lambda\in \cV(L^\flat)} \cN_\Lambda\]
		is a finite union, where  $\cV(L^\flat)$ is the set of vertices of the tree $\cL(L^\flat)$ described as follows.
		\begin{enumerate}
			\item Assume $L^\flat\approx \cH_{2a+1}$ for some $a\in \Z_{\ge 0}$. Then
			$\cL(L^\flat)$ is a ball centered at a vertex lattice of type $2$ with radius $\frac{2a+1}{2}$.
			\item Assume $L^\flat=\spa\{\pi^a\bx_1,\pi^a\bx_2\}$ where $\mathrm{v}(\bx_1)=0$, $\mathrm{v}(\bx_2)\geq 0$ and $\spa_F\{\bx_1\}^\bot$ is non-split. Then $\cL(L^\flat)$ is a ball centered at a vertex lattice of type $0$ with radius $a$.
			\item Assume $L^\flat=\spa\{\pi^a\bx_1,\pi^{a+r}\bx_2\}$ where $\bx_1\bot\bx_2$, $\mathrm{v}(\bx_1)=\mathrm{v}(\bx_2)=0$, $r\geq 0$ and $\spa_F\{\bx_1\}^\bot$ is split. Then
			\[\cL(L^\flat)=\{\Lambda\in \cV\mid  d(\Lambda, \cL(\pi^{-a}L^\flat))\le a\},\]
			where
			\[\cL(\pi^{-a}L^\flat)=\{\Lambda\in \cL(\bx_1)\mid  d(\Lambda, \Lambda_0)\le r\},\]
			$\cL(\bx_1)$ is described in (3) of Proposition \ref{prop:n=3reducedlocusI} and $\Lambda_0$ is the unique type $0$ vertex lattice containing $\{\bx_1,\bx_2\}$.
		\end{enumerate}
	\end{proposition}
	\begin{proof}
		As in the proof of Proposition \ref{lem:cNandcV_Lambda}, for a $\Lambda \in \cV$, $\cN_\Lambda \subset \cZ^\Pap(L^\flat)_{\red} $ if and only if $\Lambda^\sharp$ contains $b(\bx_1),b(\bx_2)$.
		
		We first prove (1) when $a=0$. Suppose  $\Lambda\in \cV^2(L^\flat)$. Extend $\{b(\bx_1),b(\bx_2)\}$ to a basis $\{b(\bx_1),b(\bx_2),b_3\}$ of $\bV$ with Gram matrix $\cH_1\oplus \{-\epsilon\}$. Choose a basis $\{v_1,v_2,v_3\}$ of $\Lambda^\sharp$ with the same Gram matrix $\cH_1\oplus \{-\epsilon\}$. Then $b(\bx_i) \in  \Lambda^\sharp$ ($i=1,2$) by Lemma \ref{lem:cNandcV_Lambda} and
		\[b(\bx_i)=a_{i1} v_1+a_{i2} v_2+a_{i3} v_3\]
		where $a_{ij} \in \Oo_F$ ($j=1,2,3$). The fact that $(b(\bx_i),b(\bx_j))_{1\leq i,j \leq 2}=T$ implies $a_{i3}\in \pi \Oo_F$ for $i=1,2$ and $(a_{ij})_{1\leq i,j \leq 2}$ is in $\mathrm{GL}_2(\Oo_F)$. This guarantees that $L^\flat$ is a direct summand of $\Lambda^\sharp$ by Gram-Schmidt process.
		Hence $\Lambda^\sharp$ is, in fact, the lattice $\spa_{\Oo_F}\{b(\bx_1),b(\bx_2),b_3\}$. The fact that all $\Lambda_0\in \cV^0(L^\flat)$ are in $\Lambda$ follows from Lemma \ref{lem:cZisconnected}.

		When $a=0$, (2) follows from the fact that $\cZ^\Pap(\bx_1)=\cN^\Pap_{2,-1}$ (by Proposition \ref{prop:cZxunimodular}) and $\cZ^\Pap(L^\flat)_{\red}=\cZ^\Pap(\bx_1)_{\red}$ is a unique superspecial point.
		Similarly when $a=0$, (3) follows from the fact that $\cZ^\Pap(\bx_1)=\cN^\Pap_{2,1}$ and \cite[Corollary 3.13]{HSY}.

		Now we prove (1), (2) and (3) for general $a$. First of all, $\cL(L^\flat)=\cL(\pi^a\bx_1)\cap \cL(\pi^a\bx_2)$ by definition. By Corollary \ref{cor:cV_Lambda contained in Z(pi x)} we have
		\[\{\Lambda\in \cV \mid d(\Lambda,\cL(\pi^{-a}L^\flat))\le a\}\subset \cL(\pi^a \bx_1)\cap \cL(\pi^a \bx_2).\]
		
		Notice that for a sub-tree $\cL'$ of a tree $\cL$ and a vertex $x\in \cL\setminus \cL'$, there is a unique geodesic segment joining $x$ with $\cL'$. Given $\Lambda\in \cL(L^\flat)=\cL(\pi^a \bx_1)\cap \cL(\pi^a \bx_2)$, let $\gamma$ be the unique geodesic segment joining $\Lambda$ with $\cL(\pi^{-a} L^\flat)$. Assume that $\gamma$ intersects $\cL(\pi^{-a}L^\flat)$ at $\Lambda(L^\flat)$. Since $\cL(\pi^{-a} L^\flat)=\cL(\bx_1)\cap \cL(\bx_2)$, $\gamma$ necessarily intersects both $\cL(\bx_1)$ and $\cL(\bx_2)$. Without loss of generality we assume that $\gamma$ intersects $\cL(\bx_1)$ at $\Lambda(\bx_1)$ first. Hence the intersection of $\gamma$ with $\cL(\bx_2)$ is $\Lambda(L^\flat)$ and
		\[d(\Lambda,\Lambda(\bx_1))=d(\Lambda,\cL(\bx_1))\leq d(\Lambda,\Lambda(L^\flat))=d(\Lambda,\cL(\bx_2)).\]
		Now by Proposition \ref{prop:n=3reducedlocusI}, we have
		\[d(\Lambda,\cL(\bx_1))\leq a, \ d(\Lambda,\cL(\bx_2))\leq a.\]
		Hence $ d(\Lambda,\cL(\pi^{-a}L^\flat))\leq a$. This shows that
		\[\{\Lambda\in \cV\mid d(\Lambda,\cL(\pi^{-a}L^\flat))\le a\}= \cL(\pi^a \bx_1)\cap \cL(\pi^a \bx_2).\]
		The general case of  (1), (2) and (3) follows from the above equation and the case $a=0$.

		Notice that (1), (2) and (3) have covered all possibilities of $L^\flat$ due to the classification of Hermitian lattices. Notice that in every case $\cV(L^\flat)$ is finite. This finishes the proof of the proposition.
	\end{proof}

	\begin{definition}\label{def:skeleton}
		Assume that $L^\flat$ is an integral lattice of rank $2$ in $\bV$. Define  $\mathcal{S}(L^\flat)$, the skeleton of $\cL(L^\flat)$, as follows.
		If the fundamental invariant of $L^\flat$ is $(2a,b)$ ($b\geq 2a$), define $\mathcal{S}(L^\flat)\coloneqq \cL(\pi^{-a}L^\flat)$. If the fundamental invariant of $L^\flat$ is $(2a+1,2a+1)$, define $\mathcal{S}(L^\flat)\coloneqq \emptyset$.
	\end{definition}
	\begin{remark}
		The skeleton $\mathcal{S}(L^\flat)$ is isomorphic to a ball in the Bruhat-Tits tree of $\cN^\Pap_{2,\pm 1}$.
	\end{remark}
	\begin{corollary}\label{cor:inductiononbruhattits}
		For each $\Lambda_2\in \cV^2(L^\flat)$ not on the skeleton $\mathcal{S}(L^\flat)$, one can find $\Lambda_0\in\cV^0(L^\flat)$ such that $\Lambda_2$ has the largest distance to the boundary $\mathcal{B}(L^\flat)$ of $\cL(L^\flat)$ among all type $2$ lattices in $\cV^2(L^\flat)$ containing $\Lambda_0$.
	\end{corollary}
	\begin{proof}
		Assume the fundamental invariant of $L^\flat$ is $(2a,b)$ or $(2a+1,2a+1)$.
		Define $M^\flat\coloneqq \pi^{-a}L^\flat$. Let $b$ be the unique integer such that $\Lambda_2\in \cL(\pi^b M^\flat)\setminus \cL(\pi^{b-1} M^\flat)$. Choose any $\Lambda_0\in \mathcal{B}(\pi^b M^\flat)$ such that $\Lambda_0\subset \Lambda$. Then by Proposition \ref{prop:n=3reducedlocusII}, $\Lambda_0$ satisfies the assumption of the corollary.
	\end{proof}

	\subsection{The Kr\"amer model}
	For $\Lambda\in \cV^2$, let $\tN_\Lambda$ be the strict transform of $\cN_\Lambda$ under the blow-up $\cN^\Kra\rightarrow \cN^\Pap$. Since the strict transform of a regular curve along any of its closed point is an isomorphism, we know $\tN_\Lambda\cong \bP^1$.
	\begin{lemma}\label{lem:tVdonotintersect}
		For $\Lambda \ne \Lambda'\in \cV^2$,  $\tN_\Lambda$ and $\tN_{\Lambda'}$ do not intersect.
	\end{lemma}
	\begin{proof}
		If $\cN_{\Lambda}$ and $\cN_{\Lambda'}$ do not intersect in $\cN^\Pap$, then obviously $\tN_\Lambda$ and $\tN_{\Lambda'}$ do not intersect. Without loss of generality we can assume $\Lambda=\spa\{e_1,e_2,e_3\}$ and $\Lambda'=\spa\{\pi^{-1}e_1,\pi e_2,e_3\}$ where the Gram matrix of $\{e_1,e_2,e_3\}$ is $\diag(\cH, \epsilon)$.
		Take $\bx_0=e_3$. Then by Proposition \ref{prop:n=3reducedlocusII}, both $\tN_{\Lambda}$ and $\tN_{\Lambda'}$ are in $\tZ(\bx_0)\cong \cN^\Kra_{2,1}$. Now by \cite[Lemma 5.3]{HSY}, $\tN_{\Lambda}$ and $\tN_{\Lambda'}$ do not intersect.
	\end{proof}
	
	\begin{lemma}\label{lem:tVintersectExc} Let  $\Lambda\in \cV^2$ and $\Lambda_0\in \cV^0$. When $\Lambda_0 \subset \Lambda$,   $\tN_\Lambda$ intersects properly with $\Exc_{\Lambda_0}$ and 	
		\begin{equation}
			\chi(\cN^\Kra,\Oo_{\tN_\Lambda}\otimes \Oo_{\Exc_{\Lambda_0}})=1.
		\end{equation}
		When  $\Lambda_0$ is not contained in $\Lambda$,  $\tN_\Lambda$  does not intersect with $\Exc_{\Lambda_0}$.
	\end{lemma}
	\begin{proof}
		First assume $\Lambda_0\subset \Lambda$. Since $\tN_\Lambda$ is a strict transformation of a curve, it intersects the exceptional divisor properly.
		Let $\bx_0$ be as in the proof of Lemma \ref{lem:tVdonotintersect}. Then $\tN_{\Lambda}$ is in $\tZ(\bx_0)\cong\cN^\Kra_{2,1}$.
		\begin{align*}
			\chi(\cN^\Kra,\Oo_{\tN_\Lambda}\otimes \Oo_{\Exc_{\Lambda_0}})
			&=\chi(\cN^\Kra,\Oo_{\tN_\Lambda}\otimes_{\Oo_{\tZ(\bx_0)}} \Oo_{\tZ(\bx_0)}\otimes\Oo_{\Exc_{\Lambda_0}})\\
			&=\chi(\tZ(\bx_0),\Oo_{\tN_\Lambda}\otimes_{\Oo_{\tZ(\bx_0)}} \Oo_{\Exc'}).
		\end{align*}
		Here $\Exc'\cong \bP^1_k$ is the exceptional divisor on $\tZ(\bx_0)$ corresponding to the rank $2$ self-dual lattice
		\[\Lambda'=\{v\in \Lambda_0\mid v\bot \bx_0\}.\]
		By \cite[Lemma 5.2(a)]{HSY}, we know $\chi(\tZ(\bx_0),\Oo_{\tN_\Lambda}\otimes_{\Oo_{\tZ(\bx_0)}} \Oo_{\Exc'})=1$.  When $\Lambda_0$ is not contained in $\Lambda$, the superspecial point $\cN_{\Lambda_0}(k)$ is not contained in $\cN_\Lambda$, hence  $\tN_\Lambda$  does not intersect with $\Exc_{\Lambda_0}$.
	\end{proof}

	\section{Intersection of vertical components and special divisors}\label{sec:intersection of vertical and special}
	In this section we study the intersection of $\tN_{\Lambda}$ and special divisors. The main result is Theorem \ref{thm: Int Lambda}. To proceed we first study the decomposition of ${}^\bL \cZ^\Kra(L^\flat)$ when $\mathrm{v}(L^\flat)=0$. Since $n$ is odd, we can without loss of generality assume that $\chi(\bV)=\chi(C)=1$. In the rest of the paper, we identify $\bV$ with $C$ by the isomorphism $b$ defined in \eqref{eq: C V isomorphism}.
	\subsection{Decomposition of $\LZ(L^\flat)$}
	Let $L^\flat=\spa\{\bx_1,\bx_2\}$ where $\bx_1,\bx_2\in \bV$ are linearly independent and the Hermitian form restricted to $L$ is non-degenerate.
	\begin{lemma}\label{lem:cZisinF^2}
	We have that	${}^\bL \cZ^\Kra(L^\flat)=[\Oo_{\cZ^\Kra(\bx_1)}\otimes^\bL \Oo_{\cZ^\Kra(\bx_2)}]\in K_0(\cN^\Kra)$ is in fact in $\rF^{2} K_0(\cN^\Kra)$. Moreover we have the decomposition in $\Gr^2 K_0^{\cZ^\Kra(L^\flat)}(\cN^\Kra)$
		\begin{equation}
			{}^\bL \cZ^\Kra(L^\flat)=\cZ^\Kra(L^\flat)_h+{}^\bL \cZ^\Kra(L^\flat)_v.
		\end{equation}
		where $\cZ^\Kra(L^\flat)_h$ is described in Theorem \ref{thm:horizontalpart} and ${}^\bL \cZ^\Kra(L)_v\in \Gr^2 K_0^{\cZ^\Kra(L^\flat)_v}(\cN^\Kra)$.
	\end{lemma}
	\begin{proof}
		By Lemma \ref{lem:cZNoetherian}, $\cZ^\Kra(L^\flat)$ is Noetherian and has a decomposition
		\[\cZ^\Kra(L^\flat)=\cZ^\Kra(L^\flat)_h\cup \cZ^\Kra(L^\flat)_v.\]
		Expressing $\cZ(\bx_i)$ ($i=1,2$) as in \eqref{eq:m_Lambda} and applying
		Propositions \ref{prop:multiplicityofExc}, \ref{prop:tZintersectExc} and Lemma \ref{lem:Excselfintersect}, ${}^\bL \cZ^\Kra(L^\flat)$ equals
		\[[\Oo_{\tZ(\bx_1)}\otimes^\bL\Oo_{\tZ(\bx_2)}]+\sum_{\Lambda_0\in \cV^0(L^\flat)} (2m_{\Lambda_0}(\bx_1)m_{\Lambda_0}(\bx_2)+m_{\Lambda_0}(\bx_1)+m_{\Lambda_0}(\bx_2)) H_{\Lambda_0}.\]
		$\cZ^\Kra(L^\flat)_h$ is contained in $\tZ(\bx_1)\cap\tZ(\bx_2)$ and has dimension $1$ by Theorem \ref{thm:horizontalpart}. $\tZ(\bx_1)\cap\tZ(\bx_2)_v$ also has dimension $1$ as it is supported on the reduced locus of $\cN^\Kra$ by Lemma \ref{lem:cZ_vsupportedoncN_red} and does not contain any exceptional divisor $\Exc_{\Lambda_0}$. Hence
		\begin{equation}\label{eq:tZ1intersecttZ2}
			[\Oo_{\tZ(\bx_1)}\otimes^\bL\Oo_{\tZ(\bx_2)}]=[\Oo_{\tZ(\bx_1)\cap\tZ(\bx_2)}]\in \rF^{2} K_0(\cN^\Kra),
		\end{equation}
		see for example \cite[Lemma B.2]{zhang2021AFL}.
		Hence, we know that ${}^\bL\cZ^\Kra(L^\flat)\in \rF^{2} K_0(\cN^\Kra)$. The desired decomposition then follows from Theorem \ref{thm:horizontalpart}.
	\end{proof}
	
	By Lemma \ref{lem:cZisinF^2} and \eqref{eq:K^Y and K'(Y)} we know that ${}^\bL \cZ^\Kra(L^\flat)_v\in K'_0(Y)$ where we can take $Y$ to be the reduced locus of $\cN^\Kra$. By the Bruhat-Tits stratification of $\cN^\Kra$ and the fact that $\Gr^1 K_0^{\Exc_{\Lambda_0}}(\cN^\Kra)\cong\mathrm{CH}^1(\Exc_{\Lambda_0})$ is generated by $H_{\Lambda_0}$, we have the following decomposition in $\Gr^2 K_0(\cN^\Kra)$:
	\begin{equation}\label{eq: bL ZKra Lflat v}
		{}^\bL \cZ^\Kra(L^\flat)_v=\sum_{\Lambda_2\in \cV^2(L^\flat)} m(\Lambda_2,L^\flat)[\Oo_{\tN_{\Lambda_2}}]+\sum_{\Lambda_0\in \cV^0(L^\flat)} m(\Lambda_0,L^\flat)H_{\Lambda_0}.
	\end{equation}
	We will determine the multiplicities $m(\Lambda_2,L^\flat)$ and $m(\Lambda_0,L^\flat)$ when $\val(L^\flat)=0$ in this section and deal with the general case in Section \ref{sec:proof when n=3}.
	
	Now assume  $L^\flat=\spa\{\bx_1,\bx_2\}$  with  Gram matrix  $ \diag(u_1,u_2 (-\pi_0)^n)$ with $u_1,u_2\in \Oo_{F_0}^\times$. Applying  Proposition \ref{prop:multiplicityofExc} to $\cZ^\Kra(\bx_1)$, we find
	\[{}^\bL\cZ^\Kra(L^\flat)=[\Oo_{\tZ(\bx_1)}\otimes^\bL\Oo_{\cZ^\Kra(\bx_2)}]+\sum_{\Lambda_0\in \cV^0(\bx_1)} [\Oo_{\Exc_{\Lambda_0}}\otimes^\bL\Oo_{\cZ^\Kra(\bx_2)}].\]
	By Proposition \ref{prop:cZxunimodular}, we know the intersection $\tZ(\bx_1)\cap\cZ^\Kra(\bx_2)$ is proper and is isomorphic to $\cZ^\Kra_{2,\chi(u_1)}(\bx_2)$. Combining this with Corollary \ref{cor:cZintersectExc} we obtain
	\begin{equation}\label{eq:cZLwhenx_1isunit}
		{}^\bL\cZ^\Kra(L^\flat)=i_*({}^\bL\cZ^\Kra_{2,\chi(u_1)}(\bx_2))-\sum_{\Lambda_0\in \cV^0(L^\flat)} H_{\Lambda_0}.
	\end{equation}
	where $i_*$ is the map $\Gr^1 K_0(\cN^\Kra_{2,\chi(u_1)})\rightarrow \Gr^2 K_0(\cN^\Kra_{3,1})$ induced by the closed immersion $i:\cN^\Kra_{2,\chi(u_1)}\rightarrow \cN^\Kra_{3,1}$.
	Equation \eqref{eq:cZLwhenx_1isunit} reduces the problem of decomposing $\LZ(L^\flat)$ in this case to \cite[Theorem 4.5]{Shi2} and \cite[Theorem 4.1]{HSY}. We do not make the effort to write the complete result down, but instead look at two basic examples.

	Let us begin by the case when $L^\flat$ is unimodular.
	By \eqref{eq:cZLwhenx_1isunit} and either \cite[Theorem 4.5]{Shi2} or \cite[Theorem 4.1]{HSY}, we have
	\begin{equation}
		{}^\bL\cZ^\Kra(L^\flat)=[\Oo_{\tZ(L^\flat)^\circ}]
	\end{equation}
	in the notation of Theorem \ref{thm:horizontalpart}.

	Next consider $L^\flat=\spa\{\bx_1,\bx_2\}$ with Gram matrix $
	\diag(1,-u\pi_0)$ where $u\in \Oo_{F_0}^\times$. Then  $\spa\{\bx_1\}^{\perp}$ is split and $\tZ(\bx_1)\cong \cN^\Kra_{2,1}$. By Proposition \ref{prop:n=3reducedlocusII} (3), $\cV^2(L^\flat)$ consists of two adjacent lattices $\Lambda$ and $\Lambda'$. Moreover by \cite[Theorem 4.1]{HSY} and \eqref{eq:cZLwhenx_1isunit}, we have
	\begin{equation}\label{eq:almostminusculedecomposition}
		{}^\bL\cZ^\Kra(L^\flat)_v=[\Oo_{\tN_{\Lambda}}]+[\Oo_{\tN_{\Lambda'}}]+H_{\Lambda\cap
			\Lambda'}.
	\end{equation}

	\subsection{The intersection number}
	Assume $\Lambda\in \cV^2$. For $\bx\in \bV\backslash\{0\}$, define
	\begin{equation}
		\Int_\Lambda(\bx)\coloneqq \chi(\cN^\Kra,\Oo_{\tN_\Lambda}\otimes^\bL \Oo_{\cZ^\Kra(\bx)}).
	\end{equation}
	In this subsection we prove the following theorem.
	\begin{theorem}\label{thm: Int Lambda} Let $\Lambda\in \cV^2$ and $\bx\in \bV\backslash\{0\}$. Then
		\[\Int_\Lambda(\bx)=1_{\Lambda}(\bx)\]
		where $1_{\Lambda}$ is the characteristic function of $\Lambda\subset \bV$.
	\end{theorem}
	\begin{corollary}\label{cor:tVintersecttZy_0}
		Assume that $\Lambda_0\in \cL_0$ and $\Lambda\in \cL_2$ such that $\Lambda_0\subset \Lambda$. Then for any $y_0\in \Lambda_0\setminus \pi \Lambda_0$ such that $y_0^\bot$ is non-split, we have
		\[\chi(\cN^\Kra,\Oo_{\tN_{\Lambda}}\otimes^\bL\Oo_{\tZ(y_0)})=0.\]
	\end{corollary}
	\begin{proof}
		By Proposition \ref{prop:multiplicityofExc}, we know
		\[\cZ^\Kra(y_0)=\tZ(y_0)+\Exc_{\Lambda_0}.\]
		Now the corollary follows immediately from Theorem \ref{thm: Int Lambda} and Lemma \ref{lem:tVintersectExc}.
	\end{proof}
	
	\noindent\textit{Proof of Theorem \ref{thm: Int Lambda}:}
	We consider three different cases.
	First if $x\notin\Lambda$ or $\mathrm{v}(\bx)<0$, then  by Lemma \ref{lem:cNandcV_Lambda}, $\cZ^\Kra(\bx)\cap \tN_\Lambda=\emptyset$ hence $\Int_\Lambda(\bx)=0$. From now on we assume $x\in\Lambda$ and $\mathrm{v}(\bx)\geq 0$.
	Write $\bx= \bx_0  \pi^{n}$ with $\bx_0 \in \Lambda\backslash  \pi \Lambda$ and $n  \ge 0$.

	{\bf Case 1}:   First we assume  $\bx_0\in \Lambda\backslash \Lambda^\sharp$.  	Choose a   basis $\{e'_1,e'_2,e'_3\}$ of $\Lambda$ with Gram matrix $\cH_{3,1}^1$ such that
	\[\bx_0=x e'_1+y e'_2+z e'_3.\]
	Then one of $x$ and $y$ is in $\Oo_{F}^\times$ as $\Lambda^\sharp=\spa\{\pi e'_1, \pi e'_2, e_3\}$.
	Apparently the equation
	\[2u-v\bar{v}=h(\bx_0,\bx_0)\]
	has a solution $(u,v)\in \Oo_{F_0}^2$ with $u\in \Oo_{F_0}^\times$. Now according to Lemma \ref{lem:primitive in H, vector}, we can find a matrix $g\in \rU(\cH_{3,1}^1)(\Oo_{F_0})$ such that
	\[g\left(\begin{array}{c}
		x  \\
		y  \\
		z
	\end{array}\right)=\left(\begin{array}{c}
		\pi u  \\
		1  \\
		v
	\end{array}\right).\]
	Now replace the basis $\{e'_1,e'_2,e'_3\}$ by $\{e_1,e_2,e_3\}=\{e'_1,e'_2,e'_3\}g^{-1}$, we have
	\[\bx_0=\pi u e_1+e_2+v e_3\]
	where $u\in \Oo_{F_0}^\times$, $v\in \Oo_F$.

	Define
	\begin{equation}
		f_1=\frac{1}{\pi} u^{-1} e_2,\ f_2=\pi u e_1,\ f_3=e_3.
	\end{equation}
	Then $\{f_1,f_2,f_3\}$  has also Gram matrix $\cH_{3,1}^1$ and $\Lambda'\coloneqq \spa\{f_1,f_2,f_3\}$ is a type $2$ lattice adjacent to $\Lambda$ with $\Lambda_c=\Lambda\cap\Lambda'=\spa\{\pi e_1,e_2,e_3\}$ is a type $0$ lattice. Now in terms of the basis $\{f_1,f_2,f_3\}$ we have
	\[\bx=\pi^n(\pi u f_1+f_2+v f_3).\]
	Define $\theta\in \rU(\bV)$ by taking the basis $\{e_1,e_2,e_3\}$ to $\{f_1,f_2,f_3\}$. Then
	\[\theta(\bx)=\bx, \quad  \theta(\Lambda)=\Lambda'.\]
	In particular $\theta(\cZ^\Kra(\bx))=\cZ^\Kra(\bx)$ and
	\begin{equation} \label{eq:Int Lambda=Int Lambda'}
		\Int_{\Lambda'}(\bx)=\Int_\Lambda(\bx).
	\end{equation}

	Now let
	\[\by_0=e_3,\quad \by_1=\pi(-\pi u e_1+e_2),\]
	$L^\flat=\spa\{\by_0,\by_1\}$, and  $L=\spa\{\by_0,\by_1,\bx\}$. Then by \eqref{eq:almostminusculedecomposition} and Theorem \ref{thm:horizontalpart}  we have
	\begin{equation}
		{}^\bL\cZ^\Kra(L^\flat)=[\Oo_{\tN_{\Lambda}}]+[\Oo_{\tN_{\Lambda'}}]+H_{\Lambda_c}+[\Oo_{\tZ(M^\flat)}].
	\end{equation}
	where $\tZ(M^\flat)$ is the quasi-canonical lifting cycle of the lattice
	\[M^\flat\coloneqq \spa\{e_3,-\pi  u e_1+e_2\}.\]
	Combining with \eqref{eq:Int Lambda=Int Lambda'}, we have
	\begin{equation}\label{eq:10.10}
		\Int(L) =2 \cdot \Int_\Lambda(\bx) + \chi(\cN^{\Kra},{}^\bL\cZ^\Kra(\bx) \cdot H_{\Lambda_c})+\chi(\cN^{\Kra},{}^\bL\cZ^\Kra(\bx) \cdot [\Oo_{\tZ(M^\flat)}]).
	\end{equation}
	Let $\bx'=\pi^n(\pi u e_1+e_2)= \bx - \pi^n v e_3$. Then we have
	\begin{align*}
		\Int(L)
		&=\chi(\cN^\Kra, {}^\bL\cZ^\Kra(\by_0)\cdot {}^\bL\cZ^\Kra(\bx)\cdot {}^\bL\cZ^\Kra(\by_1))\\
		&=\chi(\cN^\Kra, {}^\bL\cZ^\Kra(\by_0)\cdot{}^\bL\cZ^\Kra(\bx')\cdot {}^\bL\cZ^\Kra(\by_1))\\
		&=\chi(\cN^\Kra, \Oo_{\tZ(\by_0)}\otimes^\bL\Oo_{\cZ^\Kra(\bx')}\otimes^\bL\Oo_{\cZ^\Kra(\by_1)})
		\\
		&\quad +\sum_{\Lambda_0\in \cV^0(L)} \chi(\cN^\Kra, \Oo_{\Exc_{\Lambda_0}}\otimes^\bL\Oo_{\cZ^\Kra(\bx')}\otimes^\bL\Oo_{\cZ^\Kra(\by_1)})
	\end{align*}
	where we have used linear invariance (\cite[Corollary D]{Ho2}) and Proposition \ref{prop:multiplicityofExc}.
	Notice that the Gram matrix of $\{\bx',\by_1\}$ is $\diag(2u(-\pi_0)^n,-2u\pi_0)$. By Proposition \ref{prop:cZxunimodular} and \cite[Theorem 1.1]{HSY},
	\[\chi(\cN^\Kra, \Oo_{\tZ(\by_0)}\otimes^\bL\Oo_{\cZ^\Kra(\bx')}\otimes^\bL\Oo_{\cZ^\Kra(\by_1)})=\begin{cases}
		1 & \text{ if } n=0,\\
		1+n-2\q & \text{ if } n\geq 1.
	\end{cases}\]
	By Corollary \ref{cor:multiplecZintersectExc} and \cite[Lemma 3.15]{HSY}, we know that
	\begin{align*}
		&\sum_{\Lambda_0\in \cV^0(L)} \chi(\cN^\Kra, \Oo_{\Exc_{\Lambda_0}}\otimes^\bL\Oo_{\cZ^\Kra(\bx')}\otimes^\bL\Oo_{\cZ^\Kra(\by_1)})=|\cV^0(L)|
		=\begin{cases}
			1 & \text{ if } n=0,\\
			2\q+1 & \text{ if } n\geq 1.\end{cases}
	\end{align*}

	Combining the above two equations we know that
	\[\chi(\cN^\Kra, {}^\bL\cZ^\Kra(\bx)\cdot {}^\bL\cZ^\Kra(L^\flat))=n+2.\]
	On the other hand, by Corollary \ref{cor: P1 dot cZ},
	\[\chi(\cN^\Kra, H_{\Lambda_0}\cdot {}^\bL\cZ^\Kra(\bx))=-1.\]
	By \cite[Proposition 3.3]{G}
	\[\chi(\cN^\Kra, \Oo_{\tZ(M^\flat)}\otimes^\bL\Oo_{\cZ^\Kra(\bx)})=n+1.\]
	Hence we obtain  by (\ref{eq:10.10})
	\begin{equation}
		\Int_\Lambda(\bx)=1.
	\end{equation}
	
	{\bf Case 2:} Now we Assume  $\bx_0\in \Lambda^\sharp\backslash\pi\Lambda$. As in the proof of the previous case,  we can find a basis $\{e_1,e_2,e_3\}$ of $\Lambda$ with Gram matrix $\cH_{3,1}^1$ by Lemma \ref{lem:vector primitive in S} such that
	\[\bx=\pi^n(ue_3+\pi e_1).\]
	where $u\in \Oo_{F_0}^\times$. Define
	\[\Lambda'=\spa\{\pi e_1,\frac{1}{\pi}e_2,e_3\},\ \Lambda_c=\Lambda\cap\Lambda',\]
	then $\bx_0\in\Lambda'\backslash\Lambda'^\sharp$. Also define
	\[\by_0=e_3,\ \by_1=\pi(-\pi  e_1+e_2),\]
	and $L^\flat\coloneqq \spa\{\by_0,\by_1\}$. Then by Theorem \ref{thm:horizontalpart} and \eqref{eq:almostminusculedecomposition} we have
	\begin{equation} \label{eq:10.13}
		{}^\bL\cZ^\Kra(L^\flat)=[\Oo_{\tN_{\Lambda}}]+[\Oo_{\tN_{\Lambda'}}]+H_{\Lambda_c}+[\Oo_{\tZ(M^\flat)}],
	\end{equation}
	where $\tZ(M^\flat)$ is the quasi-canonical lifting cycle of the lattice
	\[M^\flat\coloneqq \spa\{e_3,-\pi  e_1+e_2\}.\]
	Let $\bx'\coloneqq \pi^{n+1}  e_1=\bx -\pi^n u e_3$, then we have
	\begin{align*}
		\chi(\cN^\Kra, {}^\bL\cZ^\Kra(\bx)\cdot{}^\bL\cZ^\Kra(L^\flat))
		=&\chi(\cN^\Kra, \Oo_{\tZ(\by_0)}\otimes^\bL\Oo_{\cZ^\Kra(\bx')}\otimes^\bL\Oo_{\cZ^\Kra(\by_1)})\\
	 &+\sum_{\Lambda_0\in \cV^0(L)} \chi(\cN^\Kra, \Oo_{\Exc_{\Lambda_0}}\otimes^\bL\Oo_{\cZ^\Kra(\bx')}\otimes^\bL\Oo_{\cZ^\Kra(\by_1)}).
	\end{align*}
	Notice that the Gram matrix of $\{\bx',\by_1\}$ is equivalent to $\cH_1$ when $n=0$, and to $\diag(u_1\pi_0^n,u_2\pi_0)$  for some $u_1,u_2\in \Oo_{F_0}^\times$ when $n\geq 1$. Hence by Proposition \ref{prop:cZxunimodular} and \cite[Theorem 1.1]{HSY}, we know that
	\[\chi(\cN^\Kra, \Oo_{\tZ(\by_0)}\otimes^\bL\Oo_{\cZ^\Kra(\bx')}\otimes^\bL\Oo_{\cZ^\Kra(\by_1)})=\begin{cases}
		-(\q-1) & \text{ if } n=0,\\
		1+n-2\q & \text{ if } n\geq 1.
	\end{cases}\]
	By Corollary \ref{cor:multiplecZintersectExc} and Lemmas 3.15 and 3.16 of \cite{HSY}, we know that
	\begin{align*}
		& \sum_{\Lambda_0\in \cV^0(L)} \chi(\cN^\Kra, \Oo_{\Exc_{\Lambda_0}}\otimes^\bL\Oo_{\cZ^\Kra(\bx')}\otimes^\bL\Oo_{\cZ^\Kra(\by_1)}) =|\cV^0(L)|=\begin{cases}
			\q+1 & \text{ if } n=0,\\
			2\q+1 & \text{ if } n\geq 1.
	\end{cases}  \end{align*}
	Hence we know that
	\[\chi(\cN^\Kra, {}^\bL\cZ^\Kra(\bx)\cdot{}^\bL\cZ^\Kra(L^\flat))=n+2.\]
	On the other hand, by Corollary \ref{cor: P1 dot cZ},
	\[\chi(\cN^\Kra,H_{\Lambda_0}\cdot {}^\bL\cZ^\Kra(\bx))=-1.\]
	By \cite[Proposition 3.3]{G}
	\[\chi(\cN^\Kra,\Oo_{\tZ(M^\flat)}\otimes^\bL \Oo_{\cZ^\Kra(\bx)})=n+1.\]
	Since $\bx\in \Lambda'\backslash \Lambda'^\sharp$, by the previous case we also  have
	\[\Int_{\Lambda'}(\bx)=1.\]
	Combining all above, we have by (\ref{eq:10.13})
	\begin{equation}
		\Int_\Lambda(\bx)=1.
	\end{equation}
	This finishes the proof of Theorem \ref{thm: Int Lambda}. \qedsymbol

	\section{Proof of the modified Kudla-Rapoport conjecture: three-dimension case}\label{sec:proof when n=3}
	In  this section, we will prove Theorem \ref{thm:mainthmintroduction}. We need some preparation.

	\begin{proposition}\label{prop: v(L)=0 case}
		Assume that  $L \subset \bV$ has a Gram matrix
				\noindent $T=\diag(u_1, u_2 (-\pi_0)^{b}, u_3 (-\pi_0)^c)$ with $u_i \in \Oo_{F_0}^\times$ and $0 \le b \le c$. Then
		\begin{align*}
			\Int(L)=\pden(L).
		\end{align*}
		Moreover, for every decomposition $L =L^\flat \oplus \spa\{\bx\}$, we have
		$$
		\Int(L)^{(2)}=\pden(L)^{(2)}.
		$$
	\end{proposition}
	\begin{proof}
		Fix a basis $\{\bx_1,\bx_2,\bx_3\}$ of $L$ such that the Gram matrix of $\{\bx_1,\bx_2,\bx_3\}$ is $T=\diag(u_1, T_2)$ where $u_1\in \Oo_{F_0}^\times$ and $T_2\in \Herm_2(\Oo_{F})$. Let $u_1^{-1}\cdot L$ be a lattice represented by $u_1^{-1}\cdot T$.
		Since  $\Int(u_1^{-1}\cdot L)=\Int(L)$ and $\pden(u_1^{-1}\cdot L)=\pden(L)$, we may assume $u_1=1$.
		Let $L^{\flat}=\spa\{\bx_2,\bx_3\}$. According to Propositions \ref{prop:cZxunimodular}, \ref{prop:multiplicityofExc}, and Corollary \ref{cor:multiplecZintersectExc}, we have
		\begin{align}\label{eq: int(T)-int(T_2)}
			\Int(L)-\Int(L^\flat) &
			=\chi(\cN^\Kra,{}^\bL\cZ^\Kra(\bx_1)\cdot \LZ(L^\flat))-\chi(\cN^\Kra,{}^\bL\tZ(\bx_1)\cdot \LZ(L^\flat))\\ \nonumber
			&=\sum_{\Lambda_0\in \cV^0(L)} \chi(\cN^\Kra,[\Oo_{\Exc_{\Lambda_0}}]\cdot \LZ(L^\flat))\\ \nonumber
			&=|\{\cV^0(L)\}|.
		\end{align}
		Now the result we want follows by comparing \eqref{eq: int(T)-int(T_2)} with \eqref{eq: pDen(T)-pDen(T_2)}, and the identity $\Int(L^\flat)=\pden(L^\flat)$ proved in \cite[Theorem 1.3]{Shi2} and \cite[Theorem 1.3]{HSY}.
		(2) follows from (1) and Theorem  \ref{thm: equivalent form of modified KR} (2).
	\end{proof}

	\noindent{\bf Proof of Theorem \ref{thm:decomposition of cD intro}:}
	Under the assumption $\mathrm{v}(L^\flat)>0$, we can decompose $\cD(L^\flat)$ in $\Gr^2 K_0(\cN^\Kra)$ as
	\begin{align}\label{eq: decom of D(L flat)}
		\cD(L^\flat)=\sum_{\Lambda_2\in \cV(L^\flat)}m(\cD(L^\flat),\Lambda_2)[\Oo_{\tN_{\Lambda_2}}]+\sum_{\Lambda_0\in \cV(L^\flat)}m(\cD(L^\flat),\Lambda_0)H_{\Lambda_0},
	\end{align}
	by \eqref{eq: bL ZKra Lflat v} and Proposition \ref{prop:horizontalpartofDL}.
	
	{\bf Claim 1}: $m(\cD(L^\flat),\Lambda_0)=0$ unless $L^\flat \subset \Lambda_0$. In such a case,
	\begin{align}
		m(\cD(L^\flat),\Lambda_0)=\begin{cases}
			\q+1 & \text{ if $\Lambda_0\in \cV(L^\flat)\setminus \mathcal{B}(L^\flat)$},\\
			1 & \text{ if $\Lambda_0\in \mathcal{B}(L^\flat)$}.
		\end{cases}
	\end{align}
	Indeed,  since  $\Lambda_0$  is of type $0$, we may choose a $y_0\in \bV\setminus L^\flat_F$ such that $\spa\{y_0\}^\bot$ is non-split and  $y_0 \in \Lambda_0\setminus \pi \Lambda_0$. In this case, Proposition \ref{prop:multiplicityofExc}, Corollaries \ref{cor: P1 dot Exc} and \ref{cor: P1 dot cZ} imply that
	\begin{equation*}
		\chi(\cN^\Kra,H_{\Lambda_0} \cdot [\Oo_{\tilde{\cZ}(y_0)}])=1.
	\end{equation*}
	Thus, by \eqref{eq: decom of D(L flat)},  Corollaries \ref{cor:tVintersecttZy_0} and \ref{cor:reducedimensionofcZbyone}, we have
	\begin{align*}
		m(\cD(L^\flat),\Lambda_0) =\chi(\cN^\Kra,\cD(L^\flat) \cdot [\Oo_{\tilde{\cZ}(y_0)}]).
	\end{align*}
	
	Let $(2a,2b)$ ($b>a$) be the fundamental invariant of the projection of $L^\flat$ onto $\spa\{y_0\}^\bot$.
	Let $\varphi$ be the natural quotient map $\Lambda_0\rightarrow \Lambda_0/\pi\Lambda_0$ and define
	\[m\coloneqq \mathrm{dim}_{\F_\q} \varphi(L^\flat) \le 2.\]
	Equation \eqref{eq:lambdasharpintersection} implies that $m=0$ if and only if $\Lambda_0\in \cL(L^\flat)\setminus \mathcal{B}(L^\flat)$.
	First assume $m=0$, in other words, $L^\flat\subset \pi \Lambda_0$ so $b\geq a \geq 1$. By the definition of $\cD(L^\flat)$ and \cite[Theorem 1.2]{Shi2}, we have
	\[m(\cD(L^\flat),\Lambda_0)=\mu(a,b)-\q\mu(a-1,b)-\mu(a,b-1)+\q \mu(a-1,b-1)=\q+1,\]
	as claimed where
	\begin{align}\label{eq:mu(a,b)}
		&\mu(a,b)=\chi(\cN^\Kra,{}^\bL\cZ^\Kra(L^\flat) \cdot [\Oo_{\tilde{\cZ}(y_0)}])=\begin{cases}
			2\sum_{s=0}^{a}\q^s(a+b+1-2s)-a-b-2 & \text{ if } a\geq 0\\
			0 & \text{ if } a<0.
		\end{cases}
	\end{align}

	Now assume $m=1$, then $\varphi(L^\flat)$ is a line $\ell$ and $b\geq 1$. By the assumption that $y_0\notin L^\flat_F$, we know $\ell$ is not in $ \mathrm{Span}\{\varphi(y_0)\}$, hence the projection of $\ell$ onto $\varphi(y_0)^\bot$ is non-zero. Since $\varphi(y_0)^\bot$ is non-split, we must have $a=0$.
	Hence by the definition of $\cD(L^\flat)$ and \eqref{eq:mu(a,b)}, we have
	\begin{align*}
		m(\cD(L^\flat),\Lambda_0)
		&=\mu(0,b)-\q\mu(-1,b)-\mu(0,b-1)+\q \mu(-1,b-1)=1
	\end{align*}
	as claimed.
	Finally,  $m=2$ is impossible since  $v(L^\flat) >0$. This finshes the proof of Claim 1.
	
	{\bf Claim 2:} $m(\cD(L^\flat),\Lambda_2)=2$ for any $\Lambda_2 \in \cV^2(L^\flat)$.
	
	Indeed, according to Lemma \ref{lem: Z(L hecke) dot Exc}, we have $	\chi(\cN^\Kra,\cD(L^\flat) \cdot [\Oo_{\Exc_{\Lambda_0}}])=0$. On the other hand, Corollary \ref{cor: P1 dot Exc} and Lemma \ref{lem:tVintersectExc} imply that
	\begin{align}
		\chi(\cN^\Kra,\cD(L^\flat) \cdot [\Oo_{\Exc_{\Lambda_0}}])=\sum_{\Lambda_0\subset \Lambda_2 }m(\cD(L^\flat),\Lambda_2)-2m(\cD(L^\flat),\Lambda_0).
	\end{align}
	Combining  the above with Claim 1, we have
	\begin{align}\label{eq: relation btw m(Lambda2) m(Lambda0)}
		0&=\chi(\cN^\Kra,\cD(L^\flat) \cdot [\Oo_{\Exc_{\Lambda_0}}])=\begin{cases}
			\sum_{\Lambda_0\subset \Lambda_2 }m(\cD(L^\flat),\Lambda_2)-2(\q+1) & \text{ if  $\Lambda_0 \in \cL(L^\flat)\setminus \mathcal{B}(L^\flat)$},\\
			\sum_{\Lambda_0\subset \Lambda_2 }m(\cD(L^\flat),\Lambda_2)-2 & \text{ if  $\Lambda_0 \in \mathcal{B}(L^\flat)$}.\end{cases}
	\end{align}
	Recall  $\mathcal{S}(L^\flat)$ in Definition \ref{def:skeleton}.
	First assume $\Lambda_2\in \cL(L^\flat)\setminus\mathcal{S}(L^\flat)$. If $d(\Lambda_2,\mathcal{B}(L^\flat))$ is equal to $\frac{1}{2}$, choose $\Lambda_0\in \mathcal{B}(L^\flat)$ such that $\Lambda_0\subset \Lambda_2$, then $\Lambda_2$ is the unique lattice in $\cV^2(L^\flat)$ that contains $\Lambda_0$, hence \eqref{eq: relation btw m(Lambda2) m(Lambda0)} implies that $m(\cD(L^\flat),\Lambda_2)=2$ in this case. Now Corollary \ref{cor:inductiononbruhattits} allows us to show $m(\cD(L^\flat),\Lambda_2)=2$ by induction on the distance $d(\Lambda_2,\mathcal{B}(L^\flat))$ for any $\Lambda_2\in \cL(L^\flat)\setminus\mathcal{S}(L^\flat)$.
	
	Similarly for $\Lambda_2\in\mathcal{S}(L^\flat)$, we can show $m(\cD(L^\flat),\Lambda_2)=2$ by induction on its distance to $\mathcal{S}(L^\flat)\cap \mathcal{B}(L^\flat)$. This finishes the proof of Claim 2.
	
	Notice that for $\Lambda_0 \in  \cV(L^\flat)$
	$$
	\sum_{\Lambda_2 \in \cV^2(L^\flat)} \sum_{\Lambda_0 \subset \Lambda_2} 1
	= \begin{cases}
		q+1 &\hbox{if  }\Lambda_0 \in \cV(L^\flat)\setminus \mathcal{B}(L^\flat),
		\\
		1 &\hbox{if  } \Lambda_0  \in \mathcal{B}(L^\flat).
	\end{cases}
	$$
	This finishes the proof of Theorem \ref{thm:decomposition of cD intro}. \qedsymbol
	
	In the following discussion we freely use Theorem \ref{thm: Int Lambda} and Corollary \ref{cor: P1 dot cZ} without explicitly referring to them.
	\begin{proposition} \label{prop11.3}
		Assume $L=L^\flat \obot \spa\{\bx\}$ with Gram matrix
		\[T= \diag(\cH_a, u_3(-\pi_0)^c)\]
		where $a$ is a positive odd integer, and $c\ge 0$. Then
		\begin{align}\label{eq: pden(T_hecke), H_i}
			\Int(L)^{(2)}=\pden (L)^{(2)}= \begin{cases}
				1-\q^{a} & \text{ if $a\le 2c$},\\
				1-\q^{2c+1} & \text{ if $a> 2c$}.
			\end{cases}
		\end{align}
	\end{proposition}
	\begin{proof}  By Proposition \ref{prop: pden T_a}, it suffices to prove the identity for $\Int(L)^{(2)}$.
		
		Now we compute $\Int(L)^{(2)}$. We may take $L^\flat=\spa\{\pi^\frac{a+1}2 e_1,\pi^\frac{a+1}2 e_2\}$, where the Gram matrix of $\{e_1,e_2\}$ is $\cH$. Let $e_3=\pi^{-c} \bx $. Then  $\cL(L^\flat)$ is centered at $\spa\{e_1,e_2,e_3\}$ of radius $\frac{a}{2}$ by Proposition \ref{prop:n=3reducedlocusII}.
		
		Assume $a\le 2c$ first. In this case, $\cL(L^\flat)\subset \cL(\bx)$. As a result, we have $\Int_{\Lambda_2}(\bx)=1$ and $\Int_{\Lambda_0}(\bx)=-1$ for any $\Lambda_2\in\cV^2(L^\flat)$ and $\Lambda_0\in\cV^0(L^\flat)$.   Hence
		by Theorem \ref{thm:decomposition of cD intro}, we have
		\begin{align}\label{eq: Int when cT(L flat) is ball, c large}
			\Int(L)^{(2)}
			&=\sum_{\Lambda_2\in \cL(L^{\flat})}\chi(\cN^\Kra, (2[\Oo_{\tN_{\Lambda_2}}]+\sum_{\Lambda_0\subset \Lambda_2}H_{\Lambda_0})\cdot \cZ^{\Kra}(\bx))\\ \notag
			&=(1-\q)|\{\Lambda_2\mid \Lambda_2\in \cL(L^\flat)\}| \\ \notag
			&=(1-\q)(1+(1+\q)\q+(1+\q)\q^3+\cdots+(1+\q)\q^{a-2})\\ \notag
			&=1-\q^a,
		\end{align}
		as claimed.

		Now we assume $a>2c$. We consider the case $c=0$ first. Recall that $\tilde{\cZ}(e_3)\approx \cN^{\Kra}_{2,1}$, hence $\cL(L^\flat)\cap \cL(e_3)$ is a ball of radius $\frac{a}{2}$ in the Bruhat-Tits tree $\cL_{2,1}$ of $\cN^{\Pap}_{2,1}$ centered at the vertex lattice corresponding to $\pi^{-\frac{a+1}{2}}\cdot L^\flat$, within which a vertex lattice $\Lambda_0$ of type $0$ is contained in two vertex lattices of type $2$, and a vertex lattice $\Lambda_2$ of type $2$ contains $\q+1$ vertex lattice of type $0$. Hence
		\begin{align*}
			|\{\Lambda_0\mid \Lambda_0\in (\cL(L^\flat)\setminus \mathcal{B}(L^\flat))\cap \cL(e_3)\}|=1+\q+(1+\q)\q+\cdots+(1+\q)\q^{\frac{a-3}{2}},
		\end{align*}
		and
		\begin{align*}
			|\{\Lambda_0\mid \Lambda_0\in \mathcal{B}(L^\flat)\cap \cL(e_3)\}|=(1+\q)\q^{\frac{a-1}{2}}.
		\end{align*}
		Moreover, notice that if $e_3\in \Lambda_0$, then $\Int_{\Lambda_2}(e_3)=1$ for any $\Lambda_2$ such that $\Lambda_0\subset \Lambda_2$.  As a result,
		\begin{align*}
			\chi(\cN^\Kra,\cD(L^\flat)\cdot \cZ^{\Kra}(e_3))
			&=2(1+ \q\cdot|\{\Lambda_0\mid \Lambda_0\in (\cL(L^\flat)\setminus \mathcal{B}(L^\flat))\cap \cL(e_3)\}|)\\
			&\quad -(\q+1) |\{\Lambda_0\mid \Lambda_0\in (\cL(L^\flat)\setminus \mathcal{B}(L^\flat))\cap \cL(e_3)\}|\\
			&\quad-|\{\Lambda_0\mid \Lambda_0\in \mathcal{B}(L^\flat)\cap \cL(e_3)\}|\\
			&=2+(\q-1)(1+\q+(1+\q)\q+\cdots+(1+\q)\q^{\frac{a-3}{2}})-(1+\q)\q^{\frac{a-1}{2}}\\
			&=1-\q,
		\end{align*}
		which is compatible with \eqref{eq: pden(T_hecke), H_i}.
		
		Next we show
		\[\chi(\cN^\Kra,\cD(L^\flat)\cdot (\cZ^{\Kra}(\pi e_3)-\cZ^{\Kra}(e_3)))=\q-\q^{3}.\]
		According to Proposition \ref{prop:n=3reducedlocusI}, $\cV(\pi e_3)=\{\Lambda \mid d(\Lambda,\cL(e_3))\le 1\}$. Hence, around each $\Lambda_0\in (\cL(L^\flat)\setminus \mathcal{B}(L^\flat))\cap\cL(e_3)$, there will be $\q(\q-1)$ many new vertex lattices of type $0$  in $\cL(L^\flat)\cap\cL(\pi e_3)\setminus \cL(L^\flat)\cap\cL(e_3)$. Hence,
		\begin{align*}
			&\chi(\cN^\Kra,\cD(L^\flat)\cdot (\cZ^{\Kra}(\pi e_3)-\cZ^{\Kra}(e_3)))\\
			&=2\q\cdot \q(\q-1)(1+\q+(1+\q)\q+(1+\q)\q^2+\cdots+(1+\q)\q^{\frac{a-1}{2}-2})\\
			&\quad -\q(\q-1)(\q+1)(1+\q+(1+\q)\q+(1+\q)\q^2+\cdots+(1+\q)\q^{\frac{a-1}{2}-2})\\
			&\quad -\q(\q-1) (1+\q)\q^{\frac{a-1}{2}-1}\\
			&=\q-\q^{3}.
		\end{align*}
		Continuing in this way, we can show
		\begin{align*}
			&\chi(\cN^\Kra,\cD(L^\flat)\cdot (\cZ^{\Kra}(\pi^i e_3)-\cZ^{\Kra}(\pi^{i-1}e_3)))
			=\q^{2i-1}-\q^{2i+1}
		\end{align*}
		for $2i<a$. Thus,
		$$
		\Int(L)^{(2)}=\cD(L^\flat)\cdot \cZ^{\Kra}( \pi^c e_3) =1- q^{2c+1} =\partial \hbox{Den}^{(2)}(L)
		$$
		as claimed.
	\end{proof}

	\begin{proposition}\label{prop: main thm, T_2 diag}
		Assume $L=L^\flat \obot \spa\{\bx\}$ with Gram matrix $$T=\diag(u_1(-\pi_0)^a,u_2(-\pi_0)^b, u_3(-\pi_0)^c)$$ where  $0<a\le b \le c$, then
		$$
		\Int(L)^{(2)}=\pden (L)^{(2)}=1+\chi(-u_2 u_3)\q^{a}(q^a-q^b) -\q^{a+b}.
		$$
	\end{proposition}
	\begin{proof}
		By Proposition \ref{prop: a le b<c}, it suffices to show
		\begin{align}\label{eq: pden(T_hecke), a le b<c}
			\Int(L)^{(2)}=1+\chi(-u_2u_3)\q^{a}(q^a-q^b) -\q^{a+b}.
		\end{align}
	Notice that since $a\le b\le c$, we have $\cL(L^\flat)\subset \cL(\bx)$ by Propositions \ref{prop:n=3reducedlocusI} and \ref{prop:n=3reducedlocusII}.

	First, we assume $\chi(-u_2u_3)=-1$, then  (\ref{eq: pden(T_hecke), a le b<c})  specializes to
		\begin{align*}
			\pden (T)^{(2)}= 1-\q^{2a}.
		\end{align*}
		On the other hand,   $\cL(L^\flat)$ is a ball centered at a vertex lattice of type $0$ with radius $a$ in this case. One can show $\Int(L)^{(2)}=1-\q^{2a}$ exactly as in \eqref{eq: Int when cT(L flat) is ball, c large}.
		
		Now we assume $\chi(-u_2u_3)=1$. In this case, \eqref{eq: pden(T_hecke), a le b<c} specializes to
		\begin{align*}
			\pden (L)^{(2)}= 1+\q^{2a}-2\q^{a+b}.
		\end{align*}
		Let $r=b-a$, and $L^\flat=\spa\{x_1,x_2\}$. Then $\cL(\pi^{-a}L^\flat)$ is a ball centered at a vertex lattice of type $0$ with radius $r$ in the Bruhat-Tits tree $\cL_{2,1}$. Hence,
		\[\cL(L^\flat)=\{\Lambda\mid \Lambda\in \cL_3,\ d(\Lambda,\cL(\pi^{-a}L^\flat))\le a\}.\]
		
		When $a=1$, $ \cV^0(\pi^{-1}L^\flat)=\cV^0(L^\flat)\setminus \mathcal{B}(L^\flat)$. Then combining with Theorem \ref{thm:decomposition of cD intro}, it is not hard to see
		\begin{align*}
			\Int(L)^{(2)}
			&=2(\q+1+\q\cdot 2(\q+\q^2+\cdots+\q^r))
			-(\q+1)|\cV^0(L^\flat)\setminus \mathcal{B}(L^\flat)|-|\mathcal{B}(L^\flat)|\\
			&=1+\q^2-2\q^{b+1},
		\end{align*}
		where we use the fact
		\[|\cV^0(\pi^{-1}L^\flat)|=1+2(\q+\q^2+\cdots+\q^r),\]
		and
		\[|\mathcal{B}(L^\flat)|=(\q-1)\q(1+2(\q+\q^2+\cdots+\q^{r-1}))+2\q^{r+2}.\]

		Now assume $a> 1$. Let $T$ be the Hermitian matrix associated with $L^\flat\obot \spa\{\bx\}$, then
		\begin{align*}
			&\pden (\pi L^\flat\obot\spa\{\bx\})^{(2)}- \pden (L^\flat\obot \spa\{\bx\})^{(2)}\\
			&= 1+\q^{2a+2}-2\q^{r+2a+2}-( 1+\q^{2a}-2\q^{r+2a})\\
			&=\q^{2a}(\q^2-1)(1-\q^{2r})\\
			&=\q^2\left(\pden (L^\flat\obot\spa\{\bx\})^{(2)}- \pden (\pi^{-1}L^\flat\obot\spa\{\bx\})^{(2)}\right),
		\end{align*}
		and
		\begin{align*}
			&\Int (\pi L^\flat\obot\spa\{\bx\})^{(2)}- \Int (L^\flat\obot\spa\{\bx\})^{(2)}\\
			&=2\q|\mathcal{B}(L^\flat)|-\q|\mathcal{B}(L^\flat)|-|\mathcal{B}(\pi L^\flat)|\\
			&=(2\q-\q-\q^2)|\mathcal{B}(L^\flat)|\\
			&=\q^2\left(\Int (L^\flat\obot\spa\{\bx\})^{(2)}- \Int (\pi^{-1} L^\flat\obot\spa\{\bx\})^{(2)}\right),
		\end{align*}
		where we use the fact $|\mathcal{B}(\pi L^\flat\obot\spa\{\bx\})|=\q^2|\mathcal{B}(L^\flat)|$. Since $r$ is arbitrary, an induction on $a$ gives the result we want.
	\end{proof}

	\noindent
	{\bf Proof of Theorem \ref{thm:mainthmintroduction}}:
	The case $\mathrm{v}(L)< 0$ follows from  Proposition \ref{prop: v(T)<0} and the  fact that $\Int(L)=0$ under this condition.  Assume $v(L) \ge 0$. There are three cases.
	
	{\bf Case 1}: When  $L$ has a Gram matrix $\diag(u_1, u_2(-\pi_0)^b, u_3 (-\pi_0)^c)$ as  in Proposition \ref{prop: v(L)=0 case},  it is proved by  Proposition \ref{prop: v(L)=0 case}.

	{\bf Case 2}: When $L$ has a basis $\{ \bx_1, \bx_2, \bx_3\}$ whose Gram matrix is $T=\diag( \cH_a, u_3 (-\pi_0)^c)$, take $L^\flat= \hbox{Span}(\bx_1, \bx_2)$ and $\bx=\bx_3$.  By Propositions \ref{prop: v(L)=0 case}, \ref{prop11.3}, and \ref{prop: main thm, T_2 diag}, we have
	$$
	\Int(L^{\flat, \prime} \oplus \spa\{\bx\})^{(2)} = \pden(L^{\flat, \prime} \oplus \spa\{\bx\})^{(2)}
	$$
	for any $L^\flat \subset L^{\flat, \prime} \subset L^\flat_F$
	(direct sums in the above identity are actually orthogonal direct sums). Thus by Theorem \ref{thm: equivalent form of modified KR} (1), we have
	$$
	\Int(L)=\pden(L).
	$$
	
	{\bf Case 3}: When $L$ has a Gram matrix $\diag(u_1 (-\pi_0)^a, u_2(-\pi_0)^b, u_3 (-\pi_0)^c)$ with $0\le a \le b \le c$, the same argument as Case 2  gives
	$\Int(L)=\pden(L)$.  This finishes the proof of the theorem. \qedsymbol

	Theorem \ref{thm:mainthmintroduction} and Theorem  \ref{thm: equivalent form of modified KR} imply the following corollary.

	\begin{corollary} \label{theo:Diff} For any lattice  $L=L^\flat  \oplus  \Oo_F\bx \subset \bV$ of rank $3$, we have
		$$
		\Int(L)^{(2)}=\pden (L)^{(2)}.
		$$
	\end{corollary}
	
\section{Global Applications}\label{sec:global application}
    In this section we assume that $F$ is an imaginary quadratic field with discriminant $d_F$. Denote by $a\mapsto \bar a$ the complex conjugation on $F$.
    The result in this section can be easily extended to CM number fields and more general level structures at split places, see \cite{LZ} and \cite{HLSY}. We restrict to the imaginary quadratic fields to make the exposition  as simple as possible.

     \subsection{Unitary Shimura varieties and special cycles}\label{subsec:unitary shimura}
  In this subsection, we briefly review the definition of an integral model of Shimura variety defined  in \cite{BHKRY1} over $\Spec \Oo_F$ . Let
\[\cM^{\Kra}_{(1,n-1)} \rightarrow \Spec \Oo_F\]
be the algebraic stack which assigns to each $\Oo_F$-scheme $S$ the groupoid of isomorphism classes of quadruples $(A,\iota,\lambda,\cF_A)$ where
\begin{enumerate}
    \item $A\rightarrow S$ is an abelian scheme of relative dimension $n$;
    \item $\iota:\Oo_F\rightarrow \End(A)$ is an action satisfying the following determinant condition (the Kottwitz condition of signature $(1,n-1)$)
    \[\mathrm{char}(T-\iota(\alpha)\mid \Lie A)=(T-s(\alpha)) (T-s(\bar{\alpha}))^{n-1} \in \Oo_S[T],\]
    for all $\alpha\in \Oo_F$ where $s:\Oo_F\rightarrow \Oo_S$ is the structure morphism;
    \item $\lambda:A\rightarrow A^\vee $ is a principal polarization  whose Rosati involution satisfies $\iota(\alpha)^*=\iota(\bar{\alpha})$ for all $\alpha \in \Oo_F$;
    \item $\cF_A \subset \Lie A$ is an $\Oo_F$-stable $\Oo_S$-module local direct summand of rank $n-1$ satisfying the Kr\"amer condition: $\Oo_F$ acts on $\Lie A/\cF_A$ by the structure map $s:\Oo_F \rightarrow \Oo_S$ and acts on $\cF_A$  by the complex conjugate of the structure map.
\end{enumerate}
Two objects $(A,\iota,\lambda,\cF_A)$ and $(A',\iota',\lambda',\cF_{A'})$ in $\cM^{\Kra}_{(1,n-1)}(S)$ are isomorphic if there is an $\Oo_F$-linear isomorphism $f:A\rightarrow A'$ of abelian schemes such that $f^*(\lambda')=\lambda$ and $f_*(\cF_A)=\cF_{A'}$. The stack
$\cM^{\Kra}_{(1,n-1)}$ is flat of dimension $n-1$ over $\Spec \Oo_F$. It is smooth over $\Spec \Oo_F[\frac{1}{{d_F}}]$ and has semi-stable reduction over primes of $F$ dividing ${d_F}$. This is indicated by the corresponding behaviour of its local model studied in \cite{Kr}. Analogously one can define $\cM_{(0,1)}\rightarrow \Oo_F$
be the algebraic stack which assigns to each $\Oo_F$-scheme $S$ the groupoid of isomorphism classes of triples $(E,\iota_0,\lambda_0)$ where
\begin{enumerate}
    \item $E\rightarrow S$ is an abelian scheme of relative dimension $1$;
    \item $\iota_0:\Oo_F\rightarrow \End(E)$ is an action such that its induced action on $\Lie E$ agrees with the complex conjugate of the structural map $s:\Oo_F\rightarrow \Oo_S$;
    \item $\lambda_0:E\rightarrow E^\vee $ is a principal polarization
    whose Rosati involution satisfies $\iota_0(\alpha)^*=\iota_0(\bar{\alpha})$ for all $\alpha \in \Oo_F$.
\end{enumerate}
The stack $\cM_{(0,1)}$ is smooth of relative dimension $0$ over $\Spec \Oo_F$, see for example \cite[Proposition 2.1.2]{HowardCMII}.

Assume that $\F$ is an algebraically closed field of characteristic $p$ over $\Oo_F$. Let
\[(E_0,\iota_0,\lambda_0,A,\iota,\lambda,\cF_A)\in (\cM_{(0,1)}\times \cM^{\Kra}_{(1,n-1)})(\Spec \F).\]
For any prime number $\ell\neq p$, we can define a Hermitian form $h(x,y)$ on the Tate module
\begin{equation}
    T_\ell(E_0,A):=\Hom_{\Oo_F}(T_\ell(E_0),T_\ell(A))
\end{equation}
as in \cite[Section 2.3]{KR2} using the polarizations $\lambda_0,\lambda$ and Weil pairings on $E_0,A$.

Fix a Hermitian space $W$ over $F$  of signature $(1,n-1)$ that contains a self-dual lattice $\mathfrak a$ and a Hermitian space $W_0$ over $F$  of signature $(0,1)$ that contains a self-dual lattice $\mathfrak{a}_0$. Define
\begin{equation}\label{eq:VandL}
    V:=\Hom_F (W_0,W), \quad L:=\Hom_{\Oo_F}(\mathfrak{a}_0,\mathfrak{a}).
\end{equation}
Here $V$ and $L$ are equipped with Hermitian forms coming from the those on $W_0$ and $W$. Define $G:=\rU(W)$. Also define the group scheme $\mathrm{GU}(W)$ over $\Q$ by
\[\mathrm{GU}(W)(R)=\{g\in \GL_R(W\otimes R)\mid (gv,gw)=c(g)(v,w),\forall v,w\in W\otimes R\}\]
where $R$ is any $\Q$-algebra. Also define $Z:=\mathrm{Res}_{F/\Q} \bG_m=\mathrm{GU}(W_0)$ and
\begin{equation}\label{eq:tildeG}
    \tilde{G}:=Z\times_{\bG_m} \mathrm{GU}(W)
\end{equation}
where the maps from the factors on the right hand side to $\bG_m$ are $\mathrm{Nm}_{F/\Q}$ and the similitude character $c(g)$ respectively. We have an isomorphism of group schemes
\begin{equation}\label{eq:tildeGisomorphism}
   \tilde{G}\rightarrow Z\times \rU(W), (z,g)\mapsto (z,z^{-1}g).
\end{equation}
Let $K_G$ be the compact subgroup of $G(\bA_f)$ that stabilizes the lattice $\mathfrak{a}\otimes \hat{\Z}$ and $K_Z=\hat{\Z}^\times\subset Z(\bA_f)$. Under the isomorphism \eqref{eq:tildeGisomorphism}, define
\begin{equation}
    K:=K_Z\times K_G.
\end{equation}

Now define
$\cM\subset \cM_{(0,1)}\times \cM^{\Kra}_{(1,n-1)}$
to be the open and closed substack such that
\[(E_0,\iota_0,\lambda_0,A,\iota,\lambda,\cF_A)\in \cM(S) \]
if and only if there is an isomorphism of Hermitian $\Oo_F\otimes \Z_\ell$-modules
\begin{equation}
    T_\ell(E_{0,s},A_s)\cong L\otimes \Z_\ell
\end{equation}
for any geometric point $s\in S$ and prime $\ell$ that is not the same as the characteristic of $s$. 
Then $\cM$ is an integral model of the Shimura variety associated to the group $\tilde{G}$ with level structure defined by $K$.

Now we review the definition of special cycles. For $(E,\iota_0,\lambda_0, A,\iota, \lambda,\cF_A)\in \cM(S)$ where $S$ is an $\Oo_F$-scheme, consider the projective $\Oo_F$-module of finite rank
\[V'(E,A)=\Hom_{\Oo_F}(E,A).\]
On this module there is a Hermitian form $h'(x,y)$ defined by
\begin{equation}
    h'(x,y)=\iota_0^{-1}(\lambda_0^{-1}\circ y^\vee \circ \lambda \circ x),
\end{equation}
where $y^\vee$ is the dual homomorphism of $y$. It is proved in \cite[Lemma 2.7]{KR2} that $h'(x,y)$ is positive semi-definite. The following is \cite[Definition 2.8]{KR2}.
\begin{definition}\label{def:globalspecialcycle}
For $T\in \Herm_m(\Z)_{>0}$, the special cycle $\cZ(T)$ is the stack of collections $(E,\iota_0,\lambda_0,A,\iota ,\lambda,\cF_A,\bx)$ where
\[(E,\iota_0,\lambda_0,A,\iota ,\lambda,\cF_A)\in \cM(S)\]
and $\bx=(x_1,\ldots,x_m)\in \Hom_{\Oo_F}(E,A)^m$ such that
\[h'(\bx,\bx)=(h'(x_i,x_j))=T.\]
\end{definition}
When $t\in \Z_{>0}$, each component of $\cZ(t)$ is a divisor by \cite[Proposition 3.2.3]{HowardCMII}. In general, $\cZ(T)$ does not necessarily have the expected codimension which is the rank of $T$.

Let $\mathcal C =\{ \mathcal C_p\} $ be a incoherent collection of local  Hermitian spaces of rank $n$ such that $\mathcal C_\ell\cong V_\ell$ for all finite $\ell$ and $\mathcal C_\infty$ is positive-definite. For a non-singular Hermitian matrix $T$ of rank $n$ with values in $\Oo_F$, Let $V_T$ be the Hermitian space with Gram matrix $T$. Define
	\begin{equation}
		\mathrm{Diff}(T,\mathcal C)\coloneqq \{p \text{ a place of } \Q  \mid \mathcal C_p \text{ is not isomorphic to } (V_T)_p\}.
	\end{equation}
    Then $\mathrm{Diff}(T,\mathcal C)$ is a finite set consisting of places of $\Q$ inert or ramified in $F$.
	By \cite[Proposition 2.22]{KR2}, $\cZ(T)$ is empty if $|\mathrm{Diff}(T,\mathcal C)|>1$. If $\mathrm{Diff}(T,\mathcal C)=\{p\}$ for a finite prime $p$ inert or ramified in $F$, it is proved in loc.cit. that the support of $\cZ(T)$ is on the supersingular locus of $\cM$ over $\Spec \bar\F_p$. Let $e$ be the ramification index of $F_p/\Q_p$.  Define the arithmetic degree
	\begin{equation}\label{eq:degcZ}
\widehat{\mathrm{deg}}_T=\chi(\cZ(T),\Oo_{\cZ(t_1)}\otimes^{\mathbb{L}} \Oo_{\cZ(t_2)} \otimes^{\mathbb{L}} \Oo_{\cZ(t_n)})\cdot \log p^{2/e},
	\end{equation}
	where $\otimes^\bL$ denotes the derived tensor product on the category of coherent sheaves on $\cM$, $\chi$ is Euler characteristic and $t_i$ ($1 \le i \le n$) are the diagonal entries of $T$. When $\mathrm{Diff}(T,\mathcal C)=\{\infty\}$,
 $\cZ(T)$ is empty (\cite[Lemma 2.7]{KR2}) and the arithmetic degree $\widehat{\mathrm{deg}}_T(v)$ is the integration of a green current $G(T, v)$ ($v >0$ is a positive-definite Hermitian matrix of order $n$) defined by Liu (\cite{LiuarithmeticI}) and Garcia-Sankaran (\cite{GS}), see for example \cite[Equation (15.3.0.2)]{LZ}.
	
\subsection{Eisenstein series}\label{subsec:Eisenstein series}
	On the analytic side, let $\chi:\bA/\Q^\times \rightarrow \C^\times$ be the quadratic character associated to the extension $F/\Q$.
 Fix a character $\eta:\bA_F^\times\rightarrow \C^\times$ such that $\eta|_{\bA^\times}=\chi^n$. We consider an incoherent Eisenstein series $E(z,s,\Phi)$ associated to a section $\Phi=\otimes \Phi_p$ in a degenerate principal series representation $I(s, \eta)$ of $\rU(n,n)(\bA)$ (see \cite[\S 12]{LZ} or \cite{KR2}), where $\tau\in \mathbb{H}_n$ (see \eqref{eq:Siegel upper half plane}), $s\in \C$ and $\Phi_p$ is given  as follows. The section $\Phi_\infty$ is the standard weight $n$ section. When $p < \infty$ is unramified in $F$, $\Phi_p$ is the standard section associated to the characteristic function  of $L_p^n$ via the map $\lambda: S(\mathcal{C}_p^n) \rightarrow I(0, \eta_p)$:
     \[\lambda(\varphi)(g)=\omega(g)\varphi(0), \]
     where $\omega$ is the Weil representation of $\mathrm{U}(n,n)$ associated to the character $\chi$. When  $p$ is ramified in $F$, define
	\begin{equation}\label{eq:nonstandard section}
		\Phi_{p}=\Phi_p^0+\sum_{i=1}^{\lfloor \frac{n}{2}\rfloor} A_{p,\epsilon}^i(s) \cdot \Phi_p^i.
	\end{equation}
	Here $\Phi_p^0$ is the standard section associated to the characteristic function $L_p^n$, $\Phi_p^i$ is the standard section associated to the characteristic function of $(\cH_{n,i}^{\epsilon})^n$ at $p$ with  $\epsilon=-\chi_p(L_p)$, and
	\begin{equation}
		A_{p, \epsilon}^{i}(0)=0,\quad  \frac{d}{ds}A_{p}^{i}|_{s=0}=\frac{(-1)^n}{p^{2i}}\cdot c_{n,i}^\epsilon\cdot \log p,
	\end{equation}
 where $c_{n,i}^\epsilon$ are as in \eqref{eq:coeff}. Let $\psi$ be the standard additive character of $\bA/\Q$, i.e.,
\begin{equation}
    \psi_\infty(x)=\exp(2\pi i x),
 \quad \psi_\ell(x)=\exp(-2 \pi i \lambda(x)),
\end{equation}
where $\lambda$ is the canonical map $\Q_\ell \rightarrow \Q_\ell/\Z_\ell \hookrightarrow \Q/\Z$.
Let $E_T(\tau,s,\Phi)$ (respectively, $E'_T(\tau,\Phi)$) be the $T$th Fourier coefficient of $E(\tau,s,\Phi)$ (respectively, $E'(\tau, 0, \Phi)$) with respect to $\psi$. Then for $s\gg 0$ we have the following product formula (see for example \cite[\S 12.4]{LZ})
\begin{equation}\label{eq:product formula for E T}
    E_T(\tau,s,\Phi)=c_\infty \cdot \prod_{p<\infty} W_{T,p}(1,s,\Phi_p) \cdot q^T,
\end{equation}
where $c_\infty$ is a constant independent of $T$ calculated in \cite[Proposition 4.5]{LiuarithmeticI}, and $W_{T,p}(1,s,\Phi_p)$ is the local Whittaker integral defined in \cite[Equation (10.2)]{KR2}.

\subsection{An equivalent form of Conjecture \ref{conj:main}}
In this subsection we assume $p$ is a prime of $\Q$ ramified in $F$. Let $|\cdot|_p$ be the non-archimedean valuation on $F_p$ normalized so that $|\sqrt{d_F}|_p=\frac{1}{p}$.
For $\epsilon=\pm 1$, let $V_{p}^{\epsilon}$ be the (unique up to isomorphism) $F_p/\Q_p$-Hermitian space of dimension $n$ and sign $\epsilon$.
For any lattice $M_p$ of rank $n$ in $V_p^\epsilon$, let $\Phi_{M_p}\in I_p(s,\eta_p)$ be the standard section associated to the characteristic function of $M_p^n$.
By \cite[Proposition 9.7]{Shi2}, we have
\begin{equation}\label{eq:relation between W and alpha}
    W_{T,p}(1,r,\Phi_{M_p}) =\gamma_p (V_{p}^{\epsilon})^n \cdot |\det (M_p)|_p^n \cdot |d_F|_p^f \cdot  \alpha_v(M_p,L_T,X)|_{X=p^{-2r}},
\end{equation}
where $f=\frac{1}{2}n^2+\frac{1}{4}n(n-1)$, $\alpha_p(\cdot,\cdot,X)$ is the local density polynomial defined in \eqref{eq:definition of local density polynomial} at the place $p$ and $\gamma_p(V_{p}^{\epsilon})$ is an $8$-th root of unity defined in \cite[Equation (10.3)]{KR2}. By loc. cit., we know
\begin{equation}\label{eq:relation between gamma}
		\gamma_p (V_{p}^{\epsilon})=-\gamma_p (V_{p}^{-\epsilon}).
\end{equation}
For $T\in \Herm_n(F)$, choose a lattice $L_T$ in the Hermitian space $V_{p}^{\chi_p(T)}$ with Gram matrix $T$. Then equations \eqref{eq:relation between W and alpha} and \eqref{eq:relation between gamma} imply that Conjecture \ref{conj:main} at the place $p$ is equivalent to the following conjecture.
\begin{conjecture}\label{conj:local conjecture in terms of W}
Let $\bV$ be the space of special quasi-homomorphisms as in \eqref{eq:bV introduction} such that $\chi_p(\bV) =\epsilon$. Let $T\in \Herm_n(F)$ such that $\chi_p(T)=\epsilon$  and $L_T$ be a lattice of rank $n$ in $\bV$ with Gram matrix $T$. Then
		\[\mathrm{Int}(L_T)\cdot  \log p= \frac{W'_{T,p}(1,0,\Phi_p)}{W_{I_n^{-\epsilon},p}(1,0,\Phi_p)},\]
  where $\Phi_p$ is defined in \eqref{eq:nonstandard section}.
\end{conjecture}

 \subsection{(Global) arithmetic Siegel-Weil formula}\label{subsec:main global}
 Similar to \cite[Theorem 1.3.1]{LZ}, we have the following theorem.
	\begin{theorem}(Arithmetic Siegel-Weil formula for non-singular coefficients) \label{thm:main global}
		Assume that the fundamental discriminant of $F$ is $d_F \equiv 1 \pmod 8$ and that Conjecture \ref{conj:main}  holds for every $F_p$ with  $p|d_F$.
		For any non-singular Hermitian matrix $T$ with values in $\Oo_F$ of size $n$, we have
		\[E'_T(\tau, 0, \Phi)=C \cdot \widehat{\mathrm{deg}}_{T}(v) \cdot q^T, \quad  q^T=\exp(2\pi i \tr(T\tau)),\]
		where $C$ is an explicit constant that only depends on $F$ and $L$, $\widehat{\mathrm{deg}}_{T}(v)=\widehat{\mathrm{deg}}_{T}$ for positive-definite $T$, and $\tau=u + i v$.  In particular, the arithmetic Siegel-Weil formula holds for $n=2, 3$ for non-singular $T$.
	\end{theorem}
 \begin{proof} We sketch the main idea of the proof. When  $|\hbox{Diff}(T, \mathcal C)| >1 $, both sides are zero.
When  $\hbox{Diff}(T, \mathcal C) =\{ p\}$ for a finite prime $p\ne 2$ (as $d_F \equiv 1 \pmod 8$), then $T$ is positive-definite and the support of $\cZ(T)$ is on the supersingular locus of $\cM$ over $\Spec \bar\F_p$ although it has higher than expected dimension and needs `derivation' to make it correct dimensional cycle (we skip it here and just define its degree in the following). When $p$ is inert in $F$, the theorem is proved in \cite[Theorem 1.3.1]{LZ}. When $p$ is ramified in $F$, the theorem can be proved in a similar fashion assuming Conjecture \ref{conj:main}.
The key is that by the $p$-adic uniformization theorem (\cite[Chapter 6]{RZ}), for each component $\cZ$ of $\mathcal \cZ(T) (\bar\F_p)$, the arithmetic degree of $\cZ(T)$ supported on $\cZ$
\begin{equation}
		\chi(\cZ,\Oo_{\cZ(t_1)}\otimes^{\mathbb{L}} \Oo_{\cZ(t_2)} \otimes^{\mathbb{L}} \Oo_{\cZ(t_n)})\cdot \log p,
\end{equation}
is the same as  $\Int(L) \log p$ ($L$ has Gram matrix $T$) in Conjecture \ref{conj:main}. In particular this number is independent of the choice of $\cZ$ and depends only on $T$.
Assuming that Conjecture \ref{conj:main} (or rather its equivalent form \ref{conj:local conjecture in terms of W}) holds, this is then equal to
 $ c_{p, 1} W_{T, p}'(1, 0, \Phi_p)$ for some constant $c_{p, 1}\ne 0$.   Thus, $\widehat{\mathrm{deg}}_T$ is this number times the number of components of $\mathcal \cZ(T)(\bar \F_p)$, which can be counted via the Siegel-Weil formula.  Combining these results together with \eqref{eq:product formula for E T}, we can prove
 $$
 C_p\cdot \widehat{\mathrm{deg}}_T\cdot q^T = E_T'(\tau, 0, \Phi)
 $$
for come explicit constant $C_p$ independent of $T$. Similar argument holds when  $\hbox{Diff}(T, \mathcal C) =\{ \infty \}$ in which case the theorem is proved in \cite{LiuarithmeticI} and \cite{GS}.
Finally, one checks that $C_p$ is independent of the choice of $p$.
 \end{proof}

	\appendix
	
	\section{Calculation of primitive local density}\label{sec: calc of primitive local density}
	In this appendix, we provide the proofs of Propositions \ref{prop: ind for t} and \ref{prop: ind rank 2}. Throughout this section, $M$  is unimodular  of rank $m\ge 2$ unless clearly stated otherwise.   Let $\{ v_1, \cdots,  v_{2k}, v_{2k+1}, \cdots, v_{2k+m}\}$ be a basis of  $M^{[k]}=\cH^k\obot M$ with Gram matrix $\cH^k\obot \diag(I_{m-1}, \nu)$. Let $L$ be a Hermitian lattice of rank $n$ with Gram  matrix $T$.
	An isometric embedding $\varphi:L\to M$ is called primitive if its image in $M/\pi M$ has dimension $\mathrm{rank}_{\Oo_F} (L)$.  We call a  vector $v$ primitive in $M$ if $\pi^{-1}v \not \in M$, or  equivalently the natural embedding $\varphi:\spaF\{v\}\hookrightarrow M$ is primitive.  For a $v\in M^{[k]}$, we let $\Pr_{\cH^k}(w_i)$ be the projection of $w_i$ to $\cH^{k}$.
	\subsection{Proof of Proposition \ref{prop: ind for t} }\label{subsec: 1-4 5.9}
	The main purpose of this subsection is to prove the first four parts of Proposition \ref{prop: ind for t}. Part (5) of this proposition follows from Proposition \ref{prop:ind formula reducing size} and   Corollaries \ref{cor: beta2} and  \ref{cor: beta N2}.
	\begin{proof} For (1), choose $M(1)=\frac{t\pi }{2}v_1+v_2 \in M^{[k]}$ with  $q(M(1))=t$.
		Then
		\begin{align}\label{eq: perp of M(1)}
			M(1)^{\perp}&=\spaF\{\frac{-t\pi
			}{2}v_1+v_2,v_3,\ldots,v_{2k}, v_{2k+1}, \cdots, v_{2k+m}\}
			\\ \notag
			&\cong  \langle -t \rangle  \obot \cH^{k-1} \obot M ,\end{align}
		which is represented by $\diag(-t, \cH^{k-1}, S)$.
		It is easy to check
		$$
		|M^{[k]}:M(1)\obot M(1)^{\perp}|^{-1}|M(1)^{\vee}:M(1)|=|t \pi|_F |t \pi|_F^{-1}=1.
		$$
		
		For (2) and (3), assume first that  $M$ is isotropic (and unimodular).
		In this case, we may choose a basis $\{v_{2k+1}',\dots,v_{2k+m}'\}$ of $M$ with  Gram matrix  $\mathrm{Diag}(\cH_0,1,\dots,1,-\nu)$. Choose $M(0)= \frac{t}2 v_{2k+1}'+ v_{2k+2}'$ with $q(M(0))=t$.
		Then
		\begin{align*}
			M(0)^\perp&=\hbox{Span}\{v_1, \cdots, v_{2k}, -\frac{t}2 v_{2k+1}'+ v_{2k+2}', v_{2k+3}', \cdots, v_{2k+m}'\}
			\\
			&   \cong \cH^k \obot \hbox{Span}\{v_{2k+3}', \cdots, v_{2k+m}'\} \obot \langle -t \rangle.
		\end{align*}
		as claimed.  Moreover
		$$
		|M^{[k]}:M(0)\obot M(0)^{\perp}|^{-1}|M(0)^{\vee}:M(0)|=|t|_F  |t \pi|_F^{-1}=q.
		$$
		
		Next, assume that $M$ is anisotropic. In this case, $M$ has rank $2$ and has Gram matrix $\diag(1, \nu)$ with $\chi(M)= \chi(-\nu) =-1$.   In this case, $E=F_0(\sqrt{-\nu})$ is a unramified quadratic field extension of $F_0$, and $N_{E/F_0}\Oo_E^\times = \Oo_{F_0}^\times$. When $\mathrm{v}(t)=0$,  $t \in N_{E/F_0}\Oo_E^\times$, i.e.,  $t =a \bar a + b \bar b \nu$. Take $M(0) = a v_{2k+1} + b v_{2k+2}$. Then  $q(M(0))=t$, and
		$$
		M(0)^\perp= \hbox{Span}\{v_1, \cdots, v_{2k}, -\nu \bar b v_{2k+1} +\bar a v_{2k+2}\}= \cH^k\obot \langle t\nu \rangle,
		$$
		and
		$$
		|M^{[k]}:M(0)\obot M(0)^{\perp}|^{-1}|M(0)^{\vee}:M(0)|=|\pi|_F^{-1} =q.
		$$
		When $\mathrm{v}(t) >0$, $t \not \in N_{E/F_0}\Oo_E^\times$. So there is no primitive $M(0) \in M$ with $q(M(0))=t$.
		This proves (1)---(3) of Proposition \ref{prop: ind for t}.
		
		The proof of (4) follows from the following four lemmas.

		\begin{lemma}\label{lem:H}
			For  primitive vectors $w_1,\ w_2\in \cH_i$ with $q(w_1)=q(w_2)$, we can find an element $g\in \mathrm{U}(\cH_i)$ such that $g(w_1)=w_2.$
		\end{lemma}
		\begin{proof}
			We treat the case $i$ is odd first. Assume $v=a_1v_1+a_2v_2$. Then $v$ is primitive implies that $a_1$ or $a_2$ is a unit. Without loss of generality, we assume $a_2$ is a unit and we can further assume $a_2=1$ by the action of $\begin{pmatrix}\bar{a}_2&0\\0&a_2^{-1}\end{pmatrix}$. Now notice that $q(v)=(v,v)=(a_1-\bar{a}_1)\pi^{i}$. Hence we can write $a_1=\alpha+\frac{q(v)\pi^{-i}}{2}$, where $\alpha \in \Oo_{F_0}$. Now let $g=\begin{pmatrix}1&-\alpha\\0&1\end{pmatrix}$, and it is straightforward to check that $g\in \mathrm{U}(\cH_i)$ and $g(v)=\frac{q(v)\pi^{-i}}{2}v_1 + v_2 $.
			
			Now we deal with the case $i$ is even. Again, we can assume $v=a_1v_1+v_2$. Then $q(v)=(a_1+\bar{a}_1)\pi^i$. Hence we can write $a_1=\frac{q(v)}{2}\pi^{-i}+\beta \pi$, where $\beta\in \Oo_{F_0}$. Now let $g=\begin{pmatrix}1&-\beta \pi \\0&1\end{pmatrix}$, and it is straightforward to check that $g\in \mathrm{U}(\cH_i)$ and $g(v)=\frac{q(v)\pi^{-i}}{2}v_1+v_2$.
		\end{proof}

		\begin{lemma}\label{lem:primitive in H, vector}
			Assume $M$ is any lattice such that $\mathrm{v}(M)\geq i$. For $w_1,\ w_2 \in \cH_i^k\obot M$, if $\Pr_{\cH_{i}^k}(w_1)$  and $\Pr_{\cH_{i}^k}(w_2)$ are primitive and $q(w_1)=q(w_2)$, then there exists $g\in \mathrm{U}(\cH_i^k\obot M)$ with $g(w_1)=w_2$.
		\end{lemma}
		\begin{proof}
			Choose a basis $\{v_1,\ldots,v_{2k}\}$ of $\cH_i^{k}$ such that the associated  Gram matrix is $\cH_i^k$. We also choose a basis $\{v_{2k+1},\ldots,v_{2k+m}\}$ of $M$. Write $w_1=\sum_{i=1}^{2k+m}a_iv_i$. Since $\mathrm{Pr}_{\cH_i^k}(w_1)$ is primitive, $a_i$ is a unit for some $i\in\{1,\ldots,2k\}.$ Without loss of generality, we may assume $a_1=1$. Let $w'=w_1+\frac{(-1)^{i+1}q(w_1)\pi^{-i}}{2} v_2$, then
			$$
			q(w')=q(w_1)+(w_1,\frac{(-1)^{i+1}q(w_1)\pi^{-i}}{2} v_2)
			+(\frac{(-1)^{i+1}q(w_1)\pi^{-i}}{2} v_2,w_1)\\
			=0,
			$$
			and $(w', v_2) =(v_1, v_2)$. As a result, $M_1=\mathrm{Span}_{\Oo_F}\{w_1,v_2\}=\mathrm{Span}_{\Oo_F}\{w',v_2\} $ is isometric to $\cH_i$. Notice that $\mathrm{val}_{\pi}(q(w_1))\ge i$ is guaranteed by the assumption $\mathrm{v}(M)\ge i$.
			
			Similarly, we can show $w_2\in M_2$ for some $M_2$ that is isometric to $\cH_i$. However, the assumption $\mathrm{v}(M)\ge i$ and  \cite[Proposition $4.2$]{J} imply that there exist $g\in \mathrm{U}(\cH_i^k\obot M)$ such that $g(M_1)=M_2$. In particular, $g(w_1)\in M_2$. Since both $g(w_1)$ and $w_2$ are in $M_2$, the problem is reduced to Lemma \ref{lem:H}.
		\end{proof}

		\begin{lemma}\label{lem:unimodularL}
			For  primitive vectors $w_1,\ w_2\in M$ with $q(w_1)=q(w_2)$, we can find an element $g\in \mathrm{U}(M)$ such that $g(w_1)=w_2.$
		\end{lemma}
		\begin{proof}  Since $M$ is unimodular, we can decompose
			$$
			M=\cH_0^k \obot M',
			$$
			where $M'=0$ or an anisotropic unimodular Hermiatian lattice of rank $1$ or $2$. If  $\Pr_{\cH_{0}^k}(w_1)$ and $\Pr_{\cH_{0}^k}(w_2)$ are primitive, this is Lemma \ref{lem:primitive in H, vector}. If $\Pr_{\cH_{0}^k}(w_1)$ is not primitive, then  $\Pr_{M'}(w_1)$ is primitive and thus $q(\Pr_{M'}(w_1)) \in \OO_F^\times$. This implies that $q(w_2)=q(w_1)$ is a unit, and $M= \OO_F w_i \obot  (\OO_F w_i)^\perp$.  Therefore there is some $g \in \mathrm{U}(M)$ with $g(w_1)=w_2$.
		\end{proof}

		\begin{lemma} \label{lem:vector primitive in S}
			Assume that  $w_1,\ w_2\in M^{[k]}$ are primitive and that  $\Pr_{\cH^k}(w_1)$ and $\Pr_{\cH^k}(w_2)$ are not primitive. Then we can find $g\in \mathrm{U}(M^{[k]})$ such that $g(w_1)=w_2$.
		\end{lemma}
		\begin{proof}
			Let $\{v_1,\ldots,v_{2k+m}\}$ be a basis of $\cH^k\obot M$, whose Gram matrix is $\cH^k\obot \mathrm{Diag}(1,\ldots,\nu)$ where $\nu$ is a unit. Assume $v\in M^{[k]}$ is primitive and $\Pr_{\cH^k}(v)$ is not primitive, then we can write $v=\sum_{i=1}^{2k}\pi a_i v_i + \sum_{j=2k+1}^{2k+m}a_jv_j$, where some $a_j$ is a unit for $2k+1\leq j \leq 2k+m.$ Again, without loss of generality, we may assume $a_{2k+m}=1.$ For $i\leq k$, we set
			\begin{align*}
				v_{2i-1}'=v_{2i-1}+\frac{\bar{a}_{2i}}{\nu} v_{2k+m},\quad v_{2i}'=v_{2i}+\frac{-\bar{a}_{2i-1}}{\nu} v_{2k+m}.
			\end{align*}
			Let $M_v=\spa_{\Oo_F}\{v_1',\ldots,v_{2k}'\}.$ Then it is easy to check that $M_v$ is perpendicular to $v$. Moreover, $M_v$ is isometric to $\cH^{k}$ since $\mathrm{val}_{\pi}((v'_{2i-1},v'_{2i}))=-1$ and $0\leq \mathrm{val}_{\pi}((v'_i,v'_j))$ for other $1\leq i,j\leq 2k$. Hence we can find $g_v\in \mathrm{U}(M^{[k]})$ such that $g_v(M_v)=\spaF\{v_1,\ldots,v_{2k}\},$ and $g_v(v)\in \spaF\{v_{2k+1},\ldots,v_{2k+m}\}=M$.
			
			Applying the above to $w_1$ and $w_2$, we can find $g_{w_1},g_{w_2}\in \mathrm{U}(M^{[k]})$ such that $g_{w_1}(w_1),\ g_{w_2}(w_2)\in M$. Now the problem is reduced to Lemma \ref{lem:unimodularL}, and the lemma is proved.
		\end{proof}
		
		According to Lemma  \ref{lem:primitive in H, vector} and Lemma \ref{lem:vector primitive in S},  a  primitive vector $v\in M^{[k]}$ is either in the same orbit of a  vector $M(1)\in \cH^k$ or a vector $M(0)\in M$. Lemma \ref{lem:H} implies that primitive vectors $M(1),M'(1)\in \cH^k$ with $q(M(1))=q(M'(1))$ lie in the same orbit.  Lemma \ref{lem:unimodularL} implies the similar result for primitive $M(0),M'(0)\in M$ with $q(M(0))=q(M'(0))$. A combination of the above proves Part $(4)$ of Proposition \ref{prop: ind for t}.
	\end{proof}

	\subsection{Proof of Proposition \ref{prop: ind rank 2} }\label{subsec: first part 5.10}
	In this subsection, we prove the first part of  Proposition \ref{prop: ind rank 2}, which we restate as follows for the convenience of the reader.

	\begin{proposition}\label{prop:A perbofHi} Let $L$ be a Hermitian $\Oo_F$-lattice of rank $2$ and $\mathrm{v}(L) >0$.
		Let $\varphi:L\rightarrow M^{[k]}$ be a primitive isometric embedding.
		Let $d(\varphi)$ be the dimension of the image of the map
		\begin{align*}
			\mathrm{Pr}_{\cH^k}\circ \varphi:L\rightarrow \cH^k
		\end{align*}
		in $\cH^k/\pi \cH^k$.
		Then
		$$
		\varphi(L)^\bot\cong (-L)\obot \cH^{k-d(\varphi)}\obot M^{(d(\varphi))}
		$$
		where $M^{(d(\varphi))}$ is unimodular of rank  equal to $(\mathrm{rank}(M)-2(2-d(\varphi)))$ and $\mathrm{det} M^{(d(\varphi))}=(-1)^{d(\varphi)}\mathrm{det} M$.
		In particular, if $d(\varphi)=1$ then $\mathrm{rank}(M)\geq 2$, and if $d(\varphi)=0$ then $\mathrm{rank}(M)\geq 4$.
	\end{proposition}
	\begin{proof}
		This proposition follows from Lemmas \ref{lem: perp of prim lattice} and \ref{lem:cancel primitive} below.
	\end{proof}
	
	\begin{lemma} \label{lem: perp of prim lattice} Let the notation be as in Proposition \ref{prop:A perbofHi}.
		If   $\mathrm{rank}(M^{[k]})\le 4$, then
		\[\varphi(L)^\bot\approx -L.\]
		In  particular,  such an $\varphi$ does not exist if $\chi(M^{[k]})=-1$ or $\mathrm{rank}(M^{[k]}) <4$.	
	\end{lemma}
	\begin{proof}
		First, assume $M^{[k]}=\cH^2$ and $L\approx \cH_i$ where $i>0$. Let $\varphi(L)=\spaF\{w_1,w_2\}$ such that the Gram matrix of $\{w_1,w_2\}$ is $\cH_i$. By Lemma \ref{lem:H}, we may assume $w_1=v_1$. Then we may write $w_2=a_1v_1+\pi^{i+1}v_2+a_3v_3+a_4v_4,$ and
		$\mathrm{min}\{\mathrm{v}_{\pi}(a_3),\mathrm{v}_{\pi}(a_4)\}=0$ by assumption. Without loss of generality, we may assume $a_3=1$. Now a direct calculation shows that
		\begin{align*}
			\varphi(L)^{\perp}=\spaF\{v_1+(-\pi)^{i+1}v_4,v_3+\bar{a}_4v_4\}.
		\end{align*}
		Its Gram matrix is
		\[\left(\begin{array}{cc}
			0 & (-\pi)^i \\
			\pi^i & (a_4-\bar{a}_4)\pi^{-1}
		\end{array}\right)=\left(\begin{array}{cc}
			0 & (-\pi)^i \\
			\pi^i & -a_1 (-\pi)^i-\bar{a}_1 \pi^i
		\end{array}\right)\approx \left(\begin{array}{cc}
			0 & -(-\pi)^i \\
			-\pi^i & 0
		\end{array}\right),\]
		hence
		\begin{align*}
			\varphi(L)^{\perp}\approx -L.
		\end{align*}
		
		Now we treat the case $M^{[k]}=\cH^2$ and $L\approx \diag(u_1(-\pi_0^a),u_2(-\pi_0)^b)$ where $0<a\le b$. Again,  let $\varphi(L)=\spaF\{w_1,w_2\}$ such that the Gram matrix of $\{w_1,w_2\}$ is $\diag(u_1(-\pi_0^a),u_2(-\pi_0)^b)$, and we can assume $w_1=v_1-\frac{q(w_1)\pi }{2}v_2$ without loss of generality. Then we may write $w_2=a_1(v_1+\frac{q(w_1)\pi }{2}v_2)+a_3v_3+a_4v_4$, hence $\mathrm{min}\{v_{\pi}(a_3),v_{\pi}(a_4)\}=0$ by assumption again. We may assume $a_3=1$ and a direct calculation shows that
		\begin{align*}
			\varphi(L)^{\perp}=\spaF\{v_1+\frac{q(w_1)\pi}{2}v_2-\bar{a}_1q(w_1)\pi v_4, v_3+\bar{a}_4v_4\}.
		\end{align*}
		Set $v_3'=v_1+\frac{q(w_1)\pi}{2}v_2-\bar{a}_1q(w_1)\pi v_4$ and $v_4'=a_1v_3'+v_3+\bar{a}_4v_4$. Then $\varphi(L)^\perp=\spaF\{v_3',v_4'\}$ and the Gram matrix of $\{v_3',v_4'\}$ is
		\begin{align*}
			\begin{pmatrix}
				-q(w_1)&0\\0&a_1\bar{a}_1q(w_1)-(a_3\bar{a}_4-\bar{a}_3a_4)\pi^{-1}
			\end{pmatrix}=
			\begin{pmatrix}
				-q(w_1)&0\\0&-q(w_2)
			\end{pmatrix}.
		\end{align*}

		Now assume $M^{[k]}=\cH\obot M$, where $M$ is unimodular of rank $2$. We only treat the case $L\approx \cH_i$ in detail, and the argument for $L$ represented by a diagonal matrix is similar. We assume that $M^{[k]}$ has a basis $\{v_1,\ldots,v_4\}$ with Gram matrix $\cH\obot \mathrm{Diag}(1,\nu)$ where $\nu$ is a unit.
		Let $\varphi(L)=\spaF\{w_1,w_2\}$ where the Gram matrix of $\{w_1,w_2\}$ is $\cH_i$. Then one can check that at least one of $w_1$ and $w_2$ is primitive in $\cH$.
		By Lemma \ref{lem:H}, we can assume that
		\[w_1\coloneqq \varphi(m_1)=v_1, w_2\coloneqq \varphi(m_2)=a_1 v_1 +\pi^{i+1}v_2 +a_3v_3 +a_4 v_4\]
		and
		\begin{equation}\label{eq:w_2norm=0again}
			(w_2,w_2)=a_1\pi^i-\bar{a}_1\pi^i+a_3 \bar{a}_3+a_4 \bar{a}_4 \nu=0.
		\end{equation}
		By our assumption we know that $\mathrm{min}\{\mathrm{v}_\pi(a_3),\mathrm{v}_\pi(a_4)\}=0$. Since we assume $i\geq 1$, \eqref{eq:w_2norm=0again} implies that both $a_3$ and $a_4$ are in $\Oo_{F_0}^\times$. This, in turn, implies that $-\nu\in \mathrm{Nm}_{F/F_0^\times}(\Oo_F^\times)=\Oo_{F_0}^2$. Hence $M^{[k]}\approx\cH\obot\cH_0$ and we can instead assume that $\{v_1,v_2,v_3,v_4\}$ has Gram matrix $ \cH\obot \cH_0$. We can further assume that
		\[w_1=v_1,w_2=a_1 v_1+ \pi^{i+1} v_2+v_3 +a_4 v_4\]
		with
		\[(w_2,w_2)=a_1\pi^i-\bar{a}_1\pi^i+a_4+\bar{a}_4=0.\]
		By direct calculation, it is easy to see that
		\[\varphi(L)^\bot=\mathrm{Span}_{\Oo_F}\{v_1-(-\pi)^i v_4,v_3-\bar{a}_4 v_4\}.\]
		Its Gram matrix is
		\[\left(\begin{array}{cc}
			0 & -(-\pi)^i \\
			-\pi^i & -a_4-\bar{a}_4
		\end{array}\right)=\left(\begin{array}{cc}
			0 & -(-\pi)^i \\
			-\pi^i & a_1\pi^i-\bar{a}_1\pi^i
		\end{array}\right)\approx \left(\begin{array}{cc}
			0 & -(-\pi)^i \\
			-\pi^i & 0
		\end{array}\right).\]

		Finally, assume $M^{[k]}$ is unimodular of rank $4$. We treat the case $L\approx \cH_i$ in detail, and the other cases follow from a similar argument. Let $\varphi(L)=\spaF\{w_1,w_2\}$ such that the Gram matrix of $\{w_1,w_2\}$ is $\cH_i$. Apparently $M^{[k]}$ contains a $\cH_0$. We can assume that $M^{[k]}$ has a basis $\{v_1,v_2,v_3,v_4\}$ with Gram matrix $ \cH_0\obot \mathrm{diag}\{1,\epsilon\}$ where $\epsilon\in \Oo_{F_0}^\times$. By Lemma \ref{lem:unimodularL} we can assume that $w_1=v_1$. Then we have
		\[w_2=a_1 v_1+\pi^{i} v_2+\sum_{j=3}^4 a_j v_j,\]
		and
		\begin{equation}\label{eq:w_2norm=0}
			(w_2,w_2)=a_1 (-\pi)^i+\bar{a}_1 \pi^i+a_3\bar{a}_3+a_4\bar{a}_4\epsilon=0.
		\end{equation}
		By our assumption we know that $\mathrm{min}\{\mathrm{v}_\pi(a_3),\mathrm{v}_\pi(a_4)\}=0$. Since we assume $i\geq 1$, \eqref{eq:w_2norm=0} implies that both $a_3$ and $a_4$ are in $\Oo_{F_0}^\times$. This in turn implies that $-\epsilon\in \mathrm{Nm}_{F/F_0^\times}(\Oo_F^\times)=\Oo_{F_0}^2$. Hence $M^{[k]}=\cH_0^2$ and we can instead assume that $\{v_1,v_2,v_3,v_4\}$ has Gram matrix $ \cH_0\obot \cH_0$. We can further assume that
		\[w_1=v_1,w_2=a_1 v_1+ \pi^i v_2+v_3 +a_4 v_4\]
		with
		\[(w_2,w_2)=a_1 (-\pi)^i+\bar{a}_1 \pi^i+a_4+\bar{a}_4=0.\]
		By a direct calculation, it is easy to see that
		\[\varphi(L)^\bot=\mathrm{Span}_{\Oo_F}\{v_1-(-\pi)^i v_4,v_3-\bar{a}_4 v_4\}.\]
		Its Gram matrix is
		\[\left(\begin{array}{cc}
			0 & -(-\pi)^i \\
			-\pi^i & -a_4-\bar{a}_4
		\end{array}\right)=\left(\begin{array}{cc}
			0 & -(-\pi)^i \\
			-\pi^i & a_1 (-\pi)^i+\bar{a}_1 \pi^i
		\end{array}\right)\approx \left(\begin{array}{cc}
			0 & -(-\pi)^i \\
			-\pi^i & 0
		\end{array}\right).\]
		
		Notice that, as a byproduct of the above argument, we actually also proved that if $\mathrm{rank}(M^{[k]})<4$ or $M$ is not split, then no such $\varphi$ exists.
		The lemma is proved.
		
	\end{proof}

	\begin{lemma}\label{lem:cancel primitive}
		Assume $\mathrm{v}(L)\ge 0$. Let $\varphi:L\rightarrow M^{[k]}$ be a primitive isometric embedding.
		Let $d(\varphi)$ be the dimension of $\mathrm{Pr}_{\cH^k}(\varphi(L))\otimes_{\Oo_F}\mathbb{F}_{\q}$
		in $\cH^k/\pi \cH^k$.
		Then there exist a $g\in \mathrm{U}(M^{[k]})$ such that
		\begin{align*}
			g(\varphi(L))\subset \cH^{d(\varphi)}\obot I_{4-2d(\varphi)}\subset M^{[k]},
		\end{align*}
		where $I_{4-2d(\varphi)}$ is a unimodular sublattice of $M^{[k]}$ with rank $4-2d(\varphi)$.
	\end{lemma}
	\begin{proof}
		We prove the case for $L\approx \cH_i$ in detail, and the other cases are similar.
		Let $\{v_1,\ldots,v_{2k+m}\}$ be a basis of $M^{[k]}$ whose Gram matrix is $\cH^k\obot\mathrm{diag}\{1,\ldots,1,\nu\}$ where $\nu$ is a unit. Set $\varphi(L)=\spaF\{w_1,w_2\}$.
		
		\noindent
		Assume $d(\varphi)=2$. If $i=-1$, then there is nothing to prove. Therefore, we may assume $i>-1$. By Lemma \ref{lem:primitive in H, vector}, without loss of generality, we can assume that $ w_1=v_1$. Then
		\[w_2=a_1v_1+\pi^{i+1} v_2+\sum_{j=3}^{2k+m} a_j v_j.\]
		By the assumption that $d(\varphi)=2$, we know that
		\[\mathrm{min}\{\mathrm{v}_\pi(a_j)\mid 3\leq j \leq 2k\}=0.\]
		Hence applying Lemma \ref{lem:primitive in H, vector} to $\cH^{k-1}\obot M$, we can find a $g\in \mathrm{U}(M^{[k]})$ such that
		\[g w_1=v_1, \quad g w_2\in \cH^2\]
		where $\cH^2$ refers to the first direct summand in the decomposition $\cH^k\obot M=\cH^2\obot\cH^{k-2}\obot M$.

		When $d(\varphi)=1$, without loss of generality, we can assume $\Pr_{\cH^k}(w_1)$ is primitive. By Lemma \ref{lem:primitive in H, vector}, we can assume that $ w_1=v_1$. Then
		\[w_2=a_1v_1+\pi^{i+1} v_2+\sum_{j=3}^{2k+m} a_j v_j.\]
		By the assumption that $d(\varphi)=1$, we know that
		\[\mathrm{min}\{\mathrm{v}_\pi(a_j)\mid 3\leq j \leq 2k\}\geq 1.\]
		Since we assume $\varphi$ is primitive, we know that
		\[\mathrm{min}\{\mathrm{v}_\pi(a_j)\mid 2k+1\leq j \leq 2k+r\}=0.\]
		Then we are done by applying Lemma \ref{lem:vector primitive in S} to $\cH^{k-1} \obot M$.
		
		When $d(\varphi)=0$, without loss of generality, we may assume $w_1=v_{2k+1}+v_{2k+2}$ by Lemma \ref{lem:vector primitive in S}. Here, we pick $v_{2k+i}$ so that the corresponding Gram matrix is $\mathrm{Diag}(1,-1,1,\dots,-\nu)$ (this is possible since we assume $m\ge 4$). Since $\varphi$ is primitive with $d(\varphi)=0$, then
		\begin{align*}
			w_2=\sum_{i=1}^{2k}\pi a_iv_i+\sum_{i=2k+1}^{2k+m}a_iv_i,
		\end{align*}
		and
		\[\mathrm{min}\{\mathrm{v}_\pi(a_j)\mid 2k+3\leq j \leq 2k+r\}=0.\]
		We are done by applying Lemma \ref{lem:vector primitive in S} to $\cH^k\obot \spaF\{v_{2k+3},\dots,v_{2k+m}\}$.
	\end{proof}

	\subsection{Calculation of primitive local density}\label{subsec: primitive ld}
	In this subsection, we compute primitive local density polynomials and prove the formulas in  Propositions \ref{prop: ind for t} and \ref{prop: ind rank 2}. Assume $L$ is represented by a non-singular Hermitian matrix $T$ of rank $n\le 2$.
	We let $\bar{v}$  denote the image of $v$ in $M^{[k]}\otimes_{\Oo_F}\mathbb{F}_\q$. Let
	\begin{align*}
		& (M^{[k]})^n(i)\coloneqq  \{(v_j)\in M^{n,(n)}_{k}\mid \mathrm{Span}_{\mathbb{F}_\q}\{\mathrm{Pr}_{\cH^k}(\bar{v}_j),\ 1\le j\le n \} \text{ has rank $i$} \}
	\end{align*}
	where $M^{n,(n)}$ is as in \eqref{eq: def of M^{n,l}}, and
	\begin{align}\label{eq: beta integral version}
		&\beta_{i}(M,L,X)
		\coloneqq\int_{\Herm_{n}(F)} dY \int_{(M^{[k]})^n(i)}  \psi (\langle Y, T(\bx )-T \rangle)d\bx.
	\end{align}
	Notice that
	\begin{equation}
		\sum_{i=0}^{n}\beta_i(M, L,X)=\beta(M, L,X)^{(n)}
	\end{equation}
	is the primitive local density defined earlier, and we will shorten it as $\beta(M, L,X)$. Notice that if $L$ is of the form $\cH^{j}$, then $\beta(M, L,X)=\beta_{n}(M, L,X)$.

 First, by a variant of \cite{CY}, Chao Li and Yifeng Liu obtained the following formula of 	$\beta(\cH^k,L)$.
	\begin{lemma}\cite[Lemma $2.16$]{LL2}\label{lem:beta(H,L)}
		Let $b_1\le \cdots \le b_n$ be the unique integers such that $L^{\vee}/L\approx \Oo_F/(\pi^{b_1})\oplus \cdots \oplus \Oo_F/(\pi^{b_n})$.	Let $t_o(L)$ be the number of non-zero entries in $(b_1,\cdots,b_n)$.  Then
		\begin{align*}
			\beta(\cH^k,L)= \prod_{k-\frac{n+t_o(L)}{2}<i\le k} (1-\q^{-2i}).
		\end{align*}
	\end{lemma}
	\begin{lemma}\label{lem: cancel beta_2} Assume $L$ is of rank $n$, then
		\begin{align*}
			\beta_{n}(M,L,q^{-2k})=\beta(\cH^k,L).
		\end{align*}
  In particular, if $L$ is of the form $\cH^j$, then
  \begin{align*}
\beta(M,L,q^{-2k})=	\beta_{n}(M,L,q^{-2k})=\beta(\cH^k,L).
		\end{align*}
	\end{lemma}
	\begin{proof}
 Recall that $T(\bx)=(\bx, \bx)$ is the moment matrix of $\bx \in (M^{[k]})^n$. For a  $\bx_2 \in M^n$, let  $T'(\bx_2)=T-T(\bx_2)$.
		Then
		\begin{align*}
			\beta_n( M,L,q^{-2k})
			&=\int_{\Herm_{n}(F)} dY \int_{M^n}   \int_{(\cH^k)^{n,(n)}}\psi (\langle Y,T({}^t(\bx_1,\bx_2))-T \rangle)d\bx_{1} d\bx_{2} \\
			&= \int_{\Herm_{n}(F)} dY \int_{M^n}   \int_{(\cH^k)^{n,(n)}} \psi (\langle Y, T(\bx_1)+T(\bx_2)-T\rangle)d\bx_1 d\bx_2\\
   	&= \int_{\Herm_{n}(F)} dY \int_{M^n}   \int_{(\cH^k)^{n,(n)}} \psi (\langle Y, T(\bx_1)-T'(\bx_2)\rangle)d\bx_1 d\bx_2
		\end{align*}
Notice that if $L$ and $L'$ are two Hermitian $\Oo_F$-lattices with moment matrix $T$ and $T'$ such that $T-T'\in \Herm_n(\Oo_{F_0})$, then $t_o(L)=t_o(L')$. Hence, for any $\bx_2 \in M^n$, we have  by Lemma \ref{lem:beta(H,L)}
\begin{align*}
   \beta(\cH^k,T'(\bx_2))&=
    \int_{\Herm_{n}(F)} dY     \int_{(\cH^k)^{n,(n)}} \psi (\langle Y, T(\bx_1)-T'(\bx_2)\rangle)d\bx_1  \\ &=\int_{\Herm_{n}(F)} dY     \int_{(\cH^k)^{n,(n)}} \psi (\langle Y, T(\bx_1)-T\rangle)d\bx_1\\
    &=\beta(\cH^k,T).
\end{align*}
Therefore
		\begin{align*}
		\beta_n(M,L,q^{-2k})
  &=\mathrm{vol}(M^n,d\bx_2) \cdot \int_{\Herm_{n}(F)} dY    \int_{(\cH^k)^{n,(n)}} \psi (\langle Y, T(\bx_1)-T\rangle)d\bx_1 \\
			&=  \int_{\Herm_{n}(F)} dY    \int_{(\cH^k)^{n,(n)}} \psi (\langle Y, T(\bx_1)-T\rangle)d\bx_1 \\
			&= \beta(\cH^k,L).
		\end{align*}
		
	\end{proof}
	
	Combining the above two lemmas, we have the following.
	\begin{corollary}\label{cor: beta2} \hfill
 \begin{enumerate}
     \item
 If $L$ is of rank $1$, then we have
		\begin{align*}
			\beta_{1}(M, L,X)=1-X.
		\end{align*}
\item		If $L$ is of rank $2$, then we have
		\begin{align*}
			\beta_{2}(M,L,X)=\begin{cases}
				(1-X)& \text{ if $L=\cH$},\\
				(1-X)(1-\q^2X)& \text{ otherwise}.
			\end{cases}
		\end{align*}
		\end{enumerate}
	\end{corollary}

	\begin{lemma}\label{lem: cancel beta_0} For an  $\Oo_F$-Hermitian lattice, let $\bar L= L/\pi L$ be its reduction modulo $\pi$ with resulting quadratic form.
		Let $r(\overline{M},\overline{L})$ to be the number of isometries from $\overline{L}$ to $\overline{M}$. Then
		\begin{align*}
			\beta_{0}(M,L,X)=X^n\beta(M,L) =\q^{-mn+n^2}r(\overline{M},\overline{L}) X^n.
		\end{align*}
	\end{lemma}
	\begin{proof}
		The second identity follows from the same proof of  \cite[Theorem 3.12]{CY}.  Then a similar argument as in the proof of Lemma \ref{lem: cancel beta_2} gives the first identity. In this case, we need to replace $M^n\obot (\cH^{k})^{n,(n)}$ in the proof of Lemma \ref{lem: cancel beta_2} with $M^{n,(n)}\obot  (\pi \cH^{k})^n$. The factor $X^n$ shows up because $\mathrm{vol}((\pi\cH^k)^n)=(q^{-2k})^n$. We leave the details to the reader.
	\end{proof}

	Notice that  \cite[Lemma 3.2.1]{LZ2} provides a uniform formula for $|r(\overline{M},\overline{L})|$. As a result, we obtain the following corollaries.

	\begin{corollary}\label{cor: beta N2}
		Assume $L=\Oo_F \bx $ is of  rank   $1$ (we allow $q(\bx) =0$).
		\begin{enumerate}
			\item If $\mathrm{v}(L)=0$, then
			\begin{align*}\beta_0(M,L,X)=\begin{cases}
					(1+\chi(M) \chi(L)q^{-\frac{m-1}{2}})X & \text{ if $m$ is odd},\\
					(1-\chi(M)q^{-\frac{m}{2}})X  & \text{ if $m$ is even}.
			\end{cases}\end{align*}
			\item If $\mathrm{v}(L)>0$, then
			\begin{align*}\beta_0(M,L,X)=\begin{cases}
					(1-q^{1-m})X & \text{ if $m$ is odd},\\
					 (1-q^{1-m}+\chi(M)(q-1)q^{-\frac{m}{2}} )X & \text{ if $m$ is even}.
			\end{cases}\end{align*}
		\end{enumerate}
	\end{corollary}
	
	\begin{corollary}\label{cor:beta0} Assume $L$ is of rank $2$. When $t(L)=1$, we assume that $L$ has Gram matrix  $T=\mathrm{Diag}(u_1,u_2(-\pi_0)^b)$ with $b >0$.
		\begin{enumerate}
			\item If $m$ is odd, then
			\begin{align*}
				&\beta_0(M,L,X)=\begin{cases}
					\q(1-\q^{1-m})X^2& \text{ if $t(L)=0$},\\
					\q(1+\chi(M)\chi( u_1)\q^{\frac{3-m}{2}})(1-\q^{1-m})X^2& \text{ if $t(L)=1$},\\
					\q(1-\q^{1-m})(1-\q^{3-m})X^2 & \text{ if $t(L)=2$}.
				\end{cases}
			\end{align*}
			\item If $m$ is even, then
			\begin{align*}
				&\beta_0(M,L,X)=\begin{cases}
					\q  ( 1-\chi(L)\q^{1-m}  +\chi(L)\chi(M)(\q-\chi(L))\q^{-\frac{m}{2}}  ) X^2 & \text{ if $t(L)=0$},\\
					\q(1-\chi(M)\q^{-\frac{m}{2}})(1-\q^{2-m})X^2 & \text{ if $t(L)=1$},\\
					\q ( (1-\q^{2-m})+\chi(M)(\q^2-1)\q^{-\frac{m}{2}} )(1-\q^{2-m})X^2 & \text{ if $t(L)=2$}.
				\end{cases}
			\end{align*}
		\end{enumerate}
	\end{corollary}

	Finally, we calculate $\beta_1(M,L,X)$.
	\begin{proposition}\label{prop:beta1 new method} Assume $L$ is as in Corollary \ref{cor:beta0}. Let $\delta_e(m)=1$ or $0$ depending on whether $m$ is even or odd.
		\begin{enumerate}
			\item  If $t(L)=2$, then
			\begin{align*}
				\beta_1(M,L,X)=
				\q(\q+1) ( (1-\q^{1-m})+\delta_{e}(m)\chi(M)(\q-1)\q^{-\frac{m}{2}} )X(1-X).
			\end{align*}
			\item 		
			If $t(L)=1$, then
			\begin{align*}
				&\beta_1(M,L,X) =\begin{cases}
					\q(1+\q-\q^{1-m}+\chi(M)\chi( u_1)q^{\frac{3-m}{2}})X(1-X) & \text{ if $m$ is odd},\\
					\q(1+\q-\q^{1-m}-\chi(M)\q^{-\frac{m}{2}})X(1-X) & \text{ if $m$ is even}.
				\end{cases}
			\end{align*}
			\item   If $t(L)=0$ and $\chi(L)=1$, i.e. $L\cong\cH_0$, then
			\begin{align*}
				\beta_1(M,L,X)&=
				\q   ( \q+1-2\q^{1-m}+\delta_{e}(m)\chi(M)(\q-1)\q^{-\frac{m}{2}} )X(1-X).
			\end{align*}
			\item  If $t(L)=0$ and $\chi(L)=-1$, then
		\begin{align*}
				\beta_1(M,L,X)&=
				\q  (\q+1) ( 1-\delta_{e}(m)\chi(M)\q^{-\frac{m}{2}} )X(1-X).
				\end{align*}
		\end{enumerate}
	\end{proposition}
	\begin{proof}
	First we assume  $L=\cH_i$. We claim that
		\begin{align*}
			&\beta_1(M, \cH_i, X) =\begin{cases}
				\q(1-X)\left(2\beta_0(M,0,X)+\sum_{\alpha \in (\Oo_{F_0}/(\pi_0))^{\times}}\beta_0(M, \langle -2\alpha \rangle,X)  \right)& \text{ if $i=0$},   \\
				\q(\q+1) (1-X)\beta_0(M,0,X) & \text{ if $i\ge 1$}.
			\end{cases}
		\end{align*}
		Here $\alpha (M, 0, X)=\alpha(M, \Oo_F \bx, X)$ with $q(\bx) =0$ and $\bx \ne 0$.
		Assuming the claim, the proposition for $L=\cH_i$ follows from Corollary \ref{cor: beta N2}.

	To prove the claim, it suffices to show the identity for $X=q^{-2k}$ for sufficiently many $k\ge 0$.
		Recall
		\begin{align*}
			I(M^{[k]},L,d)=\{ \phi \in &  \mathrm{Hom}_{\Oo_F}(L/\pi_0^d L, M^{[k]}/ \pi_0^d M^{[k]}) \mid \\
			&(\phi(x),\phi(y))\equiv (x,y) \mod \pi^{2d-1},\  \forall x,y \in L\}.
		\end{align*}
		Let
		\begin{align*}
			J(M^{[k]},L,d)&\coloneqq \{\phi\in I(M^{[k]},L,d)\mid \mathrm{dim}_{\mathbb{F}_{\q}}\overline{\mathrm{Pr}_{\cH^k}(\phi(L))}=
			\mathrm{dim}_{\mathbb{F}_{\q}}\overline{\mathrm{Pr}_{M}(\phi(L))}=1\}.\end{align*}
		Then
		\begin{align*}
			\beta_1(M,L,q^{-2k})=\lim_{d\to \infty}\q^{-(4(2k+m)-4)d}|J(M^{[k]},L,d)|.
		\end{align*}  Let $\{l_1,l_2\}$ be a basis of $L$ with Gram matrix  $\cH_i$. For $\phi \in J(M^{[k]},L,d)$, it will be determined by $w_i=\phi(l_i)$. Let $w_{i,\cH}=\mathrm{Pr}_{\cH^k}(w_i)$, and  $w_{i,M}=\mathrm{Pr}_{M}(w_i)$. Since $\mathrm{rank}_{\mathbb{F}_{\q}}\overline{\mathrm{Pr}_{\cH^k}(\phi(L))}=1$,  $\mathrm{rank}_{\mathbb{F}_{\q}}\overline{\mathrm{Pr}_{\cH^k}(w_i)}=1$ for $i=1$ or $2$.
		
		Now we define a partition of $J(M^{[k]},L,d)$. Assume $\alpha \in  \Oo_{F_0}$.   Let
		\begin{align*}
			J_{\alpha}(M^{[k]},L,d)&\coloneqq \{\phi\in I(M^{[k]},L,d)\mid \mathrm{rank}_{\mathbb{F}_{\q}}\overline{w}_{1,\cH} =1,\
			\overline{w}_{2,\cH}=\alpha \overline{w}_{1,\cH}\}, \text{ and }\\
			J_{\infty}(M^{[k]},L,d)&\coloneqq \{\phi\in I(M^{[k]},L,d)\mid \mathrm{rank}_{\mathbb{F}_{\q}}\overline{w}_{2,\cH}=1,\ \overline{w}_{1,\cH}=0 \}.
		\end{align*}
		Then it is easy to verify
		\begin{align*}
			J(M^{[k]},L,d)=\bigcup_{\alpha\in \Oo_{F_0}/(\pi_0)} J_{\alpha}(M^{[k]},L,d)\cup J_{\infty}(M^{[k]},L,d).
		\end{align*}

		Now we compute $|J_{\alpha}(M^{[k]},L,d)|$. To determine a $\phi \in J_{\alpha}(M^{[k]},L,d)$, we choose $w_1=\phi(l_1)$ first. By definition, we have
		\begin{align}\label{eq: w_1, H_i}
			&	\lim_{d\to \infty}\q^{(2(2k+m)-1)d}\#\{w_1\in M^{[k]}/\pi_0^d M^{[k]} \mid \text{$w_{1,\mathcal{H}}$ is primitive, and } q(w_1)\equiv 0 \text{ mod } \pi_0^d  \}\\ \notag
			&=\beta_1(M^{[k]},0)=1-q^{-2k}.
		\end{align}
		
		Given such a $w_1$, now we find the number of $w_2=\phi(l_2)$ such that  $\phi$ lies in $J_{\alpha}(M^{[k]},L,d)$. By Lemma \ref{lem:primitive in H, vector}, we may assume $w_{1,S}=0$. Let $w_2=w_{2,M}+\alpha w_1+\pi w_\cH$, where $w_\cH\in \cH^k$. Then the corresponding $\phi$ lies in $J_{\alpha}(M^{[k]},L,d)$ if and only if
		\begin{align*}
			 \pi^i\equiv (w_1,w_2)\equiv(w_1,\pi w_\cH)   \mod \pi^{2d-1}  \end{align*}
		and
		\begin{align*}0\equiv q(w_2)&\equiv\mathrm{tr}((\alpha w_1,\pi w_\cH))-\pi_0 q(w_\cH)+q(w_{2,M}) \\ &\equiv\alpha\mathrm{tr}(\pi^i)-\pi_0 q(w_\cH)+q(w_{2,M}) \mod \pi^{2d-1} .
		\end{align*}

		First,
		\begin{align}\label{eq: pi w_H, H_i}
			\lim_{d\to \infty}\q^{-2d(2k-1)}\#\{\pi w_\cH \in \cH^k/\pi_0^d  \cH^k\mid(w_1,\pi w_{\cH})\equiv\pi^i\mod \pi^{2d-1}\}=\q^{1-2k}.
		\end{align}
		Second,  for each fixed $\pi w_\cH$  we have
		\begin{align}\label{eq: w_{2,M}, H_i}
			&\lim_{d\to \infty}\q^{(-2m+1)d} \#\{ w_{2,M} \in M/\pi_0^d  M\mid  w_{2,M} \text{ primitive, }  q(w_{2,M})\equiv -\alpha \mathrm{tr}(\pi^i)+\pi_0 q(w_\cH) \mod \pi^{2d-1}\}
			\\
			&=\beta(M,\langle -\alpha \mathrm{tr}(\pi^i)+\pi_0 q(w_\cH)\rangle)\notag
			\\
			&=\begin{cases}
				\beta(M, \langle -2 \alpha \rangle) &\hbox{if } i=0,
				\\
				\beta(M, 0)  &\hbox{if } i > 0.
			\end{cases}  \notag
		\end{align}
		By symmetry, $| J_{\infty}(M^{[k]},L,d)|=| J_{0}(M^{[k]},L,d)|$.
		Now a combination of \eqref{eq: w_1, H_i}, \eqref{eq: pi w_H, H_i} and \eqref{eq: w_{2,M}, H_i} implies that
		\begin{align*}
			\beta_1(M, \cH_i, q^{-2k})
			&=\lim_{d\to \infty}\q^{(-4(2k+m)+4)d}  \left(\sum_{\alpha\in \Oo_{F_0}/(\pi_0)}| J_{\alpha}(M^{[k]},L,d)|+| J_{\infty}(M^{[k]},L,d)|\right)
			\\
			&=\begin{cases}
				\q(1-X)\left(2\beta_0(M,0,q^{-2k})+\sum_{\alpha \in \Oo_{F_0}^{\times}/(\pi_0)}\beta_0(M,-2\alpha,q^{-2k})  \right) & \text{ if $i=0$},   \\
				\q(\q+1) (1-X)\beta_0(M,0,q^{-2k}) & \text{ if $i\ge 1$},
			\end{cases}
		\end{align*}
		as claimed.		
		
		Next, we assume $L$ has a basis $\{ l_1,  l_2\}$ whose Gram matrix is
		$\diag(u_1(-\pi_0)^a,u_2(-\pi_0)^b)$ with $0 \le a \le b$. Let  $w_i=\phi(l_i)$ as before. Then  the number of possible choices for $w_1$ is given by $$\q^{(2(2k+m)-1)d}\beta_1(M,\langle u_1(-\pi_0)^a \rangle , q^{-2k})$$ for sufficiently large $d$. We may assume $w_1=w_{1,\cH}$ without loss of generality. Let $w_2=w_{2,M}+\alpha w_1+\pi w_\cH$ as before. Then  $\phi$ lies in $J_{\alpha}(M^{[k]},L,d)$ if and only if
		\begin{align*}
		0	\equiv (w_1,w_2)\equiv(w_1,\alpha w_1)+(w_1,\pi w_\cH) \mod \pi^{2d-1}\end{align*}
		and
		\begin{align*}u_2(-\pi_0)^b&\equiv q(w_2)\equiv(w_{2,M}+\alpha w_{1}+\pi w_\cH , w_2)\\
			&\equiv q(w_{2,M})-\alpha^2 q( w_1)-\pi_0 q(w_\cH)\mod \pi^{2d-1}.
		\end{align*}	
		Now
		$$\lim_{d\to \infty}\q^{(-4k+2)d}\#\{\pi w_\cH \in \cH^k/\pi_0^d  \cH^k\mid (w_1,\pi w_{\cH})\equiv -(w_1,\alpha w_1)  \mod \pi^{2d-1}\}=\q^{1-2k},$$
		and for a fixed $\pi w_\cH$ we have
		\begin{align*}
			&\lim_{d\to \infty}\q^{(-2m+1)d} \#\{ w_{2,M} \in L_{S}/\pi_0^d  L_{S}\mid w_{2,M} \text{ primitive, }\\
			&\quad \quad \quad \quad \quad 	\quad \quad \quad q(w_{2,M})\equiv u_2(-\pi_0)^b+\alpha^2 q( w_1)+\pi_0 q(w_\cH)\mod \pi^{2d-1}\}\\
			&=\beta(M, \langle u_2(-\pi_0)^b+\alpha^2 q( w_1)+\pi_0 q(w_\cH)\rangle).
		\end{align*}
		Now this proposition follows from a similar argument as before, and we leave the details to the reader.
	\end{proof}

	\bibliographystyle{alpha}
	\bibliography{reference}
\end{document}